\documentclass[a4paper,reqno]{amsart}

\usepackage{geometry}
\usepackage{dsfont}
\usepackage{amssymb}
\usepackage{stmaryrd}

\usepackage{graphicx}
\usepackage{subfigure}

\title[A convergent scheme for a fourth order equation]{A convergent Lagrangian discretization for a nonlinear fourth order equation}

\author{Daniel Matthes}
\address{Daniel Matthes \\ Zentrum Mathematik \\ TU M\"unchen \\ Boltzmannstr. 3 \\ D-85748 Garching \\ Germany}
\email{matthes@ma.tum.de}

\author{Horst Osberger}
\address{Horst Osberger \\ Zentrum Mathematik \\ TU M\"unchen \\ Boltzmannstr. 3 \\ D-85748 Garching \\ Germany}
\email{osberger@ma.tum.de}

\thanks{This research was supported by the DFG Collaborative Research Center TRR 109, ``Discretization in Geometry and Dynamics''.}

\newcommand{\setR}{\mathbb{R}}
\newcommand{\setRnn}{\mathbb{R}_{\ge0}}
\newcommand{\setN}{\mathbb{N}}
\newcommand{\dd}{\,\mathrm{d}}
\newcommand{\dn}{\mathrm{d}}
\newcommand{\indy}{\mathbf{1}}
\newcommand{\eins}{\mathds{1}}

\newcommand{\theX}{\mathrm{X}}

\newcommand{\xvec}{\vec{\mathrm x}}
\newcommand{\yvec}{\vec{\mathrm y}}
\newcommand{\zvec}{\vec{\mathrm z}}
\newcommand{\vvec}{\vec{\mathrm v}}
\newcommand{\wvec}{\vec{\mathrm w}}
\newcommand{\spr}[2]{\left\langle#1,#2\right\rangle_\delta}
\newcommand{\nrm}[1]{\left\|#1\right\|_\delta}

\newcommand{\cf}{\mathbf{u}}

\newcommand{\cz}{\mathbf{z}}
\newcommand{\cX}{\mathbf{X}}

\newcommand{\theh}{\delta}%{{\boldsymbol \xi}}
\newcommand{\vech}{{\vec\delta}}

\newcommand{\xspc}{\mathfrak{X}}
\newcommand{\xseq}{\mathfrak{x}}
\newcommand{\xseqN}{\xseq_\theh}
\newcommand{\xseqNN}{\xseq_\vech}

\newcommand{\dens}{\mathcal{P}(\Omega)}
\newcommand{\densN}{\mathcal{P}_\theh(\Omega)}
\newcommand{\densNN}{\mathcal{P}_\vech(\Omega)}

\newcommand{\tv}[1]{{\mathrm{TV}}\left[#1\right]}
\newcommand{\ti}[1]{\left\lbrace#1\right\rbrace_{\tau}}%{\overline{#1}}%{\left\langle#1\right\rangle_\tau}
\newcommand{\Wmat}{\mathrm{W}}

\newcommand{\ival}{{\mathbb{I}_K^0}}
\newcommand{\ivalp}{{\mathbb{I}_K^+}}
\newcommand{\hval}{{\mathbb{I}_K^{1/2}}}

\newcommand{\kmh}{{k-\frac{1}{2}}}
\newcommand{\kph}{{k+\frac{1}{2}}}
\newcommand{\kpd}{{k+\frac{3}{2}}}
\newcommand{\kmd}{{k-\frac{3}{2}}}

\newcommand{\kappm}{{\kappa-\frac12}}
\newcommand{\kappp}{{\kappa+\frac12}}

\newcommand{\Kmh}{{K-\frac{1}{2}}}
\newcommand{\imh}{{\frac{1}{2}}}

\newcommand{\thegrad}{\operatorname{grad}_{\mathbb W}}
\newcommand{\eps}{{\varepsilon}}
\newcommand{\wass}{\mathbb{W}}%{\mathbf{W}_2}
\newcommand{\wassN}{\mathbb{W}_\delta}
\newcommand{\grad}{\partial_{\xvec}}
\newcommand{\wgrad}{\nabla_\delta}

\newcommand{\F}{\mathcal{F}} % Fisher-Info für u
\newcommand{\Fz}{\mathbf{F}_\delta} % Fisher-Info für zvec
\newcommand{\Fzo}{\mathbf{F}} % Fisher-Info für zvec
\newcommand{\Fy}{\mathbf{F}_\Delta} % Fisher-Info für zvec

\newcommand{\HF}{\mathcal{H}}

\newcommand{\HFz}{\mathbf{H}_\delta}
\newcommand{\olH}{\overline{\HF}}

\newcommand{\hatf}{\theta}

\newcommand{\baru}{\bar{u}}
\newcommand{\hatz}{\widehat{z}}
\newcommand{\hatu}{\widehat{u}}

\newcommand{\ee}{\mathbf{e}}
\newcommand{\intom}{\int_\Omega}
\newcommand{\err}{\mathrm{e}}

\newcommand{\tlo}{{\underline t}}
\newcommand{\thi}{{\overline t}}
\newcommand{\tint}{[\underline t,\overline t]}

\newcommand{\loc}{\text{loc}}

\newtheorem{thm}{Theorem}
\newtheorem{prp}[thm]{Proposition}
\newtheorem{lem}[thm]{Lemma}

\newtheorem{rmk}[thm]{Remark}

\newtheorem{xmp}[thm]{Example}

\begin{document}

\begin{abstract}
  A fully discrete Lagrangian scheme for numerical solution of 
  the nonlinear fourth order DLSS equation in one space dimension is analyzed.
  The discretization is based on the equation's gradient flow structure in the $L^2$-Wasserstein metric. 
  We prove that the discrete solutions are strictly positive and mass conserving.
  Further, they dissipate \emph{both} the Fisher information and the logarithmic entropy.
  Numerical experiments illustrate the practicability of the scheme.

  Our main result is a proof of convergence of fully discrete to weak solutions
  in the limit of vanishing mesh size.
  Convergence is obtained for arbitrary non-negative initial data with finite entropy,
  without any CFL type condition.
  The key ingredient in the proof is a discretized version of the classical entropy dissipation estimate.
\end{abstract}

\maketitle

%%%%%%%%%%%%%%%%%%%%%%%%%%%%%%%%%%%%%%%%%%%%%%%%%%%%%%%%%%%%%%%%%%%%%%%%%%%%%
%%%%%%%%%%%%%%%%%%%%%%%%%%%%%%%%%%%%%%%%%%%%%%%%%%%%%%%%%%%%%%%%%%%%%%%%%%%%%

\section{Introduction}
\subsection{The equation and its properties}
In this paper, we study a full discretization of the following initial boundary value problem 
on the one-dimensional interval $\Omega=[a,b]$:
\begin{align}
  \label{eq:dlss}
  \partial_tu+\partial_x\left(u\,\partial_x\left(\frac{\partial_{xx}\sqrt{u}}{\sqrt{u}}\right)\right) = 0& \quad \text{for $t>0$ and $x\in\Omega$}, \\
  \label{eq:bc}
  \partial_xu=0,\quad u\,\partial_x\left(\frac{\partial_{xx}\sqrt{u}}{\sqrt{u}}\right)=0& \quad\text{for $t>0$ and $x\in\partial\Omega$}, \\
  \label{eq:ic}
  u=u^0&\quad\text{at $t=0$}.
\end{align}
Equation \eqref{eq:dlss} is known as the \emph{DLSS} equation, 
where the acronym refers to Derrida, Lebowitz, Speer and Spohn, 
who introduced \eqref{eq:dlss} in \cite{DLSS1,DLSS2} for studying interface fluctations in the anchored Toom model.
In the context of semi-conductor physics, 
\eqref{eq:dlss} appears as a simplified \emph{quantum drift diffusion equation} \cite{Degond,Juengel}.

The analytical treatment of \eqref{eq:dlss} is far from trivial:
see e.g. \cite{BLS,Funique,GJT,GST,JMdlss,JPdlss} for results on existence and uniqueness of solutions in various different settings,
and \cite{CCTdlss,CTthin,CDGJ,GST,JMdlss,JTdecay,MMS} for qualitative and quantitative descriptions of the long-time behavior.
The main difficulty in the development of the time-global well-posedness theory has been 
that the nonlinear operator in \eqref{eq:dlss} is defined only for positive functions $u$,
but there is no maximum principle available 
which would provide an a priori positive lower bound on $u$.
Ironically, solutions are known to be $C^\infty$-smooth as long as they remain strictly positive \cite{BLS},
but the question if strict positivity of the initial datum $u^0$ is sufficient for that remains open,
despite much effort and some recent progress in that direction, see \cite{Fpositive}.
In order to deal with the general case --- allowing arbitrary non-negative initial data $u^0$ of finite entropy ---
a theory for non-negative weak solutions has been developed \cite{GST,JMdlss}
on grounds of the a priori regularity estimate
\begin{align}
  \label{int:apriori}
  \sqrt{u}\in L^2_\loc\big(\setR_+;H^2(\Omega)\big),
\end{align}
which gives a meaning to \eqref{eq:dlss} in the formally equivalent representation
\begin{align}
  \label{int:preweak}
  \partial_t u + \partial_{xxxx}u - \partial_{xx}\left(\partial_x\sqrt{u}\right)^2 = 0.
\end{align}

On the other hand, the problem \eqref{eq:dlss}--\eqref{eq:ic} has several remarkable structural properties,
and these eventually paved the way to a rigorous analysis.
We list some of those properties:
\begin{itemize}
\item The evolution is mass preserving,
  \begin{align*}
    \int_a^b u(t;x)\dd x = M := \int_a^b u^0(x)\dd x \quad \text{for all $t>0$}.
  \end{align*}
\item There are infinitely many (formal) Lyapunov functionals \cite{BLS,CCTdlss,JMalg}.
  The two most important ones are 
  the (logarithmic) \emph{entropy},
  \begin{align}
    \label{eq:intro.entropy}
    \HF(u) = \int_a^bu\ln u\dd x - \HF_0 \quad \text{with} \quad \HF_0 = M\ln\Big(\frac{M}{b-a}\Big),
  \end{align}
  and the \emph{Fisher information},
  \begin{align}
    \label{eq:intro.fisher}
    \F(u) = \int_a^b \big(\partial_x\sqrt{u}\big)^2\dd x.
  \end{align}
\item The Fisher information is more than just a Lyapunov functional:
  in \cite{GST}, it has been shown that \eqref{eq:dlss}\&\eqref{eq:bc} is a gradient flow in the potential landscape of $\F$ 
  with respect to the $L^2$-Wasserstein metric $\wass$. 
  That is, formally one can write \eqref{eq:dlss}\&\eqref{eq:bc} as
  \begin{align}
    \label{eq:gradflow}
    \partial_t u = -\thegrad\F(u).
  \end{align}
  % We refer the reader to \cite{VilBook} and to \cite{AGS} for an introduction to mass transportation metrics and to metric gradient flows, respectively.
\item Also $\HF$ is not ``an arbitrary'' Lyapunov functional:
  the $L^2$-Wasserstein gradient flow of $\HF$ is the \emph{heat equation} \cite{JKO},
  \begin{align}
    \label{eq:heat}
    \partial_s v = -\thegrad\HF(v) = \partial_{xx}v,
  \end{align}
  and the
  Fisher information $\F$ equals the dissipation of $\HF$ along its own gradient flow,
  \begin{align}
    \label{eq:magic}
    \F(v(s)) = -\frac{\dn}{\dd s}\HF(v(s)).
  \end{align}
  In view of \eqref{eq:gradflow}, 
  this relation makes the DLSS equation the ``big brother'' of the heat equation, 
  see \cite{DMfourth,MMS} for structural consequences.
\end{itemize}

\subsection{Fully discrete approximation}
For the numerical approximation of solutions to \eqref{eq:dlss}--\eqref{eq:ic}, 
it is natural to ask for structure preserving discretizations that inherit at least some of the nice properties listed above.
At the very least, the scheme should produce non-negative (preferably positive) discrete solutions,
but there is no reason to expect that behavior from a standard discretization approach.
Several (semi-)discretizations for \eqref{eq:dlss}--\eqref{eq:ic} that guarantee positivity
have been proposed in the literature \cite{BEJnum,CJTnum,JPnum,JuVi}.
In all of them, positivity actually appears as a consequence of another, more fundamental feature:
each of these schemes also inherits a Lyapunov functional, 
either a logarithmic/power-type entropy \cite{BEJnum,CJTnum,JPnum},
or a variant of the Fisher information \cite{BEJnum,DMMnum,JuVi}.
An exception is the discretization from \cite{DMMnum},
which preserves the Lagrangian representation of \eqref{eq:dlss}, see below, and thus enforces positivity by construction.
Apparently, at least \emph{some} structure preservation seems necessary to obtain an acceptable numerical scheme.

Here we follow further the ansatz from \cite{DMMnum}, which lead to a discretization with a very rich structure:
the scheme is positivity and mass preserving, it dissipates the Fisher information,
it has the same Lagrangian structure as \eqref{eq:dlss}, 
and it even inherits (in a certain sense) the gradient flow structure \eqref{eq:gradflow}.
By a small change of that discretization, we obtain a new scheme which still has all of these properties,
but in addition also dissipates the logarithmic entropy.
Thus, we have \emph{two} discrete Lyapunov functionals at our disposal,
and the interplay between these allows us to give a proof of convergence in the limit of vanishing mesh size.

\emph{We emphasize that our scheme is the first one to preserve more than one Lyapunov functional for \eqref{eq:dlss},
  and it is the only fully discrete scheme for which a rigorous convergence analysis is available.}

Below, we give the ``pragmatic'' definition of our full discretization, which is actually very simple.
In Section \ref{sec:discretization}, we show how this scheme arises 
from a structure preserving discretization of the gradient flow structure.
The starting point is the \emph{Lagrangian representation} of \eqref{eq:dlss}\&\eqref{eq:bc}.
Since each $u(t;\cdot)$ is of mass $M$, there is a Lagrangian map $\theX(t;\cdot):[0,M]\to\Omega$
--- the so-called \emph{pseudo-inverse distribution function} of $u(t;\cdot)$ ---
such that
\begin{align}
  \label{int:pseudo}
  \xi = \int_0^{\theX(t;\xi)}u(t;x)\dd x, \quad \text{for each $\xi\in[0,M]$}.
\end{align}
Written in terms of $\theX$, the Wasserstein gradient flow \eqref{eq:gradflow} for $\F$ 
turns into an $L^2$-gradient flow for
\begin{align*}
  \Fzo(\theX) = \int_0^M \left[\partial_\xi\left(\frac1{\partial_\xi\theX}\right)\right]^2\dd\xi,
\end{align*}
that is,
\begin{align}
  \label{eq:zeq}
  \partial_t\theX = \partial_\xi\big(Z^2\partial_{\xi\xi}Z\big), 
  \quad \text{where} \quad Z(t;\xi):=\frac1{\partial_\xi\theX(t;\xi)}=u\big(t;\theX(t;\xi)\big). 
\end{align}
At this point, a standard discretization of \eqref{eq:zeq} with parameter $\Delta=(\tau;\delta)$ is performed:
we use the implicit Euler method for time discretization with fixed time step $\tau>0$, 
and central finite differences for equidistant discretization on the mass space $[0,M]$ with mesh width $\delta>0$.
More explicitly: denote by $\xvec_\Delta=(x^n_k)$ a fully discrete solution on the $\Delta$-mesh,
so that $x^n_k$ approximates $\theX(n\tau;k\delta)$, 
then the $x^n_k$ satisfy
\begin{align}
  \label{eq:dgf}
  \frac{x^n_k-x^{n-1}_k}\tau 
  = \frac1\delta\left[
    (z^n_\kph)^2\left(\frac{z^n_{\kpd}-2z^n_\kph+z^n_\kmh}{\delta^2}\right)
    -(z^n_\kmh)^2\left(\frac{z^n_{\kph}-2z^n_\kmh+z^n_\kmd}{\delta^2}\right)
  \right],
\end{align}
where the values $z^n_{\ell-\frac12} = \delta/(x^n_\ell-x^n_{\ell-1})$ are associated to the mid-points of the spatial grid.
At each time step $n\in\setN$, $\xvec_\Delta^n=(x_1^n,\ldots,x_{K-1}^n)$ approximates a Lagrangian map, 
so we assume that $\xvec_\Delta^n$ is monotone, i.e., $x^n_k>x^n_{k-1}$,
and in accordance with \eqref{int:pseudo}, 
we associate to $\xvec_\Delta^n$ a piecewise constant function $\baru_\Delta^n:\Omega\to\setR_+$ 
with
\begin{align*}
  \baru_\Delta^n (x) = \frac\delta{x^n_k-x^n_{k-1}} \quad \text{for $x^n_{k-1}<x<x^n_k$}.
\end{align*}
As replacements for the entropy $\HF$ and the Fisher information $\F$,
we introduce
\begin{align*}
  \HFz(\xvec^n) = \delta\sum_{k=1}^{K}\log z^n_\kmh,
  \qquad
  \Fz(\xvec^n) = \sum_{k=1}^{K-1}\left(\frac{z^n_\kph-z^n_\kmh}\delta\right)^2.
\end{align*}
These choices are made such that $\HFz$ is the restriction of $\HF$, i.e., $\HFz(\xvec^n_\Delta)=\HF(\baru_\Delta^n)$,
and such that $\Fz$ is related to $\HFz$ in the same way \eqref{eq:magic} as $\F$ is related to $\HF$;
see Section \ref{sec:discretization} for datails.

\subsection{Results}
Our first result is concerned with qualitative properties of the discrete solutions.
For the moment, fix a discretization parameter $\Delta=(\tau;\delta)$.
\begin{thm}
  \label{thm:pre}
 From any monotone discrete initial datum $\xvec_\Delta^0$, 
  a sequence of monotone $\xvec_\Delta^n$ satisfying \eqref{eq:dgf} can be constructed 
  by inductively defining $\xvec_\Delta^n$ as a global minimizer of
  \begin{align*}
    \xvec\mapsto\frac\delta{2\tau}\sum_k \big(x_k-x_k^{n-1})^2 + \Fz(\xvec).
  \end{align*}
  This sequence of vectors $\xvec_\Delta^n$ and the associated densities $\baru_\Delta^n$
  have the following properties:
  \begin{itemize}
  \item \emph{Positivity:} $\baru_\Delta^n$ is a strictly positive function.
  \item \emph{Mass conservation:} $\baru_\Delta^n$ has mass equal to $M$.
  \item \emph{Dissipation:} Both the entropy and the discrete Fisher information are dissipated,
    \begin{align*}
      \HFz(\xvec_\Delta^n)\le\HFz(\xvec_\Delta^{n-1}) \quad\text{and}\quad \Fz(\xvec_\Delta^n)\le\Fz(\xvec_\Delta^{n-1}).
    \end{align*}
  \item \emph{Equilibration:} There is a constant $r>0$ only depending on $b-a$ such that
    \begin{align}
      \label{eq:equilibrate}
      \HFz(\xvec_\Delta^n) \le \HFz(\xvec_\Delta^0)e^{-r n\tau}.
    \end{align}
    % This implies that $\baru_\Delta^n$ converges to a constant function in $L^1(\Omega)$ with exponential rate.
  \end{itemize}
\end{thm}
Some of these properties follow immediately from the construction,
while others (like the equilibration) are difficult to prove.
Note that even well-posedness (which involves existence of a monotone minimizer for the functional)
is a non-trivial claim.

To state our main result about convergence, 
we need to introduce the time-interpolation $\ti{\baru_\Delta}:\setR_+\times\Omega\to\setR_+$,
which is given by
\begin{align*}
  \ti{\baru_\Delta}(t;x) = \baru_\Delta^n(x) \quad \text{for $(n-1)\tau<t\le n\tau$}.
\end{align*}
Further, $\Delta$ symbolizes a whole sequence of mesh parameters from now on,
and we write $\Delta\to0$ to indicate that $\tau\to0$ and $\delta\to0$ simultaneously.
%
% Finally, we stress that ``locally'' on a set $Z$ is always meant in the strict sense, i.e., on all compact subsets of $Z$.
% Since $\Omega=[a,b]$ is compact, local and global are equal on $\Omega$.
% However, $L^2_\loc(\setR_+)$ is not equal to $L^1_\loc(\setRnn)$, 
% since $t\mapsto t^{-1/2}$ belongs to the first space, but not to the second.
%
\begin{thm}
  \label{thm:main}
  Let a non-negative initial condition $u^0$ with $\HF(u^0)<\infty$ be given.
  Choose initial conditions $\xvec_\Delta^0$ such that $\baru^0_\Delta$ converges to $u^0$ weakly as $\Delta\to0$,
  and 
  \begin{align}
    \label{eq:genhypo}
    \olH:=\sup_\Delta\HFz(\xvec^0_\Delta)<\infty \quad\text{and}\quad
    \lim_{\Delta\to0}(\tau+\delta)\Fz(\xvec^0_\Delta)=0.
  \end{align}
  For each $\Delta$, construct a discrete approximation $\xvec_\Delta$ according to the procedure described in Theorem \ref{thm:pre} above.
  Then, there are a subsequence with $\Delta\to0$ and a limit function $u_*\in C(\setR_+\times\Omega)$ such that:
  \begin{itemize}
  \item $\ti{\baru_\Delta}$ converges to $u_*$ locally uniformly on $\setR_+\times\Omega$,
  \item $\sqrt{u_*}\in L^2_\loc(\setRnn;H^1(\Omega))$, 
  \item there are non-increasing functions $f,h:\setR_+\to\setR$ such that
    $\F(u_*(t))=f(t)$ and $\HF(u_*(t))=h(t)$ for a.e.\ $t>0$,
    and additionally $h(t)\le\olH e^{-rt}$ with the constant $r>0$ from \eqref{eq:equilibrate},
  \item $u_*$ satisfies the following weak formulation of \eqref{eq:dlss}\&\eqref{eq:bc}, see \eqref{int:preweak}:
    \begin{align}
      \label{eq:weakdlss}
      \int_0^\infty\intom\big[\partial_t\varphi\,u_* + \partial_{xxx}\varphi\,\partial_xu_* + 4\partial_{xx}\varphi\,\big(\partial_x\sqrt{u_*}\big)^2 \big]\dd x\dd t 
      + \intom \varphi(0,x)u^0(x)\dd x= 0
    \end{align}
    for every test function $\varphi\in C^\infty_c(\setRnn\times\Omega)$ satisfying $\partial_x\varphi(t;a)=\partial_x\varphi(t;b)=0$.
  \end{itemize}
\end{thm}
\begin{rmk}
  \begin{enumerate}
  \item \emph{Quality of convergence:}
    Since $\ti{\baru_\Delta}$ is piecewise constant in space and time,
    uniform convergence is obviously the best kind of convergence that can be achieved.
    % The estimate $\sqrt{u}\in L^2_\loc(\setRnn;H^2(\Omega))$ from \cite{GST,JMdlss} seems out of reach.
    % In the proof, we make ad hoc definitions $\hatu_\Delta^n$ of the $u_\Delta^n$ that are piecewise affine,
    % and are such that the $\ti{\hatu_\Delta}$ converge to $u$ in $L^2_\loc(\setRnn;H^1(\Omega))$.
  \item \emph{Rate of convergence:}
    The scheme \eqref{eq:dgf} is formally consistent of order $\tau+\delta^2$, see Proposition \ref{prp:consist}, 
    and this is also the observed rate of convergence in numerical experiments
    with smooth initial data $u^0$, see Section \ref{sct:experiments}.
    % The weak consistency that we can prove rigorously is only of order $\tau+\delta^{1/4}$, 
    % see \eqref{eq:weakform1}\&\eqref{eq:weakform2}.
  \item \emph{Initial condition:}   
    We emphasize that our only hypothesis on $u^0$ is $\HF(u^0)<\infty$, 
    which allows the same general initial conditions as in \cite{GST,JMdlss}.
    If $\F(u^0)$ happens to be finite, and also $\sup_\Delta\Fz(\xvec_\Delta^0)<\infty$,
    then the uniform convergence of $\ti{\baru_\Delta}$ holds up to $t=0$.
    % \item \emph{Lyapunov functionals:}
    %   It can be shown that, for almost every $t>0$, both $\HF(u_*(t))$ and $\F(u_*(t))$ are equal
    %   to monotonically decaying functions.
  \item \emph{Long time behavior:}
    By means of the Csiszar-Kullback inequality, the exponential decay of $\HF(u_*(t))$ to zero implies 
    exponential convergence of $u_*$ to the constant function $u_\infty\equiv M/(b-a)$ in $L^1(\Omega)$.
  \item \emph{No uniqueness:}
    Since our notion of solution is too weak to apply the uniqueness result from \cite{Funique},
    we cannot exclude that different subsequences of $\ti{u_\Delta}$ converge to different limits.
  \end{enumerate}
\end{rmk}
The idea to derive numerical discretizations for solution of Wasserstein gradient flows from the Lagrangian representation
is not new in the literature, see e.g.\ \cite{Kinderlehrer} for a general treatise.
Several practical schemes have been developed on grounds 
of the Lagrangian representation for this class of evolution problems,
mainly for second-order diffusion equations \cite{Budd,BCW,MacCamy,Russo},
but also for chemotaxis systems \cite{BCC},
for non-local aggregation equations \cite{CarM,Mary}, % if existent, cite {CarWol},
and for variants of the Boltzmann equation \cite{GosT2}.
For certain nonlinear fourth order equations, Lagrangian numerical schemes have been developed as well,
e.g., for the Hele-Shaw flow \cite{Naldi} and for a class of thin film equations \cite{GosT2}.
On the other hand, a rigorous analysis of stability and convergence of the \emph{fully discrete} schemes is rare 
and apparently limited to the case of nonlinear diffusion in one space dimension, see \cite{GosT,dde}.
There are, however, results available for \emph{semi-discrete} Lagrangian approximations,
see e.g. \cite{ALS,Evans}.

The primary challenge in our convergence analysis is to carry out all estimates
under \emph{no additional assumptions on the regularity} of the limit solution $u_*$.
In particular, we do not exclude a priori the formation of zeros --- and the induced loss of regularity --- in the limit $u_*$,
since this cannot be excluded by the existing theory.
Also, we allow extremely general initial conditions $u^0$.
Without sufficient a priori smoothness, 
we cannot simply use Taylor approximations and the like to estimate the difference between $\ti{\baru_\Delta}$ and $u_*$.
Instead, we are forced to derive new a priori estimates directly from the scheme, using our two Lyapunov functionals.

On the technical level, the main difficulty is that our scheme is \emph{fully discrete}, 
which means that we are working with spatial difference quotients instead of derivatives.
Lacking a discrete chain rule, 
the derivation of the relevant estimates turns out to be much harder than for the original problem \eqref{eq:dlss}--\eqref{eq:ic}.
For instance, we are able to prove a compactness estimate for $\baru_\Delta$, but not for its inverse distribution function,
although both estimates would be equivalent in a smooth setting.
This forces us to switch back and forth between 
the original \eqref{eq:dlss} and the Lagrangian \eqref{eq:zeq} formulation of the DLSS equation. 

% For comparison: under certain ad hoc hypotheses on $u_*$'s regularity in time 
% (which are justified for strictly positive, but not for general weak solutions),
% a quantitative convergence result for a time-discrete/spatially continuous approximation scheme 
% could be shown in \cite{BEJnum} in just a couple of lines, 
% using Taylor expansion and the monotonicity of the nonlinear operator in \eqref{eq:dlss}.

We further remark that the convergence of a family of gradient flows to a limiting gradient flow
has been thoroughly investigated on a very abstract level, see e.g. in \cite{AGS,Serfaty},
using methods of $\Gamma$-convergence.
Unfortunately, these appealing abstract results would not help to simplify our proof significantly,
since the verification of their main hypothesis ($\Gamma$-convergence of the subdifferentials)
is essentially equivalent to the derivation of the a priori estimates, which is the main part of our work.
Therefore, we decided to give a ``hands-on proof'', 
which requires only very few elements from the general theory of metric gradient flows.

\subsection{Structure of the paper}
We start with a description of our Lagrangian discretization in Section \ref{sec:discretization};
the fully discrete scheme is defined in Subsection \ref{sec:MM}.
In Section \ref{sec:apriori}, we derive various a priori estimates on the fully discrete solutions.
This leads to the main convergence results in Propositions \ref{prp:convergence1} and \eqref{prp:convergence2},
showing the existence of a limit function $u_*$ for $\Delta\to0$.
In Section \ref{sec:weak}, it is verified that $u_*$ is indeed a weak solution to \eqref{eq:dlss}--\eqref{eq:ic}.
The formal conclusion of the proofs for Theorems \ref{thm:pre} and \ref{thm:main} is contained in the short Section \ref{sec:proofs}.
Finally, Section \ref{sec:numeric} provides a consistency analysis and results from numerical simulations of \eqref{eq:dgf}.

\subsection*{Acknowledgement}
The authors are indebted to Giuseppe Savar\'{e} for fruitful discussions on the subject,
and especially for contributing the initial idea for the entropy preserving discretization scheme.

%%%%%%%%%%%%%%%%%%%%%%%%%%%%%%%%%%%%%%%%%%%%%%%%%%%%%%%%%%%%%%%%%%%%%%%%%%%%%
%%%%%%%%%%%%%%%%%%%%%%%%%%%%%%%%%%%%%%%%%%%%%%%%%%%%%%%%%%%%%%%%%%%%%%%%%%%%%

\section{Discretization in space and time}\label{sec:discretization}
%
% In this section, we introduce finite dimensional submanifolds $\densN$ in the space of densities $\dens$,
% which we equip with an approximation of the $L^2$-Wasserstein metric.
% Then we discuss properties of the restrictions of the entropy and the energy to $\densN$.

%%%%%%%%%%%%%%%%%%%%%%%%%%%%%%%%%%%%%%%%%%%%%%%%%%%%%%%%%%%%%%%%%%%%%%%%%%%%%

\subsection{Inverse distribution functions}
Before defining the discrete quantities, let us recall some basic facts from the continuous context.
We denote by
\begin{align*}
  \dens = \left\{ u:\Omega\to\setR_+\,:\,\intom u(x)\dd x = M\right\}
\end{align*}
the space of densities of total mass $M$ on $\Omega$,
and we endow $\dens$ with the $L^2$-Wasserstein metric $\wass$.
We refer to \cite{VilBook} for a comprehensive introduction to the topic.
For our purposes here, it suffices to know that convergence with respect to $\wass$ 
is equivalent to weak-$\star$ convergence in $L^1(\Omega)$,
and that the $L^2$-Wasserstein distance on $\dens$ is isometrically equivalent 
to the usual $L^2$-distance on the space
\begin{align*}
  \xspc = \left\{ \theX:[0,M]\to\Omega\,:\,\text{$\theX$ continuous and strictly increasing, with $\theX(0)=a$, $\theX(M)=b$}\right\}
\end{align*}
of inverse distribution functions $\theX$.
The isometry is given as follows.
\begin{lem}
  \label{lem:flat}
  Given $u^0,u^1\in\dens$, introduce their Lagrangian maps $\theX^0,\theX^1\in\xspc$ such that
  \begin{align*}
    \xi = \int_0^{\theX^j(\xi)}u^j(x)\dd x \quad \text{for all $\xi\in[0,M]$}.
  \end{align*}
  Then
  \begin{align*}
    \wass(u^0,u^1) = \|\theX^0-\theX^1\|_{L^2([0,M])}.
  \end{align*}
\end{lem}
Above, the name \emph{Lagrangian map} is underlined by the following change of variables formula,
\begin{align}
  \label{eq:pushforward}
  \intom \varphi(x) u(x) \dd x = \int_0^M \varphi\big(\theX(\xi)\big) \dd\xi,
\end{align}
that holds for every bounded and continuous test function $\varphi\in C^0([a,b])$.

\subsection{Ansatz space}\label{sec:ansatz}
Fix a discretization parameter $K\in\setN$, which is the number of degress of freedom plus one.
We will need both the integers and the half-integers between $0$ and $K$,
that is
\begin{align*}
  \ivalp = \{1,2,\ldots,K-1\},\quad
  \ival = \ivalp\cup\{0,K\},\quad\text{and}\quad
  \hval = \Big\{\frac12,\frac32,\ldots,K-\frac12\Big\}.
  % \quad
  % \aval=\ival\cup\hval.
\end{align*}
For discretization of $[0,M]$,
introduce the equidistant mass grid $(\xi_0,\ldots,\xi_K)$ with 
\begin{align*}
  \xi_k=k\delta\quad\text{for}\quad\delta:=M/K.  
\end{align*}
For discretization of $\Omega=[a,b]$, 
we consider (non-equidistant) grids from
\begin{align*}
  \xseqN = \big\{ \xvec = (x_1,\ldots,x_{K-1}) \,\big|\, a < x_1 < \ldots < x_{K-1} < b \big\} \subseteq (a,b)^{K-1}.
\end{align*}
By definition, $\xvec\in\xseqN$ is a vector with $K-1$ components,
but we shall frequently use the convention that $x_0=a$ and $x_K=b$.
In the convex set $\xspc$ of inverse distribution functions,
we single out the $(K-1)$-dimensional open and convex subset
\begin{align*}
  \xspc_\theh = \big\{\theX\in\xspc\,\big|\, \text{$\theX$ is affine on each $[\xi_{k-1},\xi_k]$}\big\}.
\end{align*}
Functions $\theX\in\xspc_\theh$ are called \emph{Lagrangian maps},
since they map the (fixed reference) mesh $(\xi_0,\xi_1,\ldots,\xi_K)$ to a (variable) mesh $\xvec\in\xseqN$. 
There is a one-to-one correspondence between grid vectors $\xvec\in\xseqN$
and inverse distribution function $\theX\in\xspc_\theh$,
explicitly given by
\begin{align}
  \label{eq:cX}
  \theX = \cX_\theh[\xvec] = \sum_{k\in\ival} x_k \hatf_k ,
\end{align}
where the $\hatf_k:[0,M]\to\setR$ are the usual affine hat functions, with $\hatf_k(\xi_\ell)=\delta_{k,\ell}$.
Further, the density function $\cf_\theh[\xvec]\in\dens$ associated to $\cX_\theh[\xvec]$ is
\begin{align}
  \label{eq:cf}
  \cf_\theh[\xvec](x)= \sum_{\kappa\in\hval} z_\kappa\indy_{(x_\kappm,x_\kappp]}(x) ,
\end{align}
where the vector 
\begin{align}
  \label{eq:zk}
  \zvec =\cz_\theh[\xvec]=(z_{1/2},\ldots,z_{K-1/2}) \quad\text{of weights}\quad z_\kappa = \frac{\delta}{x_\kappp-x_\kappm}
\end{align}
is such that each interval $(x_\kappm,x_\kappp]$ contains the same amount $\delta$ of total mass.
The following convention reflects the no-flux boundary conditions:
\begin{align}
  \label{eq:zconvention}
  z_{-\frac12} = z_{\frac12}, \quad z_{K+\frac12}=z_{K-\frac12}.
\end{align}
We finally introduce the associated $(K-1)$-dimensional submanifold $\densN := \cf_\theh[\xseqN]\subset\dens$
as the image of the injective map $\cf_\theh:\xseqN\to\densN$.

%%%%%%%%%%%%%%%%%%%%%%%%%%%%%%%%%%%%%%%%%%%%%%%%%%%%%%%%%%%%%%%%%%%%%%%%%%%%%
%
\subsection{A metric on the ansatz space}
Below, we define a ``Wasserstein-like'' metric $\wassN$ on the ansatz space $\densN$.
For motivation of that definition, observe that $\densN$ is a geodesic submanifold of $\dens$, 
hence the restriction $\widetilde\wassN$ of the genuine $L^2$-Wasserstein distance $\wass$ to $\densN$ 
appears as a natural candidate for $\wassN$.
Thanks to the flatness of $\wass$ in one space dimension, see Lemma \ref{lem:flat},
the pull-back metric of $\wass$ on $\xseqN$ induced by $\cf_\theh$ is a homogeneous quadratic form.
More precisely,
\begin{align}
  \label{eq:metricondensN}
  \wass\big(\cf_\theh[\xvec^0],\cf_\theh[\xvec^1]\big)^2 = \sum_{k=1}^{K-1}(\xvec^1_k-\xvec^0_k)\widetilde{\Wmat}_{k\ell}(\xvec^1_\ell-\xvec^0_\ell)
  \quad \text{for all $\xvec^0,\xvec^1\in\xseqN$},
\end{align}
where the positive matrix $\widetilde{\Wmat}\in\setR^{(K-1)\times(K-1)}$ is tridiagonal.
This approach has been followed in our previous work \cite{dde}.

Here, we take a modified approach and use \eqref{eq:metricondensN} to \emph{define} a metric $\wassN$ on $\densN$,
but with the simpler matrix $\delta\eins_{K-1}$ in place of $\widetilde{\Wmat}$ above.
In other words: up to a factor $\delta^{1/2}$, the pull-back metric of $\wassN$ via $\cf_\theh$ is the usual Euclidean distance on $\xseqN$.
\begin{rmk}
  Our proof of convergence heavily relies on several explicit estimates of quantities with respect to the metric $\wassN$.
\end{rmk}
With the rescaled scalar product $\spr{\cdot}{\cdot}$ and norm $\nrm{\cdot}$ defined for $\vvec,\wvec\in\setR^{K-1}$ by
\begin{align*}
  \spr{\vvec}{\wvec} = \delta\sum_{k=1}^{K-1}v_kw_k, \qquad \nrm{\vvec} = \left(\delta\sum_{k=1}^{K-1}v_k^2\right)^{1/2},
\end{align*}
the distance $\wassN$ is conveniently written as
\begin{align*}
  \wassN(\cf_\theh[\xvec^0],\cf_\theh[\xvec^1]) 
  = \|\xvec^1-\xvec^0\|_\delta.
\end{align*}
In \cite[Lemma 3.2]{dde}, we have shown the following.
\begin{lem}
  \label{lem:metricequivalent}
  $\wassN$ is equivalent to the Wasserstein metric restricted to $\densN$, uniformly in $K$:
  \begin{align}
    \label{eq:metricequivalent}
    \frac16\wassN(u_0,u_1)^2\le\wass(u_0,u_1)^2 \le\wassN(u_0,u_1)^2 \quad \text{for all $u_0,u_1\in\densN$}.
  \end{align}
\end{lem}
%
% \begin{proof}
%   it is shown that the matrix $\widetilde\Wmat$ from \eqref{eq:metricondensN} satisfies
%   \begin{align*}
%     \frac\delta6\eins_{K-1}\le\widetilde{\Wmat}\le\delta\eins_{K-1}.
%   \end{align*}
%   This immediately implies \eqref{eq:metricequivalent}.
% \end{proof}
%
Note that, as a direct consequence of \eqref{eq:metricequivalent}, 
we obtain that
\begin{align*}
  \left\|\cX_\theh[\xvec^0]-\cX_\theh[\xvec^1]\right\|_{L^2([0,M])}
  \le\left\|\xvec^0-\xvec^1\right\|_\delta.
\end{align*}
We shall not elaborate further on the point in which sense the thereby defined metric $\wassN$ 
is a good approximation of the $L^2$-Wasserstein distance on $\densN$.
However, Theorem \ref{thm:main} validates our choice a posteriori.
For results concerning the $\Gamma$-convergence of discretized transport metrics to the Wasserstein distance
see \cite{GigM}.

%%%%%%%%%%%%%%%%%%%%%%%%%%%%%%%%%%%%%%%%%%%%%%%%%%%%%%%%%%%%%%%%%%%%%%%%%%%%%
%
\subsection{Functions on $\densN$}
When discussing functions on $\densN$ in the following, 
we always assume that these are given in the form $f:\xseqN\to\setR$.
We denote the first and second derivatives of $f$ 
by $\grad f:\xseqN\to\setR^{K-1}$ and by $\grad^2 f:\xseqN\to\setR^{(K-1)\times(K-1)}$, respectively,
with components
\begin{align}
  \label{eq:grad}
  [\grad f(\xvec)]_k = \partial_{x_k} f(\xvec) 
  \quad\text{and}\quad 
  [\grad^2 f(\xvec)]_{k,l} = \partial_{x_k}\partial_{x_l} f(\xvec).
\end{align}
\begin{xmp}
  Each component $z_\kappa$ of $\zvec=\cz_\theh[\xvec]$ is a function on $\xseqN$, 
  and
  \begin{align}
    \label{eq:zrule}
    \grad z_\kappa = -z_\kappa^2\,\frac{\ee_{\kappa+\frac12}-\ee_{\kappa-\frac12}}{\delta},
  \end{align}
  where $\ee_k\in\setR^{K-1}$ is the $k$th canonical unit vector, with the convention $\ee_0=\ee_K=0$.  
\end{xmp}
We introduce further the gradient
\begin{align*}
  \wgrad f(\xvec) = \delta^{-1}\grad f(\xvec),
\end{align*}
where the scaling by $\delta^{-1}$ is chosen such that, for arbitrary vectors $\vvec\in\setR^{K-1}$,
\begin{align*}
  \spr{\vvec}{\wgrad f(\xvec)} = \sum_{k=1}^{K-1}v_k\partial_{x_k}f(\xvec).
\end{align*}
The \emph{gradient flow} of a function $f$ on $\densN$ with respect to $\wassN$ is then defined
as the solution $\xvec:[0;\infty)\to\xseqN$ for the system of ordinary differential equations
\begin{align}
  \label{eq:gflow}
  \dot{\xvec} = - \wgrad f(\xvec), 
  \qquad\text{or, more explicitly},\qquad
  \dot x_k = -\delta^{-1}\partial_{x_k}f(\xvec), \quad\text{for each $k\in\ivalp$}.
\end{align}

\subsubsection{The discretized Boltzmann entropy}\label{sec:Boltzmann}
The Boltzmann entropy $\HF$ as defined in \eqref{eq:intro.entropy} is a non-negative functional on $\dens$,
which vanishes precisely on the constant function $u\equiv M/(b-a)$.
In analogy to \cite{dde}, we introduce a discretization $\HFz:\xseqN\to\setR$ of the Boltzmann entropy $\HF$
by restriction to $\densN$:
\begin{align*}
  \HFz(\xvec) := \HF(\cf_\theh[\xvec]) 
  = \int_\Omega \cf_\theh[\xvec] \ln \cf_\theh[\xvec] \dd x - \HF_0
  % = -\int_0^M\ln \partial_\xi\cX_\theh[\xvec] \dd\xi - \HF_0,
  = \delta\sum_{\kappa\in\hval}\ln z_\kappa - \HF_0,
\end{align*}
where $\HF_0$ was defined in \eqref{eq:intro.entropy}, and $\zvec=\cz_\theh[\xvec]$.
Naturally, $\HFz$ inherits non-negativity, and vanishes only for $\xvec$ with $x_k=a+(b-a)k/K$.
For the derivatives, we obtain --- using the rule \eqref{eq:zrule} ---
\begin{align}
  \label{eq:Hgrad}
  \grad\HFz(\xvec) 
  &= -\delta\sum_{\kappa\in\hval}z_\kappa \frac{\ee_{\kappm}-\ee_{\kappp}}\delta
  = \delta\sum_{k\in\ivalp}\frac{z_\kph-z_\kmh}\delta\ee_k, \\
  \label{eq:Hhess}
  \grad^2\HFz(\xvec) 
  &=\delta\sum_{\kappa\in\hval}z_\kappa^2 \left(\frac{\ee_{\kappm}-\ee_{\kappp}}\delta\right) \left(\frac{\ee_{\kappm}-\ee_{\kappp}}\delta\right)^T.
\end{align}
It is obvious that $\grad^2\HFz$ is positive semi-definite, i.e., that $\HFz$ is convex.

\subsubsection{The discretized Fisher information}\label{sec:FisherInf}
%
% Unlike the discretization $\HFz$ of the entropy $\HF$,
The discrete Fisher information $\Fz:\xseqN\to\setR$ is \emph{not} defined by restriction of $\F$ from \eqref{eq:intro.fisher}. 
Instead, we mimick \eqref{eq:magic} and define accordingly
\begin{align*}
  \Fz(\xvec) = \frac12\nrm{\wgrad\HFz(\xvec)}^2
  = \frac\delta2\sum_{k\in\ivalp}\Big(\frac{z_\kph-z_\kmh}\delta\Big)^2,
\end{align*}
using \eqref{eq:Hgrad}.
Thanks to this simple structure, the gradient flow equation for $\Fz$ has an explicit and compact representation.
Using the rule \eqref{eq:zrule}, the representation \eqref{eq:Hhess} and the convention \eqref{eq:zconvention},
we obtain with $\zvec=\cz_\theh[\xvec]$:
\begin{align}
  \nonumber
  \wgrad\Fz(\xvec) &= \delta^{-2}\grad^2\HFz(\xvec)\grad\HFz(\xvec)  \\
  &= \sum_{\kappa\in\hval,\,k\in\ivalp}
  z_\kappa^2 \left(\frac{z_\kph-z_\kmh}\delta\right)\left(\frac{\ee_\kappp-\ee_\kappm}\delta\right)\left(\frac{\ee_\kappp-\ee_\kappm}\delta\right)^T\ee_k\\
  \label{eq:Fgrad}
  &= \sum_{\kappa\in\hval} z_\kappa^2\left(\frac{z_{\kappa+1}-2z_\kappa+z_{\kappa-1}}{\delta^2}\right) \left(\frac{\ee_\kappp-\ee_\kappm}\delta\right).
\end{align}
This should be understood as a discretization of the differential operator $(Z^2Z_{\xi\xi})_\xi$ appearing
on the right hand side of \eqref{eq:zeq}.
\begin{rmk}
  Without calculating the second derivative $\grad^2\Fz$ explicitly,
  we remark that it is unbounded from below on $\xspc_\theh$,
  hence $\Fz$ is \emph{not} $\lambda$-convex for any $\lambda\in\setR$.
  This is in agreement with the fact that already the original Fisher information $\F$ 
  is \emph{not} geodesically $\lambda$-convex in the Wasserstein metric, see \cite{CarS}.  
\end{rmk}

%%%%%%%%%%%%%%%%%%%%%%%%%%%%%%%%%%%%%%%%%%%%%%%%%%%%%%%%%%%%%%%%%%%%%%%%%%%%%
%
\subsection{Time stepping}\label{sec:MM}
For the definition of the fully discrete scheme for solution of \eqref{eq:gradflow},
we discretize the spatially discrete gradient flow equation
\begin{align}
  \label{eq:sdgradflow}
  \dot{\xvec} = -\wgrad\Fz(\xvec)
\end{align}
also in time, using \emph{minimizing movements}.
To this end, fix a time step with $\tau>0$; 
we combine the spatial and temporal mesh widths in a single discretization parameter \[\Delta=(\tau;\theh).\]
For each $\yvec\in\xseqN$, introduce the \emph{Yosida-regularized Fisher information} $\Fy(\cdot;\yvec):\xseqN\to\setR$
by
\begin{align*}
  \Fy(\xvec;\yvec) = \frac1{2\tau}\nrm{\xvec-\yvec}^2+\Fz(\xvec).
\end{align*}
A fully discrete approximation $(\xvec_\Delta^n)_{n=0}^\infty$ of \eqref{eq:sdgradflow} is now defined inductively
from a given initial datum $\xvec_\Delta^0$ by choosing each $\xvec_\Delta^n$ 
as a global minimizer of $\Fy(\cdot;\xvec_\Delta^{n-1})$.
Below, we prove that such a minimizer always exists, see Lemma \ref{lem:cfl}.

In practice, one wishes to define $\xvec_\Delta^n$ as --- preferably unique --- solution 
of the Euler-Lagrange equations associated to $\Fy(\cdot;\xvec_\Delta^{n-1})$,
which leads to the implicit Euler time stepping:
\begin{align}
  \label{eq:euler}
  \frac{\xvec-\xvec_\Delta^{n-1}}{\tau} = -\wgrad\Fz(\xvec).
\end{align}
Using the explicit representation \eqref{eq:Fgrad} of $\grad\Fz$, 
it is immediately seen that \eqref{eq:euler} is indeed the same as \eqref{eq:dgf}.
Equivalence of \eqref{eq:euler} and the minimization problem is guaranteed at least for sufficiently small $\tau>0$.
\begin{prp}
  \label{prp:wellposed}
  For each discretization $\Delta$ and every initial condition $\xvec^0\in\xseqN$,
  the sequence of equations \eqref{eq:euler} can be solved inductively.
  Moreover, if $\tau>0$ is sufficiently small with respect to $\delta$ and $\Fz(\xvec^0)$,
  then each equation \eqref{eq:euler} possesses a unique solution with $\Fz(\xvec)\le\Fz(\xvec^0)$,
  and that solution is the unique global minimizer of $\Fy(\cdot;\xvec_\Delta^{n-1})$.
\end{prp}
The proof of this proposition is a consequence of the following rather technical lemma.
\begin{lem}
  \label{lem:cfl}
  Fix a spatial discretization parameter $\delta$ and a bound $C>0$.
  Then for every $\yvec\in\xseqN$ with $\Fz(\yvec)\le C$, the following are true:
  \begin{itemize}
  \item for each $\tau>0$, 
    the function $\Fy(\cdot;\yvec)$ possesses at least one global minimizer $\xvec^*\in\xseqN$;
  \item there exists a $\tau_C>0$ independent of $\yvec$ such that for each $\tau\in(0,\tau_C)$,
    the global minimizer $\xvec^*\in\xseqN$ is strict and unique, 
    and it is the only critical point of $\Fy(\cdot;\yvec)$ with $\Fz(\xvec)\le C$.
  \end{itemize}
\end{lem}
\begin{proof}
  First, observe that the sublevel $A_C:=\Fz^{-1}([0,C+1])\subset\xseqN$ is a compact subset of $\setR^{K-1}$.
  Indeed, $A_C$ is a relatively closed subset of $\xseqN$ by continuity of $\Fz$.
  Moreover, thanks to \eqref{eq:Festimate}, every $\xvec\in A_C$ satisfies $x_\kappp-x_\kappm\ge \underline{x}$ for all $\kappa\in\hval$
  with a positive constant $\underline x$ that depends on $C$ only.
  Thus $A_C$ does not touch the boundary (in the ambient $\setR^{K-1}$) of $\xseqN$.
  Consequently, $A_C$ is closed and bounded in $\setR^{K-1}$.
  
  Let $\yvec\in\xseqN$ with $\Fz(\yvec)\le C$ be given.
  The restriction of the continuous function $\Fy(\cdot;\yvec)$ to the compact and nonempty (since it contains $\yvec$) set $A_C$ 
  possesses a minimizer $\xvec^*\in A_C$.
  We clearly have $\Fz(\xvec^*)\le\Fz(\yvec)\le C$, 
  and so $\xvec^*$ lies in the interior of $A_C$ and therefore is a global minimizer of $\Fy(\cdot;\yvec)$.
  This proves the first claim.

  Since $\Fz:\xseqN\to\setR$ is smooth, its restriction to $A_C$ is $\lambda_C$-convex with some $\lambda_C\le0$,
  i.e., $\grad^2\Fz(\xvec)\ge\lambda_C\eins_{K-1}$ for all $\xvec\in A_C$.
  Independently of $\yvec$, we have that
  \begin{align*}
    \grad^2\Fy(\xvec;\yvec) = \grad^2\Fz(\xvec) + \frac\delta\tau\eins_{K-1},
  \end{align*}
  which means that $\xvec\mapsto\Fy(\xvec;\yvec)$ is strictly convex on $A_C$
  if
  \begin{align*}
    0< \tau < \tau_C:=\frac\delta{(-\lambda_C)}.
  \end{align*}
  Consequently, each such $\Fy(\cdot;\yvec)$ has at most one critical point $\xvec^*$ in the interior of $A_C$, 
  and this $\xvec^*$ is necessarily a strict global minimizer.
\end{proof}

%%%%%%%%%%%%%%%%%%%%%%%%%%%%%%%%%%%%%%%%%%%%%%%%%%%%%%%%%%%%%%%%%%%%%%%%%%%%%
%
\subsection{Spatial interpolations}
Consider a fully discrete solution $(\xvec_\Delta^n)_{n=0}^\infty$.
For notational simplification, 
we write the entries of the vectors $\xvec_\Delta^n$ and $\zvec_\Delta^n=\cz_\theh[\xvec_\Delta^n]$ 
as $x_k$ and $z_\kappa$, respectively, whenever there is no ambiguity in the choice of $\Delta$ and the time step $n$.

Recall that $u_\Delta^n=\cf_\theh[\xvec_\Delta^n]\in\densN$ defines a sequence of densitites on $[a,b]$
which are piecewise constant with respect to the (non-uniform) grid $(a,x_1,\ldots,x_{K-1},b)$.
To facilitate the study of convergence of weak derivatives, 
we introduce also \emph{piecewise affine} interpolations $\hatz_\Delta^n:[0,M]\to\setR_+$ and $\hatu_\Delta^n:[a,b]\to\setR_+$. 

In addition to $\xi_k=k\delta$ for $k\in\ival$, 
introduce the intermediate points $\xi_\kappa=\kappa\delta$ for $\kappa\in\hval$.
Accordingly, introduce the intermediate values for the vectors $\xvec_\Delta^n$ and $\zvec_\Delta^n$:
\begin{align*}
  x_\kappa = \frac12\big(x_\kappp+x_\kappm) \quad \text{for $\kappa\in\hval$}, \\
  z_k = \frac12\big(z_\kph+z_\kmh\big) \quad \text{for $k\in\ivalp$}.
\end{align*}
Now define 
\begin{itemize}
\item $\hatz_\Delta^n:[0,M]\to\setR$ as the piecewise affine interpolation
  of the values $(z_{\frac12},z_{\frac32},\ldots,z_{K-\frac12})$ 
  with respect to the equidistant grid $(\xi_{\frac12},\xi_{\frac32},\ldots,\xi_{K-\frac12})$,
  and
\item $\hatu_\Delta^n:[a,b]\to\setR$ as the piecewise affine function with
  \begin{align}
    \label{eq:locaffine0}
    \hatu_\Delta^n\circ\theX_\Delta^n = \hatz_\Delta^n.
    % \widetilde{\cf_\theh}[\xvec] \circ\cX_\theh[\xvec] = \widehat{\cfM_\theh}[\xvec].
  \end{align}
\end{itemize}
Our convention is that 
$\hatz_\Delta^n(\xi)=z_{\frac12}$ for $0\le\xi\le\delta/2$ and $\hatz_\Delta^n(\xi)=z_{K-\frac12}$ for $M-\delta/2\le\xi\le M$,
and accordingly
$\hatu_\Delta^n(x)=z_{\frac12}$ for $x\in[a,x_{\frac12}]$ and $\hatu_\Delta^n(x)=z_{K-\frac12}$ for $x\in[x_{K-\frac12},b]$.
The definitions have been made such that
\begin{align}
  \label{eq:interpol1}
  x_k = \theX_\Delta^n(\xi_k),\quad z_k = \hatz(\xi_k) = \hatu (x_k) \quad \text{for all $k\in\ival\cup\hval$}.
\end{align}
Notice that $\hatu_\Delta^n$ is piecewise affine with respect to the ``double grid'' $(x_0,x_\frac12,x_1,\ldots,x_{K-\frac12},x_K)$,
but in general not with respect to the subgrid $(x_0,x_1,\ldots,x_K)$.
By direct calculation, we obtain for each $k\in\ivalp$ that
\begin{equation}
  \label{eq:tildeux}
  \begin{split}
    \partial_x\hatu\big|_{(x_\kmh,x_k)}  
    &=\frac{z_k - z_\kmh}{x_k-x_\kmh} = \frac{z_\kph - z_\kmh}{x_k-x_{k-1}} 
    =z_\kmh\frac{z_\kph - z_\kmh}{\delta}, \\
    \partial_x\hatu\big|_{(x_k,x_\kph)} 
    &= \frac{z_\kph - z_k}{x_\kph-x_k} = \frac{z_\kph - z_\kmh}{x_{k+1}-x_k} 
    = z_\kph\frac{z_\kph - z_\kmh}{\delta}.
  \end{split}
\end{equation}
Trivially, we also have that $\partial_x\hatu$ vanishes identically on the intervals $(a,x_\imh)$ and $(x_\Kmh,b)$.

%%%%%%%%%%%%%%%%%%%%%%%%%%%%%%%%%%%%%%%%%%%%%%%%%%%%%%%%%%%%%%%%%%%%%%%%%%%%%
%
\section{A priori estimates and compactness}
\label{sec:apriori}
%
%%%%%%%%%%%%%%%%%%%%%%%%%%%%%%%%%%%%%%%%%%%%%%%%%%%%%%%%%%%%%%%%%%%%%%%%%%%%%
% 
Throughout this section, we consider a sequence $\Delta=(\tau;\delta)$ of discretization parameters 
such that $\delta\to0$ and $\tau\to0$ in the limit, formally denoted by $\Delta\to0$.
We assume that a fully discrete solution $(\xvec_\Delta^n)_{n=0}^\infty$ is given for each $\Delta$-mesh, 
defined by inductive minimization of the respective $\Fy$.
The sequences $\baru_\Delta$, $\hatu_\Delta$, $\hatz_\Delta$ and $\theX_\Delta$ of spatial interpolations 
are defined from the respective $\xvec_\Delta$ accordingly.
For the sequence of initial conditions $\xvec_\Delta^0$, 
we assume that $\hatu_\Delta^0\to u^0$ weakly in $L^1(\Omega)$, 
that there is some finite $\olH$ with
\begin{align}
  \label{eq:Hbound}
  \HFz(\xvec_\Delta^0) \le \olH \quad\text{for all $\Delta$},
\end{align}
and that
\begin{align}
  \label{eq:Fbound}
  (\tau+\delta)\Fz(\xvec_\Delta^0) \to 0 \quad \text{as $\Delta\to0$}.
\end{align}
Further, we use $\ti{q}$ to denote the constant in time interpolations of sequences $(q^n)_{n=0}^\infty$ with step size $\tau>0$,
that is
\begin{align*}
  \ti{q}(t) := q^n \quad\text{for $t\in((n-1)\tau],n\tau]$}, \quad \ti{q}(0):=q^0.
\end{align*}
% In particular, we shall study the convergence of $\ti{u_\Delta}:[0,\infty)\to\dens$ as $\Delta\to0$.

\subsection{Energy inequality}\label{sec:energydiss}
The following basic energy estimates are classical for gradient flows.
\begin{lem}
  \label{lem:dF_bounded}
  One has that $\Fz$ is monotone, i.e., $\Fz(\xvec_\Delta^n)\le\Fz(\xvec_\Delta^{n-1})$, 
  and further:
  \begin{align}
    &\Fz(\xvec_\Delta^n) \leq \Fz(\xvec_\Delta^0) \quad\text{for all $n\geq 0$}, \label{eq:dissipation}\\
    &\|\xvec_\Delta^{\overline n} - \xvec_\Delta^{\underline n}\|_\delta^2
    \leq 2\Fz(\xvec_\Delta^0)\,(\overline n-\underline n)\tau
    \quad\text{for all $\overline n\geq\underline n\geq 0$}, \label{eq:uniform_time} 
    \\
    &\tau\sum_{n=1}^\infty\nrm{\frac{\xvec_\Delta^n-\xvec_\Delta^{n-1}}\tau}^2
    = \tau \sum_{n=1}^\infty\nrm{\wgrad\Fz(\xvec_\Delta^n)}^2
    \le 2\Fz(\xvec_\Delta^0). \label{eq:eee}
    % \tau\sum_{j=m+1}^n\|\grad\Fz(\xvec_\Delta^j)\|_2 &\leq \sqrt{2\delta T}\sqrt{\Fz(\xvec_\Delta^{n})-\Fz(\xvec_\Delta^m)}
    %	\quad\textnormal{for all } m\geq n \geq 0. \label{eq:dissipationT}
  \end{align}
\end{lem}
\begin{proof}
  The monotonicity \eqref{eq:dissipation} follows (by induction on $n$) from the definition of $\xvec_\Delta^n$ 
  as minimizer of $\Fy(\cdot;\xvec_\Delta^{n-1})$:
  \begin{align*}
    \Fz(\xvec_\Delta^n) &\le \frac1{2\tau}\|\xvec_\Delta^n-\xvec_\Delta^{n-1}\|_\delta^2 + \Fz(\xvec_\Delta^n) 
    =\Fy(\xvec_\Delta^n;\xvec_\Delta^{n-1}) \le \Fy(\xvec_\Delta^{n-1};\xvec_\Delta^{n-1}) = \Fz(\xvec_\Delta^{n-1}).
  \end{align*}
  Moreover, summation of these inequalities from $n=\underline n+1$ to $n=\overline n$ yields
  \begin{align*}
    \frac\tau2\sum_{n=\underline n+1}^{\overline n} \bigg[\frac{\|\xvec_\Delta^n-\xvec_\Delta^{n-1}\|_\delta}{\tau}\bigg]^2
    \le \Fz(\xvec_\Delta^{\underline n})-\Fz(\xvec_\Delta^{\overline n}) \le \Fz(\xvec_\Delta^0).
  \end{align*}
  For $\underline n=0$ and $\overline n\to\infty$, we obtain the first part of \eqref{eq:eee}.
  The second part follows by \eqref{eq:euler}.
  If instead we combine the estimate with Jensen's inequality,
  we obtain
  \begin{align*}
    \big\|\xvec_\Delta^{\overline n}-\xvec_\Delta^{\underline n}\big\|_\delta
    \le \tau\sum_{n=\underline n+1}^{\overline n}\frac{\big\|\xvec_\Delta^n-\xvec_\Delta^{n-1}\big\|_\delta}{\tau}
    \le \bigg(\tau\sum_{n=\underline n+1}^{\overline n} \bigg[\frac{\|\xvec_\Delta^n-\xvec_\Delta^{n-1}\|_\delta}{\tau}\bigg]^2\bigg)^{1/2}
    \big(\tau(\overline n-\underline n)\big)^{1/2},
  \end{align*}
  which leads to \eqref{eq:uniform_time}.
\end{proof}

\subsection{Entropy dissipation}
The key to our convergence analysis is a refined a priori estimate,
which follows from the dissipation of the entropy $\HFz$ along the fully discrete solution.
\begin{lem}\label{lem:dFdH_bounded}
  One has that $\HFz$ is monotone, i.e., $\HFz(\xvec_\Delta^n)\le\HFz(\xvec_\Delta^{n-1})$, 
  and further:
  \begin{align}
    \label{eq:dFdH_bounded}
    \tau\sum_{n=1}^\infty \delta\sum_{\kappa\in\hval} z_\kappa^2\left(\frac{z_{\kappa+1}-2z_\kappa+z_{\kappa-1}}{\delta^2}\right)^2
    \le \HFz(\xvec_\Delta^0).
  \end{align}
\end{lem}
\begin{proof}
  By convexity of $\HFz$ and the discrete evolution \eqref{eq:euler}, 
  we have
  \begin{align*}
    \HFz(\xvec_\Delta^{n-1}) - \HFz(\xvec_\Delta^n) 
    \ge \spr{\wgrad\HFz(\xvec_\Delta^n)}{\xvec_\Delta^{n-1}-\xvec_\Delta^n} 
    = \tau\spr{\wgrad\HFz(\xvec_\Delta^n)}{\wgrad\Fz(\xvec_\Delta^n)}
  \end{align*}
  for each $n=1,2,\ldots$
  Evaluate the (telescopic) sum with respect to $n$ and use that $\HFz\ge0$ to obtain
  \begin{align*}
    \tau\sum_{n=1}^\infty\spr{\wgrad\HFz(\xvec_\Delta^n)}{\wgrad\Fz(\xvec_\Delta^n)} \le \HFz(\xvec_\Delta^0).
  \end{align*}
  It remains to make the scalar product explicit, using \eqref{eq:Hgrad} and \eqref{eq:Fgrad}:
  \begin{align*}
    \spr{\wgrad\HFz}{\wgrad\Fz} 
    &= \delta\sum_{\kappa\in\hval,\,k\in\ivalp} z_\kappa^2\left(\frac{z_{\kappa+1}-2z_\kappa+z_{\kappa-1}}{\delta^2}\right)
    \left(\frac{z_\kph-z_\kmh}\delta\right)\,\left(\frac{\ee_\kappp-\ee_\kappm}\delta\right)^T\ee_k \\
    % \label{eq:dFdH}
    &= \delta\sum_{\kappa\in\hval} z_\kappa^2\Big(\frac{z_{\kappa+1}-2z_\kappa+z_{\kappa-1}}{\delta^2}\Big)^2,
  \end{align*}
  using that $z_{-\frac12}=z_{\frac12}$ and $z_{K+\frac12}=z_{K-\frac12}$, according to our convention \eqref{eq:zconvention}.
\end{proof}
We draw several conclusions from \eqref{eq:dFdH_bounded}.
The first is an a priori estimate on the the $\xi$-derivative of the affine functions $\hatz_\Delta^n$.
\begin{lem}
  One has that 
  \begin{align}
    \label{eq:L4bound}
    \tau\sum_{n=1}^\infty\big\|\partial_\xi\hatz_\Delta^n\big\|_{L^4([0,M])}^4 
    = \tau\sum_{n=1}^\infty\delta\sum_{k\in\ivalp}\left(\frac{z^n_\kph-z^n_\kmh}\delta\right)^4
    \le 9\olH. % \HFz(\xvec_\Delta^0) .
    % \|\ti{(\hatz_\Delta)_\xi}\|_{L^4([0,T]\times[0,M])}^4 \leq 9\overline{\HFz}.
    % \leq 9 \sum_{n=0}^N\tau\sum_{k=1}^{K}\delta z_\kmh^2 \left(\frac{z_\kmd -2z_\kmh +z_\kph}{\delta^2}\right)^2 \leq 9\HFz(\xvec_\Delta^0)
  \end{align}
\end{lem}
\begin{rmk}
  Morally, a bound on $\partial_\xi\hatz$ in $L^4([0,M])$ corresponds to a bound on $\partial_x\sqrt[4]{\hatu}$ in $L^4(\Omega)$.
\end{rmk}
%
% To motivate the calculations in the proof below, 
% we indicate the formal idea for a function $Z:[0,M]\to\setR$ that 
% is twice continuously differentiable with homogeneous Neumann boundary conditions:
% \begin{align*}
%   \int_0^M (\partial_\xi Z)^4 \dd\xi 
%   = -3\int_0^M Z\,(\partial_\xi Z)^2\,\partial_{\xi\xi} Z\dd\xi
%   \leq 3 \left(\int_0^M(\partial_\xi Z)^4\dd\xi\right)^{\frac{1}{2}}\left(\int_0^MZ^2(\partial_{\xi\xi}Z)^2\dd\xi\right)^{\frac{1}{2}},
% \end{align*}
% and this yields
% \begin{align*}
%   \|\partial_\xi Z\|_{L^4([0,M])}^4 \leq 9 \|Z\,\partial_{\xi\xi}Z\|_{L^2([0,M])}^2.
% \end{align*}
% Notice that the expression on the right is of the same shape as the quantity controlled by \eqref{eq:dFdH_bounded}.
% The technical difficulty to adapt this proof from $Z$ to $\hatz$ is to avoid the use of the second order derivative.
%
\begin{proof}
  Fix $n\in\setN$.
  Invoking our convention \eqref{eq:zconvention},
  one obtains
  \begin{align*}
    \big\|\partial_\xi\hatz_\Delta^n\big\|_{L^4(\Omega)}^4
    &= \sum_{k\in\ival} (z_\kph-z_\kmh)\left(\frac{z_\kph-z_\kmh}{\delta}\right)^3 \\
    &= -\sum_{k\in\hval} z_\kappa\left[\left(\frac{z_{\kappa+1}-z_\kappa}{\delta}\right)^3-\left(\frac{z_\kappa-z_{\kappa-1}}{\delta}\right)^3\right]
    = (A)
  \end{align*}
  Using the elementary identity $(p^3-q^3)=(p-q)(p^2+q^2+pq)$ and Young's inequality, 
  one obtains further 
  \begin{align*}
    (A) &= -\delta \sum_{\kappa\in\hval} z_\kappa\frac{z_{\kappa+1}-2z_\kappa+z_{\kappa-1}}{\delta^2}\times \\
    & \qquad \times
    \Bigg[\left(\frac{z_{\kappa+1}-z_\kappa}{\delta}\right)^2+\left(\frac{z_\kappa-z_{\kappa-1}}{\delta}\right)^2
    +\left(\frac{z_{\kappa+1}-z_\kappa}{\delta}\right)\left(\frac{z_\kappa-z_{\kappa-1}}{\delta}\right) \Bigg] \\
    &\leq \frac{3\delta}{2}\sum_{\kappa\in\hval} \left|z_\kappa\frac{z_{\kappa+1}-2z_\kappa+z_{\kappa-1}}{\delta^2}\right|
    \,\Bigg[\left(\frac{z_{\kappa+1}-z_\kappa}{\delta}\right)^2+\left(\frac{z_\kappa-z_{\kappa-1}}{\delta}\right)^2\Bigg] \\
    &\leq \frac32\left(\delta\sum_{\kappa\in\hval} z_\kappa^2\left[\frac{z_{\kappa+1}-2z_\kappa+z_{\kappa-1}}{\delta^2}\right]^2\right)^{1/2}
    \left(4\delta\sum_{k\in\ival}\left(\frac{z_\kph-z_\kmh}{\delta}\right)^4\right)^{1/2}. 
    % \\
    % &= 3\big[\grad\HFz(\xvec_\Delta^n)^T\Wmat^{-1}\grad\Fz(\xvec_\Delta^n)\big]^{1/2}\big\|\partial_\xi\hatz_\Delta^n\big\|_{L^4(\Omega)}^2.
  \end{align*}
  Note that the last sum above is again the $L^4$-norm of $\partial_\xi\hatz^n$. 
  Taking the square on both sides, dividing by the $L^4$-norm, summing over $n=1,2,\ldots$,
  and finally applying the entropy dissipation estimate \eqref{eq:dFdH_bounded}, one arrives at \eqref{eq:L4bound}.
\end{proof}
The a priori estimate \eqref{eq:L4bound} is the basis for almost all of the further estimates.
For instance, the following control on the oscillation of the $z$-values at neighboring grid points 
is a consequence of \eqref{eq:L4bound}.
\begin{lem}
  One has
  \begin{align}
    \label{eq:weakoscillation}
    \tau\sum_{n=1}^\infty\delta\sum_{k\in\ivalp}\bigg[\Big(\frac{z^n_\kph}{z^n_\kmh}-1\Big)^4+\Big(\frac{z^n_\kmh}{z^n_\kph}-1\Big)^4\bigg]
    \le 18(b-a)^4\HFz(\xvec_\Delta^0).
  \end{align}
  Moreover, given $T>0$, then for each $N\in\setN$ with $N\tau\le T$,
  one has
  \begin{align}
    \label{eq:oscillation}
    \tau\sum_{n=1}^N\delta\sum_{k\in\ivalp}\bigg[\Big(\frac{z^n_\kph}{z^n_\kmh}-1\Big)^2+\Big(\frac{z^n_\kmh}{z^n_\kph}-1\Big)^2\bigg]
    \le 6(b-a)^2T^{1/2}\HFz(\xvec_\Delta^0)^{1/2}\delta^{1/2}.
  \end{align}
\end{lem}
\begin{proof}
  Recall that $z_\kappa\ge\delta/(b-a)$ for all $\kappa$, see \eqref{eq:lohi}.
  Consider the first term in the inner summation in \eqref{eq:weakoscillation}:
  \begin{align*}
    \delta\sum_{k\in\ivalp}\Big(\frac{z^n_\kph}{z^n_\kmh}-1\Big)^4
    =\delta\sum_{k\in\ivalp}\Big(\frac{\delta}{z^n_\kmh}\Big)^4\Big(\frac{z^n_\kph-z^n_\kmh}\delta\Big)^4
    \le (b-a)^4\|\hatz_\Delta^n\|_{L^4(\Omega)}^4.
  \end{align*}
  The same estimate holds for the second term.
  The claim \eqref{eq:weakoscillation} is now directly deduced from \eqref{eq:L4bound} above.
  The proof of the second claim \eqref{eq:oscillation} is similar,
  using the Cauchy-Schwarz inequality instead of the modulus estimate:
  \begin{align*}
    \delta\sum_{k\in\ivalp}\Big(\frac{z^n_\kph}{z^n_\kmh}-1\Big)^2
    =\delta\sum_{k\in\ivalp}\Big(\frac{\delta}{z^n_\kmh}\Big)^2\Big(\frac{z^n_\kph-z^n_\kmh}\delta\Big)^2 
    \le\left(\delta\sum_{k\in\ivalp}\Big(\frac{\delta}{z^n_\kmh}\Big)^4\right)^{1/2}\|\hatz_\Delta^n\|_{L^4(\Omega)}^2.
  \end{align*}
  Use estimate \eqref{eq:xpower}, sum over $n=1,\ldots,N$, 
  and apply the Cauchy-Schwarz inequality to this second summation.
  This yields
  \begin{align*}
    \tau\sum_{n=1}^N\delta\sum_{k\in\ivalp}\Big(\frac{z^n_\kph}{z^n_\kmh}-1\Big)^2
    \le \delta^{1/2}(b-a)^2\left(\tau\sum_{n=1}^N1\right)^{1/2}\left(\tau\sum_{n=1}^\infty\|\hatz_\Delta^n\|_{L^4(\Omega)}^4\right)^{1/2}.
  \end{align*}
  Invoking again \eqref{eq:L4bound}, and recalling that $N\tau\le T$, we arrive at \eqref{eq:oscillation}.
\end{proof}
We are now going to prove the main consequence from the entropy dissipation \eqref{eq:dissipation},
namely a control on the total variation of $\sqrt{\hatu_\Delta^n}$.
This estimate is the key ingredient for obtaining strong compactness in Proposition \ref{prp:convergence2}.
Recall that several equivalent definitions of the \emph{total variation} of $f\in L^1(\Omega)$ exist.
Most generally,
\begin{align}
  \label{eq:TVorig}
  \tv{f} = \sup\left\{ \intom f(x)\partial_x\phi(x)\dd x\,;\,\phi\in C^{0,1}(\Omega),\,\sup_x|\phi(x)|\le 1\right\}.
\end{align}
Since we are dealing with functions $f:\Omega\to\setR$ that are piecewise smooth on intervals and only have jump discontinuities,
the following definition is most appropriate:
\begin{align}
  \label{eq:defTV}
  \tv{f} = \sup\left\{ \sum_{j=1}^{J-1} |f(r_{j+1})-f(r_{j})|\,:\,
  J\in\setN,\,a<r_1<r_2<\cdots<r_N<b\right\}.
\end{align}
Further recall the notation
\begin{align*}
  \llbracket f\rrbracket_{\bar x} = \lim_{x\downarrow\bar x}f(x) - \lim_{x\uparrow\bar x}f(x).
\end{align*}
for the height of the jump in $f(x)$'s value at $x=\bar x$.
\begin{lem}
  \label{lem:BVbound}
  One has
  \begin{align}
    \label{eq:BVbound}
    \tau\sum_{n=1}^\infty \tv{\partial_x\sqrt{\hatu_\Delta^n}}^2 
    \le 10(b-a)\olH.  %\HFz(\xvec_\Delta^0).
  \end{align}
\end{lem}
\begin{proof}
  Fix $n$.
  Observe that $\sqrt{\hatu_\Delta^n}$ is 
  smooth on $\Omega$ except for the points $x_{\frac12},x_1,\ldots,x_{K-\frac12}$,
  with derivatives given by
  \begin{align*}
    \partial_x\sqrt{\hatu_\Delta^n} = \frac1{2\sqrt{\hatu_\Delta^n}}\partial_x\hatu_\Delta^n,
    \quad
    \partial_{xx}\sqrt{\hatu_\Delta^n} = -\frac1{4\sqrt{\hatu_\Delta^n}^3}\big(\partial_x\hatu_\Delta^n\big)^2 \le 0.
  \end{align*}
  Therefore, $\partial_x\sqrt{\tilde u_\Delta^n}$ is monotonically decreasing
  in between the (potential) jump discontinuities at the points $x_{\frac12},x_1,\ldots,x_{K-\frac12}$.
  Further, recall that
  \begin{align}
    \label{eq:sqrtzero}
    \partial_x\sqrt{\hatu_\Delta^n}(x) = 0 \quad \text{for all $x\in(a,a+\delta/2)$ and all $x\in(b-\delta/2,b)$},
  \end{align}
  It follows that the supremum in \eqref{eq:defTV} can be realized (in the limit $\eps\downarrow0$)
  for a sequence of just $J=2(2K-1)$ many points $r_j^\eps$,
  chosen as follows: 
  \begin{align*}
    r_{2i-1}^\eps=x_{i/2}-\eps\quad\text{and}\quad r_{2i}^\eps=x_{i/2}+\eps,\quad \text{for $i=1,\ldots,2K-1$}.
  \end{align*}
  On the one hand,
  \begin{align}
    \label{eq:tv1}
    \lim_{\eps\downarrow0}\left|\partial_x\sqrt{\hatu_\Delta}(r_{2i-1}^\eps)-\partial_x\sqrt{\hatu_\Delta}(r_{2i}^\eps)\right| 
    = \left|\left\llbracket\partial_x\sqrt{\hatu_\Delta^n}\right\rrbracket_{x_{i/2}}\right|.
  \end{align}
  On the other hand, 
  since $\partial_x\sqrt{\hatu_\Delta^n}$ is monotone decreasing in between $r_{2i}^\eps$ and $r_{2i+1}^\eps$,
  and vanishes near the boundary by \eqref{eq:sqrtzero},
  we have that
  \begin{align}
    \nonumber
    \lim_{\eps\downarrow0}\sum_{i=1}^{2K-2} \left(\partial_x\sqrt{\hatu_\Delta^n}(r_{2i}^\eps)-\partial_x\sqrt{\hatu_\Delta^n}(r_{2i+1}^\eps)\right)
    &= \lim_{\eps\downarrow0}\sum_{i=1}^{2K-1} \left(\partial_x\sqrt{\hatu_\Delta^n}(r_{2i}^\eps)-\partial_x\sqrt{\hatu_\Delta^n}(r_{2i-1}^\eps)\right) \\
    \label{eq:tv2}
    &\le \sum_{k\in\ivalp}\left|\left\llbracket\partial_x\sqrt{\hatu_\Delta^n}\right\rrbracket_{x_k}\right|
    + \sum_{\kappa\in\hval}\left|\left\llbracket\partial_x\sqrt{\hatu_\Delta^n}\right\rrbracket_{x_\kappa}\right|.
  \end{align}
  Summarizing \eqref{eq:tv1} and \eqref{eq:tv2}, we obtain the estimate
  \begin{align}
    \label{eq:tvnow}
    \tv{\partial_x\sqrt{\hatu_\Delta^n}}
    \le 2 \sum_{k\in\ivalp}\left|\left\llbracket\partial_x\sqrt{\hatu_\Delta^n}\right\rrbracket_{x_k}\right|
    + 2 \sum_{\kappa\in\hval}\left|\left\llbracket\partial_x\sqrt{\hatu_\Delta^n}\right\rrbracket_{x_\kappa}\right|,
  \end{align}
  In view of \eqref{eq:tildeux}, we have that
  \begin{align*}
    \left\llbracket\partial_x\sqrt{\hatu_\Delta^n}\right\rrbracket_{x_k} 
    &= \frac1{2\sqrt{z_k}}\frac{(z_\kmh-z_\kph)^2}\delta
    % OLD RESULT (WRONG?!?) \frac12\frac{(z_\kmh-z_\kph)\big(\sqrt{z_\kph}-\sqrt{z_\kmh}\big)}\delta
    \quad\text{for $k\in\ivalp$}, \\
    \left\llbracket\partial_x\sqrt{\hatu_\Delta^n}\right\rrbracket_{x_\kappa} 
    &= \frac12 \sqrt{z_\kappa}\, \frac{z_{\kappa+1}-2z_\kappa+z_{\kappa-1}}\delta
    \quad\text{for $\kappa\in\hval$}.
  \end{align*}
  Accordingly, using that $1/z_k\le(1/z_\kph+1/z_\kmh)/2$ by the arithmetic-harmonic mean inequality,
  \begin{equation}
    \label{eq:sqrtest1}
    \begin{split}
      \sum_{k\in\ivalp}\left|\left\llbracket\partial_x\sqrt{\hatu_\Delta^n}\right\rrbracket_{x_k}\right|
      &= \frac\delta2\sum_{k\in\ivalp}\frac{(z_\kmh-z_\kph)^2}{\delta^2}\cdot\frac1{\sqrt{z_k}} \\
      &\le \frac12\left(\delta\sum_{k\in\ivalp}\left[\frac{z_\kmh-z_\kph}{\delta}\right]^4\right)^{1/2}
      \bigg(\sum_{k\in\ivalp}\frac\delta{z_k}\bigg)^{1/2} \\
      &= \frac12\big\|\partial_\xi\hatz_\Delta^n\big\|_{L^4(\Omega)}^2(b-a)^{1/2},
    \end{split}
  \end{equation}
  and also
  \begin{equation}
    \label{eq:sqrtest2}
    \begin{split}
      \sum_{\kappa\in\hval}\left|\left\llbracket\partial_x\sqrt{\hatu_\Delta^n}\right\rrbracket_{x_\kappa}\right|
      &= \frac\delta2 \sum_{\kappa\in\hval} z_\kappa\, \left|\frac{z_{\kappa+1}-2z_\kappa+z_{\kappa-1}}{\delta^2} \right|\cdot\frac1{\sqrt{z_\kappa}}\\
      &\le \frac12\left(\delta\sum_{\kappa\in\hval}z_\kappa^2\left[\frac{z_{\kappa+1}-2z_\kappa+z_{\kappa-1}}{\delta^2}\right]^2\right)^{1/2}(b-a)^{1/2}. 
    \end{split}
    % \bigg(\sum_{\kappa\in\hval}\frac{\delta}{z_\kappa}\bigg)^{1/2} \\
    % &\le \frac12\big(\grad\Fz(\xvec_\Delta^n)\Wmat^{-1}\grad\HFz(\xvec_\Delta^n)\big)^{1/2}\, (b-a)^{1/2}.
  \end{equation}
  Combine \eqref{eq:sqrtest1} with the $L^4$ bound from \eqref{eq:L4bound}, 
  and \eqref{eq:sqrtest2} with the entropy dissipation inequality \eqref{eq:dFdH_bounded}.
  Inserting this into \eqref{eq:tvnow} to obtain the claim \eqref{eq:BVbound}.
\end{proof}

\subsection{Convergence of time interpolants}
Recall that we require the a priori bound \eqref{eq:Hbound} on the initial entropy, 
but only \eqref{eq:Fbound} on the initial Fisher information. 
This estimate improves over time.
\begin{lem}
  \label{lem:Fbound}
  One has, for every $N\ge1$,
  \begin{align}
    \label{eq:Fboundfine}
    \Fz(\xvec_\Delta^N)\le \frac32(M\olH)^{1/2}(N\tau)^{-1/2}.
  \end{align}
  Consequently, $\ti{\Fz(\xvec_\Delta)}(t)$ is bounded for each $t>0$, uniformly in $\Delta$.
\end{lem}
\begin{proof}
  Since $\Fz(\xvec_\Delta^n)$ is monotonically decreasing in $n$ (for fixed $\Delta$),
  it follows that
  \begin{align*}
    \Fz(\xvec_\Delta^N) \le \frac1N\sum_{n=1}^N\Fz(\xvec_\Delta^n)
    &=\frac1{2N}\sum_{n=1}^N\delta\sum_{k\in\ivalp}\left(\frac{z^n_\kph-z^n_\kmh}\delta\right)^2 \\
    &\le \frac1{2N\tau}\left(\tau\sum_{n=1}^N\delta\sum_{k\in\ivalp}1\right)^{1/2}
    \left(\tau\sum_{n=1}^\infty\delta\sum_{k\in\ivalp}\left(\frac{z^n_\kph-z^n_\kmh}\delta\right)^4\right)^{1/2} \\
    &\le \frac1{2N\tau}(N\tau M)^{1/2}(9\olH)^{1/2} = \frac32(M\olH)^{1/2}(N\tau)^{-1/2},
  \end{align*}
  as desired.
\end{proof}
In the following, we use the notation $\tint\Subset\setR_+$ to denote time intervals with $0<\tlo<\thi<\infty$.
\begin{lem}
  \label{lem:timeapriori}
  We have that, for each $\tint\Subset\setR_+$,
  \begin{align}
    \label{eq:help008}
    \sup_\Delta\sup_{t\in\tint}\|\ti{\hatu_\Delta}(t)\|_{H^1(\Omega)}<\infty,
  \end{align}
  and that, as $\Delta\to0$,
  \begin{align}
    \label{eq:help007}
    \sup_{t\in\setR_+}\|\ti{\hatu_\Delta}(t)-\ti{\baru_\Delta}(t)\|_{L^\infty(\Omega)}\to 0.
  \end{align}
\end{lem}
\begin{proof}
  For each $n\in\setN$,
  \begin{align*}
    \big\|\partial_x\hatu_\Delta^n\big\|_{L^2(\Omega)}^2
    &= \sum_{k\in\ivalp}
    \bigg[(x^n_\kph-x^n_k)\Big(\frac{z^n_\kph-z^n_k}{x^n_\kph-x^n_k}\Big)^2 
    +(x^n_k-x^n_\kmh) \Big(\frac{z^n_k-z^n_\kmh}{x^n_k-x^n_\kmh}\Big)^2\bigg] \\
    &\le \delta\sum_{k\in\ivalp}\frac{z^n_\kph+z^n_\kmh}2\Big(\frac{z^n_\kph-z^n_\kmh}\delta\Big)^2
    \le\Fz(\xvec_\Delta^n)\max_{\kappa\in\hval}z^n_\kappa.
  \end{align*}
  Now combine this with the estimates \eqref{eq:Fboundfine} from above and \eqref{eq:Festimate} from the appendix
  to obtain \eqref{eq:help008}.
  Estimate \eqref{eq:help007} follows directly from the elementary observation that
  \begin{align*}
    \sup_{x\in\Omega}|\baru_\Delta^n(x)-\hatu_\Delta^n(x)|^2
    \le\max_{k\in\ivalp}\big|z^n_\kph - z^n_\kmh\big|^2\le\delta\Fz(\xvec_\Delta^n)\le\delta\Fz(\xvec_\Delta^0),
  \end{align*}
  and an application of \eqref{eq:Fbound}.
\end{proof}
\begin{prp}
  \label{prp:convergence1}
  There exists a function $u_*:\setRnn\times\Omega\to\setRnn$
  with
  \begin{align}\label{eq:reg1}
    u_*\in C^{1/2}_\loc(\setR_+;\dens)\cap L^\infty_\loc(\setR_+;H^1(\Omega)),
  \end{align}
  and there exists a subsequence of $\Delta$ (still denoted by $\Delta$), 
  such that, for every $\tint\Subset\setR_+$, the following are true:
  \begin{align}
    \label{eq:weak}
    \ti{u_\Delta}(t) &\longrightarrow u_*(t) \quad\text{in $\dens$, uniformly with respect to $t\in\tint$}, \\
    \label{eq:uniformLinfty}
    \ti{u_\Delta},\ti{\hatu_\Delta} &\longrightarrow u_* \quad\text{uniformly on $\tint\times\Omega$}, \\
    \label{eq:XuniformL2}
    \ti{\theX_\Delta}(t) &\longrightarrow \theX^*(t) \quad\text{in $L^2([0,M])$, uniformly with respect to $t\in\tint$},
  \end{align}
  where $\theX^*\in C^{1/2}(\setR_+;L^2([0,M]))$ is the Lagrangian map of $u_*$.
\end{prp}
\begin{proof}
  Fix $\tlo>0$.
  From the discrete energy inequality \eqref{eq:uniform_time},
  the bound on the Fisher information in Lemma \ref{lem:Fbound},
  and the equivalence \eqref{eq:metricequivalent} of $\wassN$ with the usual $L^2$-Wasserstein metric $\wass$,
  it follows by elementary considerations that
  \begin{align}
    \label{eq:preholder}
    \wass\big(\ti{\baru_\Delta}(t),\ti{\baru_\Delta}(s)\big)^2 \le C(\tlo)\big(|t-s|+\tau\big),
  \end{align}
  for all $t,s\ge\tlo$.
  Moreover, since $\Omega=[a,b]$ is compact, also $\dens$ is compact.
  Hence the generalized version of the Arzela-Ascoli theorem from \cite[Proposition 3.3.1]{AGS} is applicable 
  and yields the convergence of a subsequence of $(\ti{\baru_\Delta})$ to a limit $u_\tlo$ in $\dens$, 
  locally uniformly with respect to $t\in[\tlo,\infty)$.
  The H\"older-type estimate \eqref{eq:preholder} implies $u_\tlo\in C^{1/2}([\tlo,\infty);\dens)$.
  The claim \eqref{eq:XuniformL2} is a consequence of the equivalence 
  between the Wasserstein metric on $\dens$ and the $L^2$-metric on $\xspc$, see Lemma \ref{lem:flat}.

  Clearly, the previous argument applies to every choice of $\tlo>0$.
  Using a diagonal argument, one constructs a limit $u_*$ defined on all $\setR_+$,
  such that $u_\tlo$ is the restriction of $u_*$ to $[\tlo,\infty)$.
  
  For the rest of the proof, let some $\tint\Subset\setR_+$ be fixed.

  For proving \eqref{eq:uniformLinfty}, it suffices to show that $\ti{\hatu_\Delta}\to u_*$ uniformly on $\tint\times\Omega$:
  indeed, \eqref{eq:help007} implies that if $\ti{\hatu_\Delta}$ converges uniformly to some limit, so does $\ti{\baru_\Delta}$.
  As an intermediate step towards proving uniform convergence of $\ti{\hatu_\Delta}$,
  we show that
  \begin{align}
    \label{eq:unifL2}
    \hatu_\Delta(t)\longrightarrow u_*(t) \quad\text{in $L^2(\Omega)$, uniformly in $t\in\tint$}.
  \end{align}
  For $t\in\tint$, we expand the $L^2$-norm as follows:
  \begin{align*}
    \|\ti{\hatu_\Delta}(t)-u_*(t)\|_{L^2(\Omega)}^2
    % &= \intom \big(\hatu_\Delta^n(x)-u^*(t,x)\big) \hatu_\Delta^n(x) \dd x 
    % - \int_\Omega \big(\hatu_\Delta^n(x)-u^*(t,x)\big) u^*(t,x) \dd x \\
    &= \intom \Big[\big(\ti{\hatu_\Delta}-u_*\big) \ti{\baru_\Delta}\Big](t;x) \dd x \\
    &\qquad + \intom \Big[\big(\ti{\hatu_\Delta}-u_*\big) \big(\ti{\hatu_\Delta} - \ti{\baru_\Delta}\big)\Big](t;x) \dd x \\
    &\qquad - \intom \Big[\big(\ti{\hatu_\Delta}-u_*\big) u_*\Big](t;x) \dd x.
    % &= \int_0^M \big(\hatu_\Delta^n(\cX_\Delta^n)-u^*(t,\cX_\Delta^n)\big) \dd\xi - \int_0^M \big(\hatu_\Delta^n(\theX^*(t))-u^*(t,\theX^*(t))\big)\dd\xi \\
    % &\qquad + \int_\Omega \big(\hatu_\Delta^n(x)-u^*(t,x)\big) \big(\hatu_\Delta^n(x) - u_\Delta^n(x)\big) \dd x,\\
    % & \leq \int_0^M \int_{\theX^*(t,\xi)}^{\cX_\Delta^n(\xi)} \big|\big(\hatu_\Delta^n\big)_x\big|(y)\dd y \dd\xi 
    % + \int_0^M \int_{\theX^*(t,\xi)}^{\cX_\Delta^n(\xi)} |u_x^*|(y)\dd y \dd\xi \\
    % &\qquad + \int_\Omega \big(\hatu_\Delta^n(x)-u^*(t,x)\big) \big(\hatu_\Delta^n(x) - u_\Delta^n(x)\big) \dd x.
  \end{align*}
  On the one hand, observe that
  \begin{align*}
    &\sup_{t\in\tint}\intom\Big[\big(\ti{\hatu_\Delta}-u_*\big) \big(\ti{\hatu_\Delta} - \ti{\baru_\Delta}\big)\Big](t;x) \dd x \\
    &\leq\sup_{t\in\tint}\left(\big(\|\ti{\hatu_\Delta}(t)\|_{L^1(\Omega)}
      +\|u_*(t)\|_{L^1(\Omega)}\big)\|\ti{\hatu_\Delta}(t) - \ti{\baru_\Delta}(t)\|_{L^\infty}\right) \\
    &\le \sup_{t\in\tint}\left(\left(2M+(b-a)\|\ti{\hatu_\Delta}(t)-\ti{\baru_\Delta}(t)\|_{L^\infty}\right)
    \|\ti{\hatu_\Delta}(t) - \ti{\baru_\Delta}(t)\|_{L^\infty}\right),
  \end{align*}
  which converges to zero as $\Delta\to0$, using both conclusions from Lemma \ref{lem:timeapriori}.
  On the other hand, we can use property \eqref{eq:pushforward} to write
  \begin{align*}
    &\intom \Big[\big(\ti{\hatu_\Delta}-u_*\big) \ti{\baru_\Delta}\Big](t;x) \dd x 
    - \intom \Big[\big(\ti{\hatu_\Delta}-u_*\big) u_*\Big](t;x) \dd x \\
    &= \int_0^M \Big[\ti{\hatu_\Delta}-u_*\Big]\big(t;\ti{\theX_\Delta}(t;x)\big)\dd\xi
    - \int_0^M \Big[\ti{\hatu_\Delta}-u_*\Big]\big(t;\theX^*(t;\xi)\big)\dd\xi.
  \end{align*}
  We regroup terms under the integrals and use the triangle inequality.
  For the first term, we obtain
  \begin{align*}    
    &\sup_{t\in\tint} \left|\int_0^M \left(\ti{\hatu_\Delta}\big(t;\ti{\theX_\Delta}(t;\xi)\big)
      -\ti{\hatu_\Delta}\big(t;\theX_*(t;\xi)\big)\right) \dd\xi \right|\\
    &\le \sup_{t\in\tint}\int_0^M \int_{\theX_*(t,\xi)}^{\ti{\theX_\Delta}(t;\xi)} \left|\partial_x\ti{\hatu_\Delta}\right|(t;y)\dd y \dd\xi \\
    &\le \sup_{t\in\tint}\int_0^M \|\ti{\hatu_\Delta}\|_{H^1(\Omega)} |\theX_*-\ti{\theX_\Delta}|(t,\xi)^{1/2} \dd\xi\\
    &\le \sup_{t\in\tint}\Big(\|\ti{\hatu_\Delta}(t)\|_{H^1(\Omega)} \|\theX_*(t)-\ti{\theX_\Delta}(t)\|_{L^2([0,M])}^{1/4}\Big).
  \end{align*}
  A similar reasoning applies to the integral involving $u_*$ in place of $\ti{\hatu_\Delta}$.
  Together, this proves \eqref{eq:unifL2},
  and it further proves that $u_*\in L^\infty(\setR_+\tint;H^1(\Omega))$, 
  since the uniform bound on $\hatu_\Delta$ from \ref{eq:help008} is inherited by the limit.

  Now the Gagliardo-Nirenberg inequality \eqref{eq:GN} provides the estimate
  \begin{align}\label{eq:unif1}
    \|\ti{\hatu_\Delta}(t)-u^*(t)\|_{C^{1/6}(\Omega)}
    \leq C \|\ti{\hatu_\Delta}(t)-u^*(t)\|_{H^1(\Omega)}^{2/3} \|\ti{\hatu_\Delta}(t)-u^*(t)\|_{L^2(\Omega)}^{1/3}.
  \end{align}
  Combining the convergence in $L^2(\Omega)$ by \eqref{eq:unifL2} with the boundedness in $H^1(\Omega)$ from \eqref{eq:help008},
  it readily follows that $\hatu_\Delta(t)\to u_*(t)$ in $C^{1/6}(\Omega)$, uniformly in $t\in\tint$.
  This clearly implies that $\ti{\hatu_\Delta}\to u_*$ uniformly on $\tint\times\Omega$.
\end{proof}
\begin{prp}
  \label{prp:convergence2}
  Under the hypotheses and with the notations of Proposition \ref{prp:convergence1},
  we have that $\sqrt{u_*}\in L^2(\setRnn;H^1(\Omega))$, and
  \begin{align}
    \label{eq:strong}
    \ti{\sqrt{\hatu_\Delta}}\to\sqrt{u_*} \quad \text{strongly in $L^2_\loc(\setR_+;H^1(\Omega))$}
  \end{align}
  as $\Delta\to0$.
\end{prp}
Notice that $\partial_x\sqrt{u_*}\in L^2([0,\thi]\times\Omega)$ for each $\thi>0$,
but strong convergence takes place only on each $\tint\times\Omega$.
\begin{proof}
  Fix $\tint\Subset\setR_+$.
  By definition \eqref{eq:TVorig} of the total variation,
  \begin{align*}
    \|\partial_xf\|_{L^2(\Omega)}^2 = \intom \partial_xf(x)^2\dd x \le \left(\sup_{x\in\Omega} |f(x)|\right)\tv{\partial_xf}
  \end{align*}
  holds for every Lipschitz function $f:\Omega\to\setR$.
  The functions $\sqrt{\hatu_\Delta^n}$ are obviously Lipschitz continuous.
  Moreover, thanks to weak lower semi-continuity of the total variation, it follows from \eqref{eq:BVbound} that
  \begin{align*}
    \int_0^\infty \tv{\partial_x\sqrt{u_*}}^2\dd t \le 10(b-a)\olH.
  \end{align*}
  In particular, the weak derivative $x\mapsto\partial_x\sqrt{u_*}(t;x)$ is in $L^\infty(\Omega)$ 
  --- and thus $x\mapsto\sqrt{u_*}(t;x)$ is Lipschitz --- for almost every $t>0$.
  Using that $\tv{f-g}\le\tv{f}+\tv{g}$, we obtain that
  \begin{align*}
    &\int_\tlo^\thi\left\|\partial_x\left(\ti{\sqrt{\hatu_\Delta}}-\sqrt{u_*}\right)\right\|_{L^2(\Omega)}^2\dd t \\
    &\le (\thi-\tlo)^{1/2}\left(\sup_{\tint\times\Omega} \left|\ti{\sqrt{\hatu_\Delta}}-\sqrt{u_*}\right|\right)
    \left(2\int_0^\infty \left(\tv{\partial_x\ti{\sqrt{u_\Delta}}}^2+\tv{\partial_x\sqrt{u_*}}^2\right)\dd t\right)^{1/2}.
  \end{align*}
  For $\Delta\to0$,
  the first term on the right-hand side converges to zero by \eqref{eq:uniformLinfty},
  and the second term remains bounded by \eqref{eq:BVbound}.
  This proves \eqref{eq:strong}.
  
  To show square integrability of the limit, fix some $T>0$.
  Below, $N$ is always such that $T<N\tau<T+1$.
  A direct calculation yields
  that
  \begin{align*}
    4\intom \left(\partial_x\sqrt{\hatu_\Delta^n}\right)^2(x)\dd x 
    = \int_0^M \frac{\partial_\xi\hatz_\Delta^n(\xi)^2}{\hatz_\Delta^n(\xi)\partial_\xi\theX_\Delta^n(\xi)}\dd\xi.
  \end{align*}
  From the properties of $\theX_\Delta^n$ and $\hatz_\Delta^n$ as linear interpolations, 
  one easily deduces that
  \begin{align*}
    \frac1{\hatz_\Delta^n(\xi)\partial_\xi\theX_\Delta^n(\xi)} \le \frac{z^n_\kph}{z^n_\kmh}+\frac{z^n_\kmh}{z^n_\kph}
  \end{align*}
  for all $\xi\in(\xi_\kmh,\xi_\kph)$.
  Therefore,
  \begin{align*}
    &4\int_0^T \left(\intom \ti{\sqrt{\hatu_\Delta}}^2(t;x)\dd x\right)\dd t \\
    &\le \tau\sum_{n=1}^N\delta\sum_{k\in\ivalp}\left(\frac{z^n_\kph-z^n_\kmh}\delta\right)^2\left(\frac{z^n_\kph}{z^n_\kmh}+\frac{z^n_\kmh}{z^n_\kph}\right) \\
    &\le 2\left(\tau\sum_{n=1}^N\delta\sum_{k\in\ivalp}\left[\frac{z^n_\kph-z^n_\kmh}\delta\right]^4\right)^{1/2}
    \left(\tau\sum_{n=1}^N\delta\sum_{k\in\ivalp}\left[\left(\frac{z^n_\kph}{z^n_\kmh}\right)^2+\left(\frac{z^n_\kmh}{z^n_\kph}\right)^2\right]\right)^{1/2}.
  \end{align*}
  The two sums are $\Delta$-uniformly bounded, thanks to the estimates \eqref{eq:L4bound} and \eqref{eq:oscillation}.
  By lower semi-continuity of norms, $\sqrt{u_*}$ obeys the same bound.
\end{proof}

%%%%%%%%%%%%%%%%%%%%%%%%%%%%%%%%%%%%%%%%%%%%%%%%%%%%%%%%%%%%%%%%%%%%%%%%%%%%%%%%%%%%%%%%%%%%%%%%%%%
%
\section{Weak formulation of the limit equation}
\label{sec:weak}
%
%%%%%%%%%%%%%%%%%%%%%%%%%%%%%%%%%%%%%%%%%%%%%%%%%%%%%%%%%%%%%%%%%%%%%%%%%%%%%%%%%%%%%%%%%%%%%%%%%%%
%
To finish our discussion of convergence, we verify that the limit $u_*$ obtained in the previous section is indeed a weak solution to \eqref{eq:dlss}.
From now on, $(\xvec_\Delta^n)_{n=0}^\infty$ with its derived functions $\baru_\Delta$, $\hatu_\Delta$, $\theX_\Delta$
is a (sub)sequence for which the convergence results stated in Proposition \ref{prp:convergence1} and Proposition \ref{prp:convergence2} holds.
We continue to assume \eqref{eq:Hbound} and \eqref{eq:Fbound}.
The goal of this section is to prove the following.
\begin{prp}
  \label{prp:weakgoal}
  For every $\rho\in C^\infty(\Omega)$ with $\rho'(a)=\rho'(b)=0$,
  and for every $\psi\in C^\infty_c(\setRnn)$,
  \begin{align}
    \label{eq:weakgoal}
    \begin{split}
      &\int_0^\infty \psi'(t)\left(\intom\rho(x)u_*(t;x)\dd x\right)\dd t + \psi(0)\intom \rho(x)u^0(x)\dd x \\
      &\qquad + \int_0^\infty\psi(t)\left(\intom\big[\rho'''(x)\partial_xu_*(t;x) + 4\rho''(x)\partial_x\sqrt{u_*}(t;x)^2\big]\dd x\right)\dd t = 0.
    \end{split}
  \end{align}
\end{prp}
For definiteness, fix a spatial test function $\rho\in C^\infty(\Omega)$ with $\rho'(a)=\rho'(b)=0$, 
and a temporal test function $\psi\in C^\infty_c(\setRnn)$ with $\operatorname{supp}\psi\subset[0,T)$ for a suitable $T>0$.
Let $B>0$ be chosen such that
\begin{align}
  \label{eq:smoothbound}
  \|\rho\|_{C^4(\Omega)}\le B, \quad \|\psi\|_{C^1(\setR_+)}\le B.
\end{align}
For convenience, we assume $\delta<1$ and $\tau<1$.
Further, we introduce the short-hand notation 
\begin{align}
  \label{eq:shorthand}
  \rho'(\xvec_\Delta^n)=\big(\rho'(x^n_1),\ldots,\rho'(x^n_{K-1})\big)\in\setR^{K-1}.
\end{align}
In the estimates that follow, the non-explicity constants possibly depend on $(b-a)$, $T$, $B$, and $\olH$,
but not $\Delta$.
The two main steps in the proof of Proposition \ref{prp:weakgoal} are to establish the following estimates, respectively:
\begin{align}
  \label{eq:weakform1}
  \begin{split}
    \err_{1,\Delta}:=\Bigg|\int_0^T \left(\psi'(t)\intom \rho(x)\ti{\baru_\Delta}(t;x) \dd x
      +\psi(t)\ti{\spr{\rho'(\xvec_\Delta)}{\wgrad\Fz(\xvec_\Delta)}}(t)\right)\dd t\\
    +\psi(0)\intom\rho(x)\baru_\Delta^0(x)\dd x\Bigg|
    \le C\big((\delta\Fz(\xvec_\Delta^0))^{1/2}+(\tau\Fz(\xvec_\Delta^0))\big),     
  \end{split}
\end{align}
and
\begin{align}
  \label{eq:weakform2}
  \begin{split}
    \err_{2,\Delta}:=\Bigg|\int_0^T \psi(t)\Bigg(\int_\Omega\big[\rho'''(x)\partial_x\ti{\hatu_\Delta}(t;x)+4\rho''(x)\partial_x\ti{\sqrt{\hatu_\Delta}}(t;x)^2\big] \dd x\\
    -\ti{\spr{\rho'(\xvec_\Delta^n)}{\wgrad\Fz(\xvec_\Delta^n)}}(t)\Bigg)\dd t\Bigg|
    \le C\delta^{1/4}.   
  \end{split}
\end{align}
We proceed by proving \eqref{eq:weakform1} and \eqref{eq:weakform2}.
At the end of this section, it is shown how the claim \eqref{eq:weakgoal} follows 
from \eqref{eq:weakform1}\&\eqref{eq:weakform2} on basis of the convergence for $\ti{\baru_\Delta}$ obtained previously.
% immediately via \eqref{eq:renew}.
%
\begin{proof}[Proof of \eqref{eq:weakform1}]
  Choose $N_\tau\in\setN$ such that $N_\tau\tau\in(T,T+1)$.
  Then, using that $\psi(N_\tau\tau)=0$,
  we obtain after ``summation by parts'':
  \begin{equation}
    \label{eq:dummy814}
    \begin{split}
      -\int_0^T&\psi'(t)\left(\int_\Omega \rho(x)\ti{\bar u_\Delta}(t;x) \dd x\right)\dd t 
      = -\sum_{m=1}^{N_\tau} \left(\int_{(m-1)\tau}^{m\tau}\psi'(t)\dd t \int_\Omega\rho(x)\bar u_\Delta^m(x)\dd x\right) \\
      &= -\tau\sum_{m=1}^{N_\tau}\left(\frac{\psi(m\tau)-\psi((m-1)\tau)}\tau\,\int_0^M\rho\circ\theX_\Delta^m(\xi)\dd\xi\right) \\
      &= \tau\sum_{n=1}^{N_\tau}\left(\psi((n-1)\tau)\,\int_0^M\frac{\rho\circ\theX_\Delta^n-\rho\circ\theX_\Delta^{n-1}}\tau(\xi)\dd\xi\right) 
      + \psi(0)\int_0^M\rho\circ\theX_\Delta^0(\xi)\dd\xi.
    \end{split}
  \end{equation}
  A Taylor expansion of the term in the inner integral yields
  \begin{align}
    \label{eq:dummy101}
    \frac{\rho\circ\theX_\Delta^n-\rho\circ\theX_\Delta^{n-1}}\tau
    &= \rho'\circ\theX_\Delta^n\,\left(\frac{\theX_\Delta^n-\theX_\Delta^{n-1}}\tau\right)
    + \frac\tau2\rho''\circ\widetilde{\theX}\,\left(\frac{\theX_\Delta^n-\theX_\Delta^{n-1}}\tau\right)^2.
  \end{align}
  where $\widetilde{\theX}$ symbolizes suitable ``intermediate values'' in $[0,M]$.
  We analyze the first term on the right-hand side of \eqref{eq:dummy101}:
  using the representation \eqref{eq:cX} of $\theX_\Delta$ in terms of hat functions $\hatf_k$,
  we can write its integral as follows,
  \begin{align}
    \label{eq:dummy102}
    \int_0^M \rho'\circ\theX_\Delta^n\,\left(\frac{\theX_\Delta^n-\theX_\Delta^{n-1}}\tau\right)\dd\xi
    = \sum_{k\in\ivalp}\left(\frac{x_k^n-x^{n-1}_k}\tau\right)\int_{\xi_{k-1}}^{\xi_{k+1}} \rho'\circ\theX_\Delta^n\hatf_k\dd\xi.
  \end{align}  
  On the other hand, since
  \begin{align}
    \label{eq:dummy808}
    \int_{\xi_{k-1}}^{\xi_{k+1}}\hatf_k\dd\xi = \delta,
  \end{align}
  the discrete evolution equation \eqref{eq:euler} yields that
  \begin{align}
    \label{eq:dummy103}
    -\spr{\rho'(\xvec_\Delta^n)}{\wgrad\Fz(\xvec_\Delta^n)}
    = \spr{\rho'(\xvec_\Delta^n)}{\frac{\xvec_\Delta^n-\xvec_\Delta^{n-1}}\tau}
    % =\delta\sum_{k\in\ivalp}\rho'(x^n_k)\left(\frac{x^n_k-x^{n-1}_k}\tau\right)
    =\sum_{k\in\ivalp}\left(\frac{x^n_k-x^{n-1}_k}\tau\right)\int_{\xi_{k-1}}^{\xi_{k+1}}\rho(x^n_k)\hatf_k(\xi)\dd\xi.
    % &=\sum_{k\in\ivalp} \int_{\xi_{k-1}}^{\xi_{k+1}} \rho'(x^n_k)\frac{x^n_k-x^{n-1}_k}\tau\hatf_k(\xi)\dd\xi \\
    % &=\sum_{k\in\ivalp} \int_{\xi_{k-1}}^{\xi_{k+1}} \rho'(x^n_k)\left(\frac{\theX_\Delta^n-\theX_\Delta^{n-1}}\tau\right) (\xi)\dd\xi.
  \end{align}
  Finally, observing that
  \begin{align*}
    |\theX_\Delta^n(\xi)-x^n_k|\le(x^n_{k+1}-x^n_{k-1}) \quad\text{for each $\xi\in(\xi_{k-1},\xi_{k+1})$},
  \end{align*}
  we can estimate the difference of the terms in \eqref{eq:dummy102} and \eqref{eq:dummy103}
  with the help of the bound \eqref{eq:smoothbound} on $\rho$ as follows:
  \begin{equation}
    \label{eq:dummy815}
    \begin{split}
      &\left|\int_0^M\rho'\circ\theX_\Delta^n(\xi)\,\left(\frac{\theX_\Delta^n-\theX_\Delta^{n-1}}\tau\right)(\xi)\dd\xi 
        - \spr{\rho'(\xvec_\Delta^n)}{\wgrad\Fz(\xvec_\Delta^n)}\right| \\
      % &\le \sum_{k\in\ivalp}\left|\int_{\xi_{k-1}}^{\xi_{k+1}}\rho'\circ\theX_\Delta^n(\xi) \left(\frac{x^n_k-x^{n-1}_k}\tau\right)\hatf_k(\xi)\dd\xi - \delta\rho'(x^n_k) \left(\frac{x^n_k-x^{n-1}_k}\tau\right)\right| \\
      &\le \sum_{k\in\ivalp}\left|\frac{x^n_k-x^{n-1}_k}\tau\right|
      \int_{\xi_{k-1}}^{\xi_{k+1}}\big|\rho'\circ\theX_\Delta^n(\xi)-\rho'(x^n_k)\big|\hatf_k(\xi)\dd\xi \\
      &\le B\delta\sum_{k\in\ivalp}\left|\frac{x^n_k-x^{n-1}_k}\tau\right|(x^n_{k+1}-x^n_{k-1}).
    \end{split}
  \end{equation}
  As a final preparation for the proof of \eqref{eq:weakform1},
  observe that
  \begin{align*}
    R':=&\left|\int_0^T\psi(t)\ti{\spr{\rho'(\xvec_\Delta)}{\wgrad\Fz(\xvec_\Delta)}}(t)\dd t 
      - \tau\sum_{n=1}^{N_\tau}\psi((n-1)\tau)\spr{\rho'(\xvec_\Delta^n)}{\wgrad\Fz(\xvec_\Delta^n)}\right|\\
    &\le \left(\tau\sum_{n=1}^{N_\tau}\left|\frac1\tau\int_{(n-1)\tau}^{n\tau}\psi(t)\dd t-\psi((n-1)\tau)\right|^2\right)^{1/2}
    \left(\tau\sum_{n=1}^\infty B^2\nrm{\wgrad\Fz(\xvec_\Delta^n)}^2\right)^{1/2}\\
    &\le \big((T+1)B^2\tau^2\big)^{1/2}(2B^2\Fz(\xvec_\Delta^0))^{1/2} = C'\Fz(\xvec_\Delta^0)^{1/2}\tau,
  \end{align*}
  using the energy estimate \eqref{eq:eee}.
  We are now ready to estimate $\err_{1,\Delta}$ in \eqref{eq:weakform1}:
  \begin{align*}
    \err_{1,\Delta}
    &\stackrel{\eqref{eq:dummy814}}{\le} R' + \tau\sum_{n=1}^{N_\tau}\left(\big|\psi((n-1)\tau)\big|
      \,\left|\int_0^M\frac{\rho\circ\theX_\Delta^n-\rho\circ\theX_\Delta^{n-1}}\tau(\xi)\dd\xi
        -\spr{\rho'(\xvec_\Delta^n)}{\wgrad\Fz(\xvec_\Delta^n)}\right|\right)
    \\
    &\stackrel{\eqref{eq:dummy101}}{\le} 
    R'+ B\tau\sum_{n=1}^{N_\tau}\Bigg(\left|\int_0^M\rho'\circ\theX_\Delta^n(\xi)\,\left(\frac{\theX_\Delta^n-\theX_\Delta^{n-1}}\tau\right)(\xi)\dd\xi 
      - \spr{\rho'(\xvec_\Delta^n)}{\wgrad\Fz(\xvec_\Delta^n)}\right| \\
    &\qquad +\frac{B\tau}2\int_0^M \left(\frac{\theX_\Delta^n-\theX_\Delta^{n-1}}\tau\right)^2(\xi)\dd\xi\Bigg)\\
    &\stackrel{\eqref{eq:dummy815}}{\le} R' + B^2\left(\tau\sum_{n=1}^\infty\delta\sum_{k\in\ivalp}\left(\frac{x^n_k-x^{n-1}_k}\tau\right)^2\right)^{1/2}
    \left(\tau\sum_{n=1}^{N_\tau}\delta\sum_{k\in\ivalp}(x^n_{k+1}-x^n_{k-1})^2\right)^{1/2} \\
    & + \frac{B^2\tau}2\tau\sum_{n=1}^\infty\left\|\frac{\theX_\Delta^n-\theX_\Delta^{n-1}}\tau\right\|_{L^2([0,M])}^2\\
    &\le C'(\tau\Fz(\xvec_\Delta^0)) + B^2\big(2(b-a)^2T\big)^{1/2}(\delta\Fz(\xvec_\Delta^0))^{1/2} + B^2(\tau\Fz(\xvec_\Delta^0)),
  \end{align*}
  where we have used the energy estimate \eqref{eq:eee} and the the bound \eqref{eq:xpower}.
\end{proof}
%
% \begin{rmk}
%   Thinking about consistency of the scheme, one would rather expect $\err_{1,\Delta}\le(\tau+\delta)$.
%   The weaker error bound obtained above is due to the use of the very general estimate \eqref{eq:xpower} with $p=2$ in the last step of the proof.
%   If one had an a priori lower bound on $\hatz$, then this estimate would improve and yield the desired extra power of $\delta^{1/2}$.
% \end{rmk}
%
The proof of \eqref{eq:weakform2} requires more calculations, which are distributed in a series of lemmata below.
The first step is to derive a fully discrete weak formulation from \eqref{eq:euler}.
\begin{lem}
  With \eqref{eq:shorthand}, one has that
  \begin{align}
    \label{eq:fourAs}
    -\spr{\rho'(\xvec_\Delta^n)}{\wgrad\Fz(\xvec_\Delta^n)} = A^n_1-A^n_2+A^n_3+A^n_4,
  \end{align}
  where
  \begin{align*}
    A^n_1 &= \delta\sum_{k\in\ivalp}\left(\frac{z^n_\kph-z^n_\kmh}{\delta}\right)^2\left(\frac{z^n_\kph+z^n_\kmh}2\right)\left(\frac{\rho'(x^n_{k+1})-\rho'(x^n_{k-1})}{\delta}\right), \\
    A^n_2 &= \delta\sum_{k\in\ivalp}\left(\frac{z^n_\kph-z^n_\kmh}{\delta}\right)^2\left(\frac{(z^n_\kph)^2+(z^n_\kmh)^2}{2 z^n_\kph z^n_\kmh}\right)\rho''(x^n_k), \\
    A^n_3 &= \delta\sum_{k\in\ivalp}\left(\frac{z^n_\kph-z^n_\kmh}{\delta}\right)\left(\frac{(z^n_\kph)^2+(z^n_\kmh)^2}2\right)
    \left(\frac{\rho'(x^n_{k+1})-\rho'(x^n_k)-(x^n_{k+1}-x^n_k)\rho''(x^n_k)}{\delta^2}\right), \\
    A^n_4 &= \delta\sum_{k\in\ivalp}\left(\frac{z^n_\kph-z^n_\kmh}{\delta}\right)\left(\frac{(z^n_\kph)^2+(z^n_\kmh)^2}2\right)
    \left(\frac{\rho'(x^n_{k-1})-\rho'(x^n_{k})-(x^n_{k-1}-x^n_{k})\rho''(x^n_k)}{\delta^2}\right).
  \end{align*}
\end{lem}
\begin{proof}
  Fix some time index $n\in\setN$ (omitted in the calculations below).
  Recall the representation of $\wgrad\Fz$ from \eqref{eq:Fgrad}.
  By definition of $\rho'(\xvec_\Delta)$, it follows via a ``summation by parts'' that
  \begin{align*}
    &-\langle\nabla_{\xvec}\Fz(\xvec_\Delta),\rho'(\xvec_\Delta)\rangle_\delta
    =-\delta\sum_{\kappa\in\hval}\frac1\delta\left(\frac{z_{\kappa+1}-z_\kappa}\delta-\frac{z_\kappa-z_{\kappa-1}}\delta\right)z_\kappa^2\left(\frac{\rho'(x_\kappp)-\rho'(x_\kappm)}\delta\right) \\
    &=\delta\sum_{k\in\ivalp}\left(\frac{z_\kph-z_\kmh}\delta\right)\frac1\delta\left(z_\kph\frac{\rho'(x_{k+1})-\rho'(x_k)}\delta - z_\kmh\frac{\rho'(x_k)-\rho'(x_{k-1})}\delta \right).
  \end{align*}
  Using the elementary identity (for arbitrary numbers $\alpha_\pm$ and $\beta_\pm$)
  \begin{align*}
    \alpha_+\beta_+-\alpha_-\beta_- = \frac{\alpha_++\alpha_-}2(\beta_+-\beta_-) + (\alpha_+-\alpha_-)\frac{\beta_++\beta_-}2,
  \end{align*}
  we obtain further:
  \begin{align}
    \nonumber
    &-\langle\nabla_{\xvec}\Fz(\xvec_\Delta),\rho'(\xvec_\Delta)\rangle_\delta \\
    \label{eq:help1}
    &=  \delta\sum_{k\in\ivalp}\left(\frac{z_\kph-z_\kmh}{\delta}\right)\left(\frac{z_\kph^2-z_\kmh^2}{2\delta}\right)\left(\frac{\rho'(x_{k+1})-\rho'(x_{k-1})}{\delta}\right) \\
    \label{eq:help2}
    &+  \delta\sum_{k\in\ivalp}\left(\frac{z_\kph-z_\kmh}{\delta}\right)\left(\frac{z_\kph^2+z_\kmh^2}2\right)\left(\frac{\rho'(x_{k+1})-2\rho'(x_k)+\rho'(x_{k-1})}{\delta^2}\right).
  \end{align}
  The sum in \eqref{eq:help1} equals to $A_1^n$.
  In order to see that the sum in \eqref{eq:help2} equals to $-A_2^n+A_3^n+A_4^n$,
  simply observe that the identity
  \begin{align*}
    \frac{x_{k+1}-x_k}\delta + \frac{x_{k-1}-x_k}\delta = \frac1{z_\kph}-\frac1{z_\kmh}=-\frac{z_\kph-z_\kmh}{z_\kph z_\kmh},
  \end{align*}
  makes the coefficient of $\rho''(x^n_k)$ vanish.
\end{proof}
\begin{lem}
  \label{lem:R1}
  There is a constant $C_1>0$ 
  % --- expressible in $(b-a)$, $T$, $B$ and $\overline{\HF}$ ---
  such that for each $N$ with $N\tau<T$, 
  one has
  \begin{align*}
    R_1:=\tau\sum_{n=1}^N\bigg|A^n_1-2\int_0^M \partial_\xi\hatz^n_\Delta(\xi)^2\rho''\circ\theX^n_\Delta(\xi)\dd\xi\bigg|
    \le C_1\delta^{1/4}.
  \end{align*}
\end{lem}
\begin{proof}
  First, observe that by definition of $\hatz$,
  \begin{align*}
    \int_0^M \partial_\xi\hatz^n_\Delta(\xi)^2\rho''\circ\theX^n_\Delta(\xi)\dd\xi
    = \sum_{k\in\ivalp} \Big(\frac{z^n_\kph-z^n_\kmh}{\delta}\Big)^2 \int_{\xi_\kmh}^{\xi_\kph}\rho''\circ\theX^n_\Delta(\xi)\dd\xi,
  \end{align*}
  and therefore, by H\"older's inequality,
  \begin{align}
    \label{eq:R1}
    R_1 \le R_{1a}^{1/2}R_{1b}^{1/2},
  \end{align}
  with, recalling \eqref{eq:L4bound},
  \begin{align}
    \label{eq:R1a}
    R_{1a} &= \tau\sum_{n=1}^N\delta\sum_{k\in\ivalp}\Big(\frac{z^n_\kph-z^n_\kmh}{\delta}\Big)^4 
    \le \tau\sum_{n=1}^\infty\|\hatz_\Delta^n\|_{L^4(\Omega)}^4\le9\overline\HF, \\
    R_{1b} &= \tau\sum_{n=1}^N\delta\sum_{k\in\ivalp}\Big[\frac{z^n_\kph+z^n_\kmh}2\frac{\rho'(x^n_{k+1})-\rho'(x^n_{k-1})}{\delta}-\frac2\delta\int_{\xi_\kmh}^{\xi_\kph}\rho''\circ\theX^n_\Delta\dd\xi\Big]^2.
  \end{align}
  To simplify $R_{1b}$, let us fix $n$ (omitted in the following), 
  and introduce $\tilde x_k^+\in(x_k,x_{k+1})$ and $\tilde x_k^-\in(x_{k-1},x_k)$ such that
  \begin{align*}
    \frac{\rho'(x_{k+1})-\rho'(x_{k-1})}{\delta} 
    &= \frac{\rho'(x_{k+1})-\rho'(x_k)}\delta + \frac{\rho'(x_k)-\rho'(x_{k-1})}\delta \\
    &= \rho''(\tilde x_k^+)\frac{x_{k+1}-x_k}\delta + \rho''(\tilde x_k^+)\frac{x_{k+1}-x_k}\delta
    = \frac{\rho''(\tilde x_k^+)}{z_\kph} + \frac{\rho''(\tilde x_k^-)}{z_\kmh}.
  \end{align*}
  For each $k\in\ivalp$, we have that --- recalling \eqref{eq:dummy808} ---
  \begin{align*}
    &\frac{z_\kph+z_\kmh}2 \Big(\frac{\rho''(\tilde x_k^+)}{z_\kph} + \frac{\rho''(\tilde x_k^-)}{z_\kmh}\Big) 
    - \frac2\delta\int_{\xi_\kmh}^{\xi_\kph}\rho''\circ\theX_\Delta\dd\xi \\
    &= \frac12\Big[\Big(\frac{z_\kmh}{z_\kph}+1\Big) \rho''(\tilde x_k^+) + \Big(\frac{z_\kph}{z_\kmh}+1\Big) \rho''(\tilde x_k^-)\Big]
    - \frac2\delta\int_{\xi_\kmh}^{\xi_\kph}\rho''\circ\theX_\Delta\dd\xi \\
    &= \frac12\Big[\Big(\frac{z_\kmh}{z_\kph}-1\Big) \rho''(\tilde x_k^+) + \Big(\frac{z_\kph}{z_\kmh}-1\Big) \rho''(\tilde x_k^-)\Big] \\
    &\qquad - \frac2\delta\int_{\xi_k}^{\xi_\kph}\big[\rho''\circ\theX_\Delta-\rho''(\tilde x_k^+)\big]\dd\xi
    - \frac2\delta\int_{\xi_\kmh}^{\xi_k}\big[\rho''\circ\theX_\Delta-\rho''(\tilde x_k^-)\big]\dd\xi.
  \end{align*}
  Since $\theX_\Delta(\xi)\in[x_k,x_\kph]$ for each $\xi\in[\xi_k,\xi_\kph]$, and $\tilde x_k^+\in[x_k,x_{k+1}]$, 
  it follows that $|\theX_\Delta(\xi)-\tilde x_k^+|\le x_{k+1}-x_k$,
  and therefore
  \begin{align}
    \label{eq:rhotox}
    \frac2\delta\int_{\xi_k}^{\xi_\kph}\big|\rho''\circ\theX_\Delta(\xi)-\rho''(\tilde x_k^+)\big|\dd\xi
    \le B(x_{k+1}-x_k).
  \end{align}
  A similar estimate holds for the other integral.
  Thus
  \begin{align*}
    R_{1b} \le B^2\tau\sum_{n=1}^N\delta \sum_{k\in\ivalp}\Big[\Big(\frac{z^n_\kmh}{z^n_\kph}-1\Big)^2 + \Big(\frac{z^n_\kph}{z^n_\kmh}-1\Big)^2 + 2(x^n_{k+1}-x^n_{k-1})^2\Big].
  \end{align*}
  Recalling the estimates \eqref{eq:oscillation} and \eqref{eq:xpower},
  we further conclude that
  \begin{align}
    \label{eq:R1bbound}
    R_{1b} \le B^2\big(6(b-a)^2(\olH T\delta)^{1/2}+4T(b-a)^2\delta\big).
  \end{align}
  In combination with \eqref{eq:R1} and \eqref{eq:R1a},
  this proves the claim.
\end{proof}
\begin{lem}
  There is a constant $C_2>0$ 
  % --- expressible in $(b-a)$, $T$, $B$ and $\overline{\HF}$ ---
  such that for each $N$ with $N\tau<T$, 
  one has
  \begin{align*}
    R_2:=\tau\sum_{n=1}^N\bigg|A^n_2 -\int_0^M \partial_\xi\hatz^n_\Delta(\xi)^2\rho''\circ\theX^n_\Delta(\xi)\dd\xi\bigg|
    \le C_2\delta^{1/4}.
  \end{align*}
\end{lem}
\begin{proof}
  The proof is almost identical to (and even easier than) the one for Lemma \ref{lem:R1} above.
  Again, we have a decomposition of the form
  \begin{align*}
    R_2 \le R_{2a}^{1/2}R_{2b}^{1/2},
  \end{align*}
  where $R_{2a}$ equals $R_{1a}$ from \eqref{eq:R1a}, 
  and
  \begin{align*}
    R_{2b} &= \tau\sum_{n=1}^N\delta\sum_{k\in\ivalp}\Big[\frac{z^n_\kph+z^n_\kmh}{2 z^n_\kph z^n_\kmh}\rho''(x^n_k)-\frac1\delta\int_{\xi_\kmh}^{\xi_\kph}\rho''\circ\theX^n_\Delta\dd\xi\Big]^2.
  \end{align*}
  By writing
  \begin{align*}
    \frac{(z^n_\kph)^2+(z^n_\kmh)^2}{2 z^n_\kph z^n_\kmh} = \frac12\Big(\frac{z^n_\kmh}{z^n_\kph}-1\Big) + \frac12\Big(\frac{z^n_\kph}{z^n_\kmh}-1\Big) + 1,
  \end{align*}
  and observing --- in analogy to \eqref{eq:rhotox} --- that
  \begin{align*}
    \frac1\delta\int_{\xi_\kmh}^{\xi_\kph}\big|\rho''\circ\theX_\Delta(\xi)-\rho''(x_k)\big|\dd\xi
    \le B(x_\kph-x_\kmh),
  \end{align*}
  we obtain the same bound on $R_{2b}$ as the one on $R_{1b}$ from \eqref{eq:R1bbound}.
\end{proof}
\begin{lem}
  There is a constant $C_3>0$ 
  % --- expressible in $b-a$, $T$, $B$ and $\overline{\HF}$ ---
  such that for each $N$ with $N\tau\le T$, 
  one has
  \begin{align*}
    R_3:=\tau\sum_{n=1}^N\left|A^n_3-\frac12\int_0^M \partial_\xi\hatz^n_\Delta(\xi)\rho'''\circ\theX^n_\Delta(\xi)\dd\xi\right|
    \le C_3\delta^{1/4}.
  \end{align*}
\end{lem}
\begin{proof}
  Arguing like in the previous proofs, we first deduce 
  --- now by means of H\"older's inequality instead of the Cauchy-Schwarz inequality --- 
  that
  \begin{align*}
    R_3 \le R_{3a}^{1/4}R_{3b}^{3/4},
  \end{align*}
  where $R_{3a}=R_{1a}$, and
  \begin{align*}
    R_{3b} = \tau\sum_{n=1}^N \delta\sum_{k\in\ivalp}\Bigg|\left(\frac{(z^n_\kph)^2+(z^n_\kmh)^2}2\right)
    \left(\frac{\rho'(x^n_{k+1})-\rho'(x^n_k)-(x^n_{k+1}-x^n_k)\rho''(x^n_k)}{\delta^2}\right) \\
    -\frac1{2\delta}\int_{\xi_\kmh}^{\xi_\kph}\rho'''\circ\theX^n_\delta\dd\xi\Bigg|^{4/3}.
  \end{align*}
  % Consider the inner sum for fixed $n$ (omitted in the following).
  Introduce intermediate values $\tilde x_k^+$ such that
  \begin{align*}
    \rho'(x^n_{k+1})-\rho'(x^n_k)-(x^n_{k+1}-x^n_k)\rho''(x^n_k)
    =\frac12(x^n_{k+1}-x^n_k)^2\rho'''(\tilde x_k^+)
    =\frac{\delta^2}{2(z^n_\kph)^2}\rho'''(\tilde x_k^+).
  \end{align*}
  Thus we have that
  \begin{align*}
    &\left(\frac{(z^n_\kph)^2+(z^n_\kmh)^2}2\right)
    \left(\frac{\rho'(x^n_{k+1})-\rho'(x^n_k)-(x^n_{k+1}-x^n_k)\rho''(x^n_k)}{\delta^2}\right)-\frac1{2\delta}\int_{\xi_\kmh}^{\xi_\kph}\rho'''\circ\theX^n_\delta\dd\xi \\
    &= \frac14\left(\left(\frac{z^n_\kmh}{z^n_\kph}\right)^2+1\right)\rho'''(\tilde x_k^+) -\frac1{2\delta}\int_{\xi_\kmh}^{\xi_\kph}\rho'''\circ\theX^n_\delta\dd\xi \\
    &= \frac14\left(\frac{z^n_\kmh}{z^n_\kph}+1\right)\left(\frac{z^n_\kmh}{z^n_\kph}-1\right)\rho'''(\tilde x_k^+) -\frac1{2\delta}\int_{\xi_\kmh}^{\xi_\kph}\big[\rho'''\circ\theX^n_\Delta-\rho'''(\tilde x_k^+)\big]\dd\xi.
  \end{align*}
  By the analogue of \eqref{eq:rhotox}, it follows further that
  \begin{align*}
    R_{3b} &\le 2B^{4/3}\tau\sum_{n=1}^N \delta\sum_{k\in\ivalp}\left[\left(\frac{z^n_\kmh}{z^n_\kph}+1\right)^{4/3}\left(\frac{z^n_\kmh}{z^n_\kph}-1\right)^{4/3}+(x^n_{k+1}-x^n_{k-1})^{4/3}\right] \\
    &\le 2B^{4/3}\left(\tau\sum_{n=1}^N \delta\sum_{k\in\ivalp}\left(\frac{z^n_\kmh}{z^n_\kph}+1\right)^4\right)^{1/3}
    \left(\tau\sum_{n=1}^N\delta\sum_{k\in\ivalp}\left(\frac{z^n_\kmh}{z^n_\kph}-1\right)^2\right)^{2/3} \\
    & + 2B^{4/3}T(b-a)^{4/3}\delta,
  \end{align*}
  where we have used \eqref{eq:xpower}.
  At this point, the estimates \eqref{eq:weakoscillation} and \eqref{eq:oscillation} are used to control
  the first and the second sum, respectively.
\end{proof}
Along the same lines, one proves the analogous estimate for $A_4$ in place of $A_3$.

It remains to identify the integral expressions inside $R_1$ to $R_3$ with those in the weak formulation \eqref{eq:weakgoal}.
\begin{lem}
  One has that
  \begin{align}
    \label{eq:cov1}
    &\int_0^M \partial_\xi\hatz_\Delta^n(\xi)\rho'''\circ\theX^n_\Delta(\xi)\dd\xi 
    = \intom \partial_x\hatu_\Delta^n(x)\rho'''(x)\dd x, \\
    \label{eq:cov2}
    &R_5:=\tau\sum_{n=1}^N\left|\int_0^M \partial_\xi\hatz_\Delta^n(\xi)^2\rho''\circ\theX^n_\Delta(\xi)\dd\xi 
      - 4\intom \left(\partial_x\sqrt{\hatu_\Delta^n}\right)^2(x)\rho''(x)\dd x\right|
    \le C_5\delta^{1/4},
 \end{align}
 where \eqref{eq:cov2} holds for each $N$ with $N\tau\le T$.
\end{lem}
\begin{proof}
  The starting point is relation \eqref{eq:locaffine0}, that is 
  \begin{align}
    \label{eq:locaffine2}
    \hatz_\Delta^n(\xi)=\hatu_\Delta^n\circ\theX_\Delta^n(\xi)     
  \end{align}
  for all $\xi\in[0,M]$.
  Both sides of this equation are Lipschitz continuous in $\xi$, 
  and are differentiable except possibly at $\xi_{\frac12},\xi_1,\ldots,\xi_{K-\frac12}$.
  At points $\xi$ of differentiability, we have that
  \begin{align*}
    \partial_\xi\hatz_\Delta^n(\xi) =
    \partial_x\hatu_\Delta^n\circ\theX_\Delta^n(\xi)\partial_\xi\theX_\Delta^n(\xi).
  \end{align*}
  Substitute this expression for $\partial_\xi\hatz_\Delta^n(\xi)$ into the left-hand side of \eqref{eq:cov1},
  and perform a change of variables $x=\theX_\Delta^n(\xi)$ to obtain the integral on the right.
  
  Next, take the square root in \eqref{eq:locaffine2} before differentiation, 
  then calculate the square and divide by $\partial_\xi\theX_\Delta^n(\xi)$ afterwards:
  \begin{align*}
    \frac{\partial_\xi \hatz_\Delta^n(\xi)^2}{4\hatz_\Delta^n(\xi)\partial_\xi\theX_\Delta^n(\xi)} =
    \big(\partial_x\sqrt{\hatu_\Delta^n}\big)^2\circ\theX_\Delta^n(\xi)\partial_\xi\theX_\Delta^n(\xi).
  \end{align*}
  Performing the same change of variables as before, this proves that
  \begin{align}
    \label{eq:dummy201}
    \int_0^M \frac{\partial_\xi \hatz_\Delta^n(\xi)^2}{\hatz_\Delta^n(\xi)\partial_\xi\theX_\Delta^n(\xi)} \rho''\circ\theX_\Delta^n(\xi) \dd\xi
    = 4 \intom \left(\partial_x\sqrt{\hatu_\Delta^n}\right)^2(x)\rho''(x)\dd x.
  \end{align}
  It remains to estimate the difference between the $\xi$-integrals in \eqref{eq:cov2} and in \eqref{eq:dummy201}, respectively.
  To this end, observe that for each $\xi\in(\xi_k,\xi_\kph)$ with some $k\in\ivalp$, 
  one has $\partial_\xi\theX_\Delta^n(\xi)=1/z^n_\kph$ and $\hatz_\Delta(\xi)\in[z_\kmh,z_\kph]$.
  Hence, for those $\xi$,
  \begin{align*}
    \left|1-\frac1{\hatz_\Delta^n(\xi)\partial_\xi\theX_\Delta^n(\xi)}\right| \le \left|1-\frac{z^n_\kph}{z^n_\kmh}\right|.
  \end{align*}
  If instead $\xi\in(\xi_\kmh,\xi_k)$, then this estimate holds with the roles of $z^n_\kph$ and $z^n_\kmh$ interchanged.
  Consequently, using once again \eqref{eq:L4bound} and \eqref{eq:oscillation},
  \begin{align*}
    &\tau\sum_{n=1}^N \left|\int_0^M \partial_\xi\hatz_\Delta^n(\xi)^2\rho''\circ\theX^n_\Delta(\xi)\dd\xi
      -\int_0^M \frac{\partial_\xi\hatz_\Delta^n(\xi)^2}{\hatz_\Delta^n(\xi)\partial_\xi\theX_\Delta^n(\xi)}\rho''\circ\theX^n_\Delta(\xi)\dd\xi\right|\\
    &\le B\tau\sum_{n=1}^N\int_0^M \partial_\xi\hatz_\Delta^n(\xi)^2\left|1-\frac1{\hatz_\Delta^n(\xi)\partial_\xi\theX_\Delta^n(\xi)}\right|\dd\xi\\
    &\le B\left(\tau\sum_{n=1}^\infty\|\partial_\xi\hatz_\Delta^n\|_{L^4}^4\right)^{1/2}
    \left(\tau\sum_{n=1}^N\delta\sum_{k\in\ivalp}\left[\left(1-\frac{z^n_\kph}{z^n_\kmh}\right)^2+\left(1-\frac{z^n_\kph}{z^n_\kmh}\right)^2\right]\right)^{1/2} \\
    &\le 3\olH^{1/2}\big(6(b-a)^2T^{1/2}\olH^{1/2}\delta^{1/2}\big)^{1/2},
  \end{align*}
  since $N\tau\le T$ by hypothesis.
  This shows \eqref{eq:cov2}.
\end{proof}
\begin{proof}[Proof of \eqref{eq:weakform2}]
  Again, let $N_\tau\in\setN$ be such that $N_\tau\tau\in(T,T+1)$.
  Combining the discrete weak formulation \eqref{eq:fourAs},
  the change of variables formulae \eqref{eq:cov1}\&\eqref{eq:cov2},
  and the definitions of $R_1$ to $R_5$,
  it follows that
  \begin{align*}
    % \Bigg|\int_0^T \psi(t)\Bigg(\int_\Omega\left[\rho'''(x)\partial_x\ti{\hatu_\Delta}(t;x)+4\rho''(x) \left(\partial_x\ti{\sqrt{\hatu_\Delta}}\right)^2(t;x)\right] \dd x \\
    % -\ti{\spr{\rho'(\xvec_\Delta)}{\wgrad\Fz(\xvec_\Delta)}}(t)\Bigg)\dd t\Bigg|\\
    \err_{2,\Delta}\le BR_5 + B\tau\sum_{n=1}^{N_\tau}\left|
      \int_0^M\left[\partial_\xi\hatz_\Delta^n\rho'''\circ\theX_\Delta^n(\xi) + \partial_\xi\hatz_\Delta^n(\xi)^2\rho''\circ\theX_\Delta^n(\xi)\right]\dd\xi
      - \big(A_1^n-A_2^n+A_3^n+A_4^n\big)\right| \\
    \le B(R_1+R_2+R_3+R_4+R_5) \le B(C_1+C_2+C_3+C_4+C_5)\delta^{1/4}.
  \end{align*}
  This implies the desired inequality \eqref{eq:weakform2}.
\end{proof}
We are now going to finish the proof of this section's main result. 
\begin{proof}[Proof of Proposition \ref{prp:weakgoal}]
  Thanks to \eqref{eq:weakform1}\&\eqref{eq:weakform2}, we know that
  \begin{align*}
    &\Bigg|\int_0^T \psi'(t)\intom \rho(x)\ti{\baru_\Delta}(t;x) \dd x\dd t
    +\psi(0)\intom\rho(x)\baru_\Delta^0(x)\dd x \\
    &+\int_0^T \psi(t)\intom\big[\rho'''(x)\partial_x\ti{\hatu_\Delta}(t;x)+4\rho''(x)\partial_x\ti{\sqrt{\hatu_\Delta}}(t;x)^2\big] \dd x\dd t\Bigg| \\
    &\le \err_{1,\Delta} + \err_{2,\Delta} \le C\big((\tau\Fz(\xvec_\Delta^0))+(\delta\Fz(\xvec_\Delta^0))^{1/2}+\delta^{1/4}\big).
  \end{align*}
  By our assumption \eqref{eq:Fbound} on $\Fz(\xvec_\Delta^0)$, the expression on the right hand side vanishes as $\Delta\to0$.
  To obtain \eqref{eq:weakgoal} in the limit $\Delta\to0$, 
  we still need to show the convergence of the integrals to their respective limits.

  A technical tool is the observation that, for each $p\in[1,4]$,
  \begin{align*}
    Q_p:=\sup_\Delta\tau\sum_{n=1}^{N_\tau}\delta\sum_{\kappa\in\hval}(z^n_\kappa)^p < \infty,
  \end{align*}
  thanks to the estimates \eqref{eq:lohi} and \eqref{eq:L4bound}.
  For the first integral, we use that $\ti{\baru_\Delta}$ converges to $u_*$ w.r.t.\ $\wass$, locally uniformly with respect to $t\in(0,T)$.
  Thus clearly
  \begin{align*}
    \intom \rho(x)\ti{\baru_\Delta}(t;x) \dd x \to \intom \rho(x)u_*(t;x)\dd x
  \end{align*}
  for each $t\in(0,T)$.
  In order to pass to the limit with the time integral, we apply Vitali's theorem.
  To this end, observe that
  \begin{align*}
    \int_0^T \left|\psi'(t) \intom \rho(x)\ti{\baru_\Delta}(t;x) \dd x\right|^2\dd t
    & \le B^2(b-a)\tau\sum_{n=1}^{N_\tau}\intom \baru_\Delta^n(x)^2\dd x \\
    & = B^2(b-a)\tau\sum_{n=1}^{N_\tau}\delta\sum_{\kappa\in\hval}z^n_\kappa
    \le Q_1B^2(b-a).
  \end{align*}
  Next, using the strong convergence from \eqref{eq:strong}, it follows that
  \begin{align*}
    \partial_x\ti{\hatu_\Delta} = 2\ti{\sqrt{\hatu_\Delta}}\partial_x\sqrt{\hatu_\Delta}
    \to 2\sqrt{u_*}\partial_x\sqrt{u_*}=\partial_xu_*
  \end{align*}
  strongly in $L^1(\Omega)$, for almost every $t\in(0,T)$.
  Again, we apply Vitali's theorem to conclude convergence of the time integral,
  on grounds of the following estimate:
  \begin{align*}
    &\int_0^T \left|\psi(t)\intom\rho'''(x)\partial_x\ti{\hatu_\Delta}\dd x\right|^2\dd t
    \le B^2(b-a)\tau\sum_{n=1}^{N_\tau}\intom\big(\partial_x\hatu_\Delta^n(x)\big)^2\dd x \\
    & = B^2(b-a)\tau\sum_{n=1}^{N_\tau}\delta\sum_{k\in\ivalp}\left(\frac{z^n_\kph-z^n_\kmh}\delta\right)^2\left(\frac{z^n_\kph+z^n_\kmh}2\right) \\
    & \le B^2(b-a)\left(\tau\sum_{n=1}^\infty\sum_{k\in\ivalp}\left(\frac{z^n_\kph-z^n_\kmh}\delta\right)^4\right)^{1/2} 
    \left(\tau\sum_{n=1}^{N_\tau}\delta\sum_{\kappa\in\hval}(z^n_\kappa)^2\right)^{1/2}
    \le 3\olH^{1/2} Q_2^{1/2}B^2(b-a),
  \end{align*}
  where we have used \eqref{eq:L4bound}.
  Finally, the strong convergence implies \eqref{eq:strong} also implies that
  \begin{align*}
    \big(\partial_x\ti{\hatu_\Delta}\big)^2 \to \big(\partial_x\sqrt{u_*}\big)^2
  \end{align*}
  strongly in $L^1(\Omega)$, for almost every $t\in(0,T)$.
  One more time, we invoke Vitali's theorem, using that
  \begin{align*}
    &\int_0^T\left|\psi(t)\intom\rho''(x)\partial_x\ti{\sqrt{\hatu_\Delta}}^2(t;x)\dd x\right|^2\dd t
    \le B^2\tau\sum_{n=1}^{N_\tau} \intom \left(\partial_x\sqrt{\hatu_\Delta^n}\right)^4(x)\dd x\\
    &\le \frac12B^2\tau\sum_{n=1}^{N_\tau}\delta\sum_{k\in\ivalp}\left(\frac{z^n_\kph-z^n_\kmh}\delta\right)^2
    \left[\left(1-\frac{z^n_\kph}{z^n_\kmh}\right)^2+\left(1-\frac{z^n_\kmh}{z^n_\kph}\right)^2\right]. \\
    &\le B^2\left(\tau\sum_{n=1}^\infty\delta\sum_{k\in\ivalp}\left(\frac{z^n_\kph-z^n_\kmh}\delta\right)^4\right)^{1/2}
    \left(\tau\sum_{n=1}^\infty\delta\sum_{k\in\ivalp}\left[\left(1-\frac{z^n_\kph}{z^n_\kmh}\right)^4+\left(1-\frac{z^n_\kmh}{z^n_\kph}\right)^4\right]\right)^{1/2}.
  \end{align*}
  The two terms in the last line are uniformly controlled in view of \eqref{eq:L4bound} and \eqref{eq:oscillation}, respectively.
\end{proof}

%%%%%%%%%%%%%%%%%%%%%%%%%%%%%%%%%%%%%%%%%%%%%%%%%%%%%%%%%%%%%%%%%%%%%%%%%%%%%
%%%%%%%%%%%%%%%%%%%%%%%%%%%%%%%%%%%%%%%%%%%%%%%%%%%%%%%%%%%%%%%%%%%%%%%%%%%%%
\section{Proof of Theorems \ref{thm:pre} and \ref{thm:main}}
\label{sec:proofs}
%%%%%%%%%%%%%%%%%%%%%%%%%%%%%%%%%%%%%%%%%%%%%%%%%%%%%%%%%%%%%%%%%%%%%%%%%%%%%
%
Below, we collect the results derived up to here to formally conclude the proofs of our main theorems.
\begin{proof}[Proof of Theorem \ref{thm:pre}]
  Well-posedness of the discrete scheme follows from Proposition \ref{prp:wellposed}.
  Positivity and mass conservation are immediate consequences of the construction:
  recall that $\bar u_\Delta=\cf_\theh[\xvec_\Delta]$, with $\cf_\theh$ defined in \eqref{eq:cf}.
  The monotonicity of $\HFz$ and $\Fz$ have been obtained 
  in Lemma \ref{lem:dFdH_bounded} and \ref{lem:dF_bounded}, respectively.

  It remains to show the exponential decay \eqref{eq:equilibrate} of $\HFz$.
  From (the proof of) Lemma \ref{lem:BVbound}, 
  it follows for each $n=1,2,\ldots$ that
  \begin{align}
    \label{eq:decay1}
    \HF(\baru_\Delta^n)-\HF(\baru_\Delta^{n-1})
    = \HFz(\xvec_\Delta^n)-\HFz(\xvec_\Delta^{n-1})
    \le -\frac\tau{10(b-a)} \tv{\partial_x\sqrt{\hatu_\Delta^n}}^2.
  \end{align}
  By the logarithmic Sobolev inequality on $\Omega$ 
  and thanks to the fact that $\partial_x\sqrt{\hatu_\Delta^n}(0)=0$,
  we further have that
  \begin{align}
    \label{eq:decay2}
    \HF(\hatu_\Delta^n) \le \frac{(b-a)^2}{2\pi^2}\intom\left(\sqrt{\hatu_\Delta^n}\right)^2(x)\dd x
    \le \frac{(b-a)^3}{2\pi^2} \tv{\partial_x\sqrt{\hatu_\Delta^n}}^2.
  \end{align}
  Now combine \eqref{eq:decay1} and \eqref{eq:decay2} with the estimate \eqref{eq:Hinterpol} from the Appendix 
  to conclude that
  \begin{align*}
    \left(1+\frac{\pi^2\tau}{5(b-a)^4}\right)\HF(\baru_\Delta^n) \le \HF(\baru_\Delta^{n-1}).
  \end{align*}
  From here, the claim \eqref{eq:equilibrate} is obtained by induction on $n$.
\end{proof}
\begin{proof}[Proof of Theorem \ref{thm:main}]
  Local uniform convergence of the $\ti{\baru_\Delta}$ to a continuous limit function $u_*$ 
  is part of the conclusions of Proposition \ref{prp:convergence1}, see \eqref{eq:reg1} and \eqref{eq:uniformLinfty}.
  The regularity $\sqrt{u_*}\in L^2(\setRnn;H^1(\Omega))$ has been observed in Proposition \ref{prp:convergence2}.
  The strong convergence stated in the same proposition implies that $\F(u_*)$ is ``almost monotone'':
  indeed, thanks to \eqref{eq:strong} we may assume 
  --- passing to a further subsequence with $\Delta\to0$ if necessary ---
  that
  \begin{align*}
   \ti{\sqrt{\hatu_\Delta}}(t)\to\sqrt{u_*(t)}\quad \text{strongly in $H^1(\Omega)$, for a.e. $t>0$},
  \end{align*}
  and therefore also
  \begin{align}
    \label{eq:auxF1}
    2\intom\ti{\partial_x\sqrt{\hatu_\Delta}}^2(t;x)\dd x \to \F(u_*(t)), \quad \text{for a.e. $t>0$}.
  \end{align}
  On the other hand, arguing just like in the proof of \eqref{eq:cov2}, it follows that
  \begin{align}
    \label{eq:auxF2}
    \int_0^\infty\left|\ti{\Fz(\xvec_\Delta)}-2\intom\ti{\partial_x\sqrt{\hatu_\Delta}}^2(t;x)\dd x\right|\dd t
    \le C\delta^{1/4}.
  \end{align}
  Now combine \eqref{eq:auxF1} and \eqref{eq:auxF2} with the fact 
  that $\ti{\Fz(\xvec_\Delta)}$ is decreasing in $t$, for each $\Delta$,
  and is $\Delta$-uniformly bounded above according to \eqref{eq:Fboundfine}.
  By Helly's selection principle, there exists a monotone $f:\setR_+\to\setR_+$ such that $\F(u_*(t))=f(t)$ for a.e.\ $t>0$.
  The proof of monotonicity for $t\mapsto\HF(u_*)$ is similar, but easier:
  here it suffices to use the local uniform convergence from \eqref{eq:uniformLinfty}.

  Finally, the weak formulation \eqref{eq:weakdlss} has been shown in Proposition \ref{prp:weakgoal}.
  Simply observe that any $\varphi\in C^\infty_c(\setRnn\times\Omega)$ can be approximated 
  by linear combinations of products $\psi(t)\rho(x)$ with functions $\psi\in C^\infty(\setRnn)$ and $\rho\in C^\infty(\Omega)$.
\end{proof}

%%%%%%%%%%%%%%%%%%%%%%%%%%%%%%%%%%%%%%%%%%%%%%%%%%%%%%%%%%%%%%%%%%%%%%%%%%%%%
%%%%%%%%%%%%%%%%%%%%%%%%%%%%%%%%%%%%%%%%%%%%%%%%%%%%%%%%%%%%%%%%%%%%%%%%%%%%%
\section{Numerical results and order of consistency}\label{sec:numeric}
%%%%%%%%%%%%%%%%%%%%%%%%%%%%%%%%%%%%%%%%%%%%%%%%%%%%%%%%%%%%%%%%%%%%%%%%%%%%%
%
The proof of convergence for our discretization given above is purely qualitative.
In this last section, we study quantitative aspects of the convergence.
First, we calculate the order of consistency for approximation of smooth and strictly positive solutions.
Second, we report on the observed order of convergence in several numerical experiments.

\subsection{Order of consistency}
The following proposition shows that our scheme is (formally) of first order in time and of second order in space.
\begin{prp}
  \label{prp:consist}
  Suppose that $\theX\in C^\infty(\setRnn\times[0,M])$ is a classical solution to
  \begin{align}
    \label{eq:cgf}
    \partial_t\theX = \partial_\xi\big(Z^2\partial_\xi^2Z\big),
  \end{align}
  which is further such that $Z=1/\partial_\xi\theX$ is smooth and strictly positive.
  Let $\Delta=(\tau;\theh)$ be a family of discretization parameters.
  Then the corresponding restrictions $(\xvec_\Delta)$ of $\theX$ to the respective meshes,
  given by $x^n_k:=X(n\tau;\xi_k)$ for $n\in\setN$ and $k\in\{1,\ldots,K\}$, 
  satisfy \eqref{eq:dgf} with an error $O(\delta^2)+O(\tau)$ as $\Delta\to0$. 
\end{prp}
\begin{proof}
  Given $\Delta=(\tau;\delta)$, introduce $\tilde Z:\setRnn\times[\delta/2,M-\delta/2]\to\setR_+$ by
  \begin{align*}
    \tilde Z(t;\xi) = \frac\delta{X(t;\xi+\delta/2)-X(t;\xi-\delta/2)},
  \end{align*}
  which is a smooth and strictly positive function, thanks to the properties of $\theX$.
  It is immediately seen that
  \begin{align}
    \label{eq:zstar2}
    \partial_\xi^m\tilde Z(t;\xi) = \partial_\xi^mZ(t;\xi) + O(\delta^2),
  \end{align}
  for each $m\in\setN$ and locally uniformly in $(t;\xi)$ as $\Delta\to0$.
  Observe that, by definition of $(\xvec_\Delta)$ as restriction of $\theX$ to $\Delta$,
  one has
  \begin{align}
    \label{eq:zstar1}
    z^n_\kappa = \frac\delta{x^n_\kappp-x^n_\kappm} = \tilde Z(n\tau;\xi_\kappa).
  \end{align}
  Fix indices $n\in\setN$ and $k\in\{1,\ldots,K-1\}$.
  In the following, we abbreviate
  \begin{align*}
   z_*=\tilde Z(n\tau;\xi_k),\quad z_*'=\partial_\xi Z(n\tau;\xi_k), \quad\ldots, \quad \dot z_*=\partial_t Z(n\tau;\xi_k).
  \end{align*}
  Relation \eqref{eq:zstar1} and a standard Taylor expansion of $\tilde Z$ around $\xi=\xi_k$ yield
  \begin{align*}
    (z^n_\kph)^2\left(\frac{z^n_\kpd-2z^n_\kph+z^n_\kmh}{\delta^2}\right)
    &= \big(z_*^2+\delta z_*z_*'+O(\delta^2)\big)\big(z_*''+\frac\delta2z_*'''+O(\delta^2)\big) \\
    &= z_*^2z_*'' + \frac\delta2(2z_*z_*'z_*''+z_*^2z_*''') + O(\delta^2).
  \end{align*}
  The same expansion --- with $(-\delta)$ in place of $\delta$ --- is obtained 
  for $z^n_\kpd$, $z^n_\kph$ and $z^n_\kmh$ replaced by $z^n_\kph$, $z^n_\kmh$ and $z^n_\kmd$, respectively.
  Therefore, 
  \begin{equation}
    \label{eq:consist1}
    \begin{split}
      \frac1\delta\left[ (z^n_\kph)^2\left(\frac{z^n_\kpd-2z^n_\kph+z^n_\kmh}{\delta^2}\right)
        - (z^n_\kmh)^2\left(\frac{z^n_\kph-2z^n_\kmh+z^n_\kmd}{\delta^2}\right)\right]\\
      = 2z_*z_*'z_*''+z_*^2z_*''' + O(\delta).
    \end{split}
  \end{equation}
  Next, observe that the expression on the left-hand side remains invariant 
  under the simultaneous exchange of $z^n_\kph$ with $z^n_\kmh$ and of $z^n_\kpd$ with $z^n_\kmd$.
  It follows that the odd terms in the Taylor expansion must vanish,
  thus the approximation error on the right-hand side is actually of order $O(\delta^2)$ rather than $O(\delta)$.
  Further, using \eqref{eq:zstar1} and \eqref{eq:zstar2}, 
  the term of order $\delta^0$ can be written as
  \begin{align*}
    2z_*z_*'z_*''+z_*^2z_*''' 
    = \partial_\xi\big(\tilde Z(n\tau;\xi_k)^2\partial_\xi^2\tilde Z(n\tau;\xi_k)\big)
    = \partial_\xi\big(Z(n\tau;\xi_k)^2\partial_\xi^2Z(n\tau;\xi_k)\big) + O(\delta^2).
  \end{align*}
  On the left-hand side of \eqref{eq:dgf}, we obtain
  \begin{align}
    \label{eq:consist2}
    \frac{x^n_k-x^{n-1}_k}\tau 
    = \frac1\tau\left(\theX(n\tau;\xi_k)-\theX((n-1)\tau;\xi_k)\right)
    = \partial_t\theX(n\tau;\xi_k) + O(\tau),
  \end{align}
  thanks to the smoothness of $\theX$ in time.
  Comining \eqref{eq:consist1}\&\eqref{eq:consist2} with the continuous equation \eqref{eq:cgf},
  we arrive at \eqref{eq:dgf}, with an error of $O(\tau)+O(\delta^2)$.
\end{proof}

\subsection{Numerical experiments}
\subsubsection{Non-uniform meshes}
In order to make our discretization more flexible, 
we are going to change our setting and allow \emph{non-equidistant} mass grids.
That is, the mass discretization of $[0,M]$ is determined by a vector $\vech=(\xi_0,\xi_1,\xi_2,\ldots,\xi_{K-1},\xi_K)$,
with
\begin{align*}
  0=\xi_0 < \xi_1 < \cdots < \xi_{K-1} < \xi_K = M,
\end{align*}
and we introduce accordingly the distances
\begin{align*}
  \delta_\kappa = \xi_\kappp-\xi_\kappm, 
  \quad\text{and}\quad 
  \delta_k = \frac12(\delta_\kph+\delta_\kmh) = \frac12(\xi_{k+1}-\xi_{k-1})
\end{align*}
for $\kappa\in\hval$ and $k\in\ivalp$, respectively.
The piecewise constant density function $\baru\in\densNN$ corresponding to a vector $\xvec\in\setR^{K-1}$
is now given by
\begin{align*}
  \baru(x) = z_\kappa \quad\text{for $x_\kappm<x<x_\kappp$}, \quad
  \text{with}\quad z_\kappa = \frac{\delta_\kappa}{x_\kappp-x_\kappm}.
\end{align*}
The Wasserstein-like metric needs to be adapted as well:
the scalar product $\spr{\cdot}{\cdot}$ is replaced by
\begin{align*}
  \langle\vvec,\wvec\rangle_\vech = \sum_{k\in\ivalp} \delta_kv_kw_k.
\end{align*}
Hence the metric gradient $\nabla_\vech f(\xvec)\in\setR^{K-1}$ of a function $f:\xseqNN\to\setR$ at $\xvec\in\xseqNN$
is given by
\begin{align*}
  \big[\nabla_\vech f(\xvec)\big]_k = \frac1{\delta_k}\partial_{x_k}f(\xvec).
\end{align*}
Otherwise, we proceed as before:
the entropy is discretized by restriction, and the discretized Fisher information is the self-dissipation of the discretized entropy.
Explicitly, the resulting fully discrete gradient flow equation
\begin{align*}
  \frac{\xvec_\Delta^n-\xvec_\Delta^{n-1}}\tau = - \nabla_\vech\F_\vech(\xvec_\Delta^n)
\end{align*}
attains the form
\begin{equation}
  \label{eq:nonuniform}
  \begin{split}
    \frac{x^n_k-x^{n-1}_k}\tau = 
    \frac1{\delta_k}\Bigg[
    \frac{(z^n_\kph)^2}{\delta_\kph}\left(\frac{z^n_\kpd-z^n_\kph}{\delta_{k+1}}-\frac{z^n_\kph-z^n_\kmh}{\delta_k}\right)\\
    - \frac{(z^n_\kmh)^2}{\delta_\kmh}\left(\frac{z^n_\kph-z^n_\kmh}{\delta_k}-\frac{z^n_\kmh-z^n_\kmd}{\delta_{k-1}}\right)
    \Bigg].    
  \end{split}
\end{equation}

\subsubsection{Initial condition}
Our main experiments are carried out using the by now classical test case from \cite{BLS}, 
that is
\begin{align}
  \label{eq:u0}
  u^0(x) = \epsilon + \cos^{16}(\pi x), \quad\text{on}\ \Omega=[0,1],
\end{align}
with $\epsilon=10^{-3}$. 
The mass grid $\vech$ is chosen in such a way that $u_\Delta^0$
is a piecewise constant approxiation of $u^0$ with respect to a \emph{spatially uniform} grid. 
That is, we choose $\vech$ such that the initial condition $\xvec_\Delta^0$ for $\xvec$ attains the simple form
\begin{align}
  \label{eq:num0}
  x^0_k=a+\frac{b-a}Kk.  
\end{align}
To construct $\vech$, we first calculate the cummulative distribution function $U^0:[0,1]\to[0,M]$ 
by numerical integration of $u^0$,
\begin{align*}
  U^0(x) = \int_a^x u^0(y)\dd y,
\end{align*}
and then define $\xi_k:=U^0(x^0_k)$, for $k=0,1,\ldots,K$.
\begin{rmk}
  An equidistant mass grid leads to a good spatial resolution of regions where the value of $u^0$ is large,
  but provides a very poor resolution in regions where $u^0$ is small.
  Since the evolution of the zones with low density are of particular interest in numerical studies of the DLSS equation,
  it is natural to use a non-uniform mass grid with an adapted spatial resolution, like the one defined above.
\end{rmk}

\subsubsection{Implementation}
From the initial condition $\xvec_\Delta^0$, the fully discrete solution is calculated inductively 
by solving the implicit Euler scheme \eqref{eq:nonuniform} for $\xvec_\Delta^n$, given $\xvec_\Delta^{n-1}$.
In each time step, a damped Newton iteration is performed, with the solution from the previous time step as initial guess.
Slow convergence of the Newton iteration has been observed in situations 
where the density $\baru_\Delta^{n-1}$ has steep gradients and/or intervals of very low values.

Our reference solution is calculated with the scheme described in \cite{DMMnum},
which is fully variational as well, but uses different ansatz functions for the Lagrangian maps.
% We remark that the scheme in \cite{DMMnum} was designed for solution of \eqref{eq:dlss} with periodic boundary conditions instead of \eqref{eq:bc}.
% However, it is easily seen that any spatially periodic solution on, say, $\Omega=[-1,1]$ 
% automatically satisfies also the boundary conditions \eqref{eq:bc}.
% So, 
Even without a rigorous result on uniqueness of weak solutions,
it seems reasonable to expect that both schemes should approximate the same solution.
A technical issue with the comparison of our solution to the reference solution is 
that both use a different way for the reconstruction of the density from the Lagrangian map.
This difference camouflages the true approximation error in the plain $L^2$-differences.
For a fair comparison, we calculate the $L^2$-difference of the linear interpolations 
of the values for the density with respect to the nodes of the Lagrangian maps.

\subsubsection{Observed rate of convergence}
\label{sct:experiments}
%
% The numerical analysis of the convergence rate has been performed for the initial function $u^0$ from \eqref{eq:u0}.
%
\begin{figure}
  \centering
  \subfigure{\includegraphics[width=0.45\textwidth]{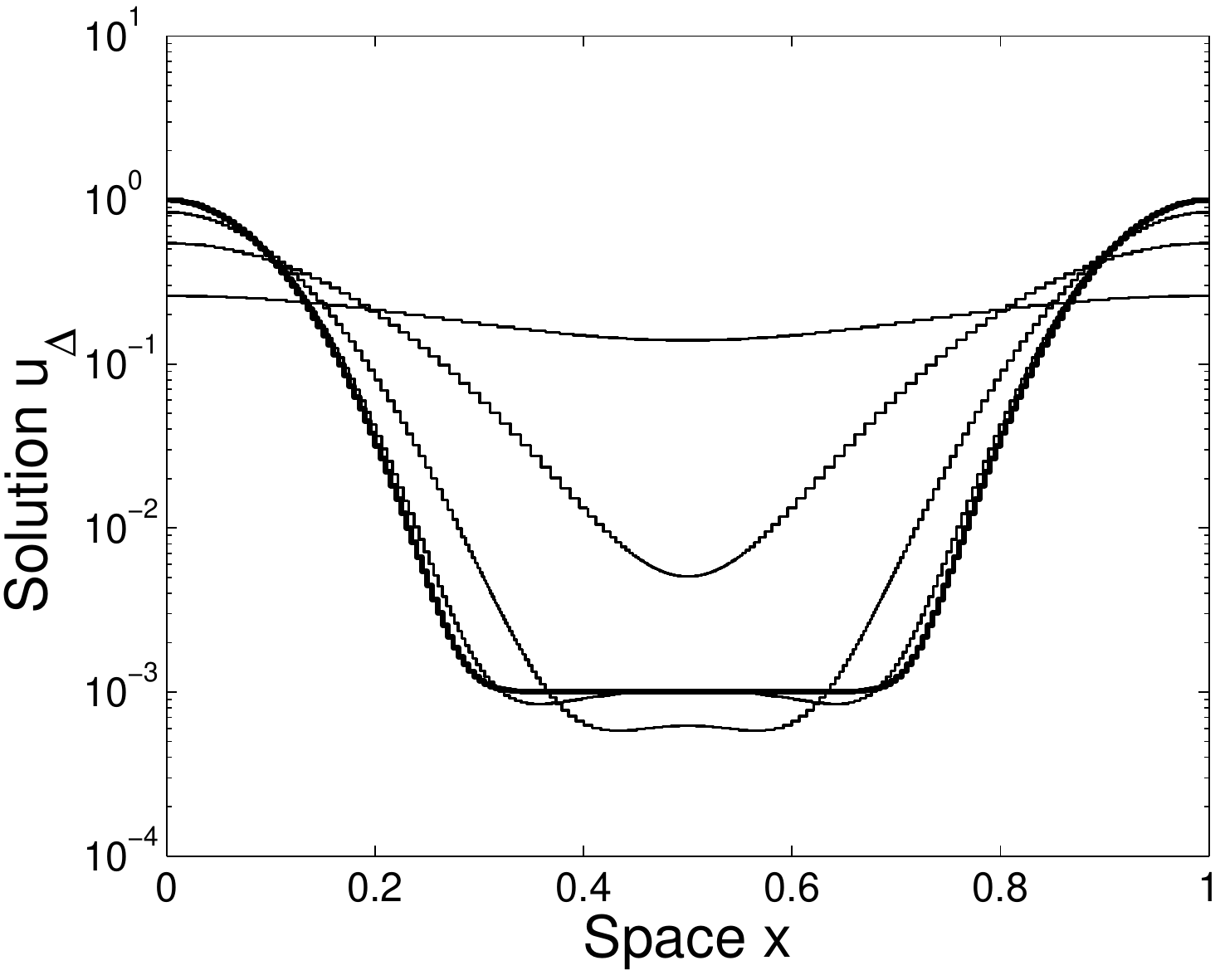}}
  % \hfill
  \subfigure{\includegraphics[width=0.46\textwidth]{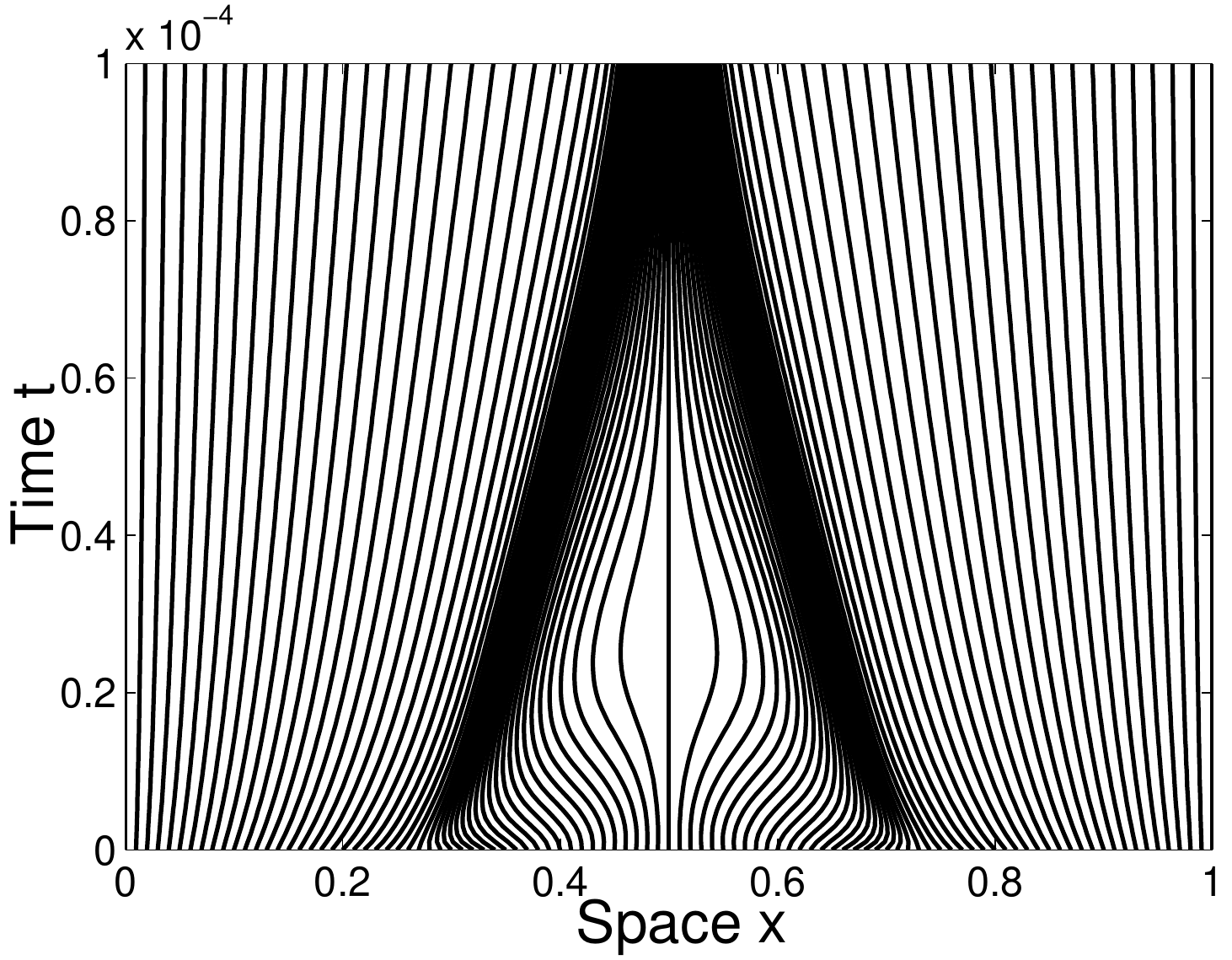}}
  \caption{\emph{Left:} snapshots of the densities $\baru_\Delta$ for the initial condition \eqref{eq:u0}
	at times $t=0$ and $t=10^{i}$, $i=-6,\ldots,-3$, 
	using $K=200$ grid points and the time step size $\tau=10^{-6}$. 
        \emph{Right:} associated particle trajectories.}
  \label{fig:fig1}
\end{figure}

\begin{figure}
  \centering
	\subfigure{\includegraphics[width=0.45\textwidth]{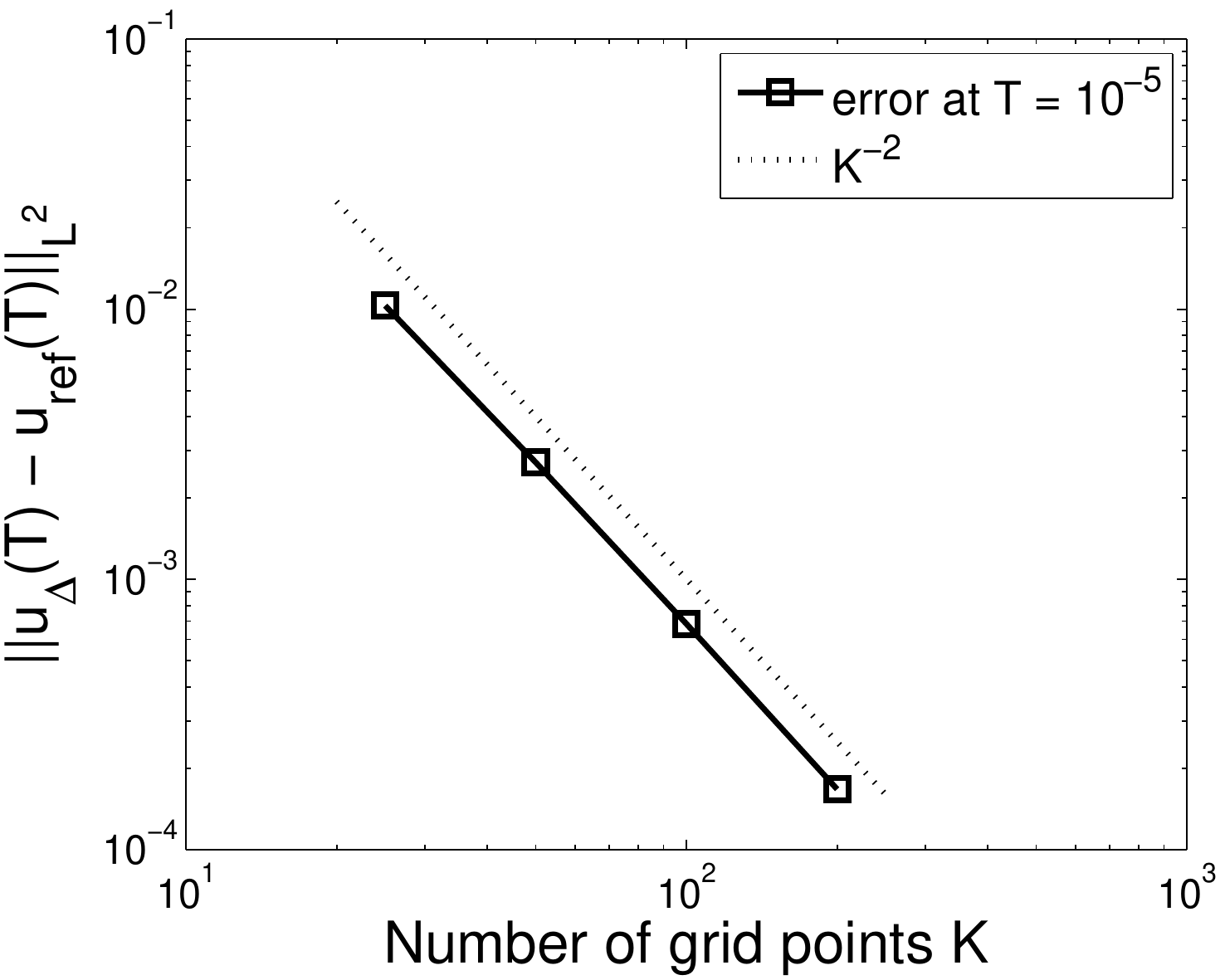}}
  % \hfill
  \subfigure{\includegraphics[width=0.46\textwidth]{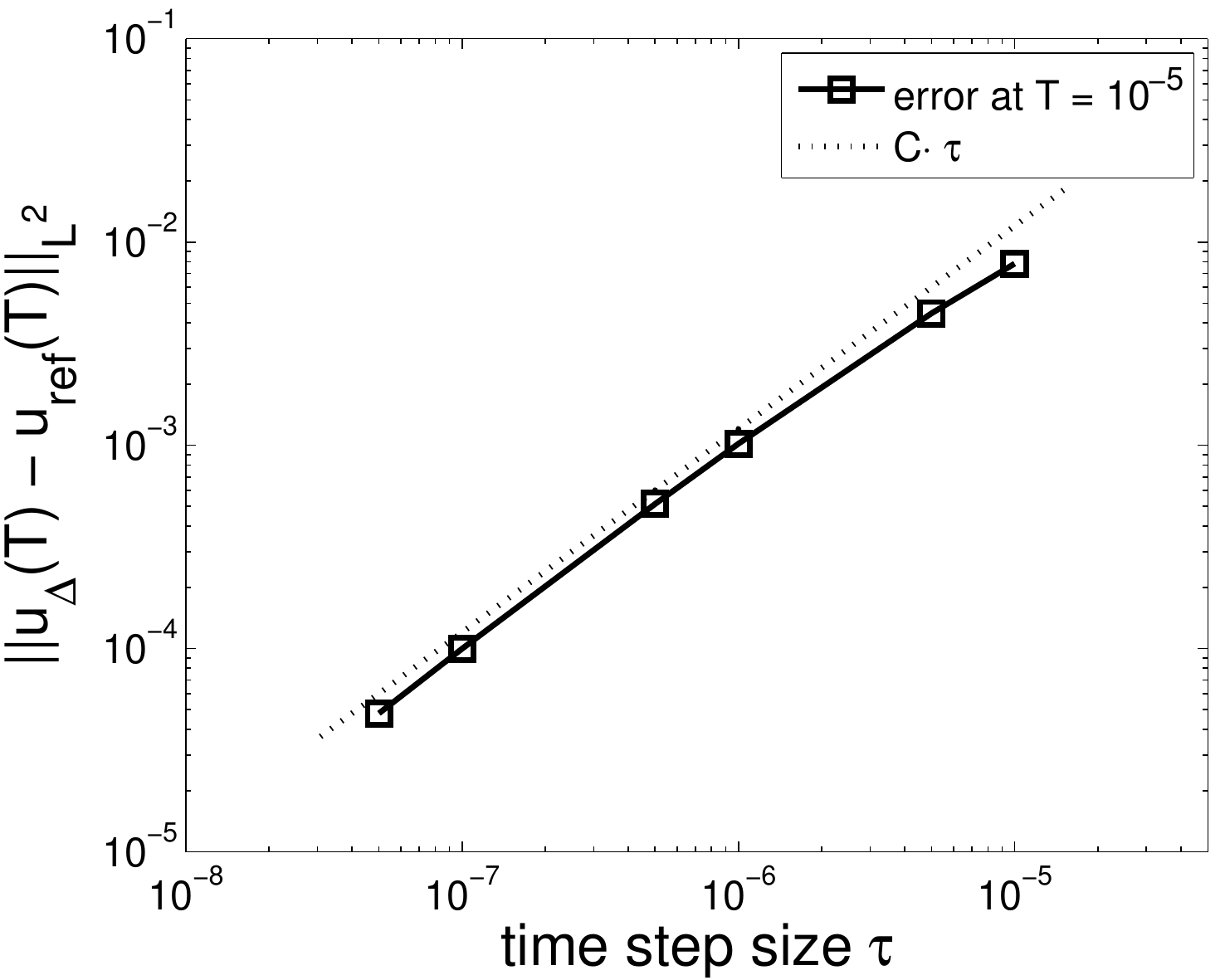}}
  \caption{Numerical error analysis for $u^0$ from \eqref{eq:u0}.
					\emph{Left:} fixed time step size $\tau=10^{-8}$ and $K=25,50,100,200$ spatial grid points.
    The $L^2$-errors are evaluated at $T=5\cdot10^{-6}$.
					\emph{Right:} fixed $K=800$ using $\tau=10^{-5},5\cdot10^{-6},10^{-6},5\cdot10^{-7},10^{-7},5\cdot10^{-8}$. 
		The error is evaluated at $T=10^{-5}$.}
  \label{fig:fig2}
\end{figure}
Figure \ref{fig:fig1} provides a qualitative picture of the evolution with initial condition $u^0$: 
the plot on the left shows the density function $\baru_\Delta$ at several instances in time,
the plot on the right visualizes the motion of the mesh points $\ti{x_k}$ associated to the Lagrangian maps $\theX_\Delta$ in continuous time.
It is clearly seen that the initial density has a very flat minimum (which is degenerate of order 16) at $x=1/2$,
which bifurcates into two sharper minima at later times, and eventually becomes one single minimum again.
This behavior underlines that the comparison principle does obviously not hold for the DLSS equation.
Both figures has been generated using $K=200$ spatial grid points and the time step size $\tau=10^{-6}$.

%For numerical analysis of the convergence rate, we have carried out a series of experiments,
%in which we fix the time step size $\tau=10^{-8}$ and vary the number of spatial grid points, using $K=25,50,100,200$. 
%Figure \ref{fig:fig2}/Left shows the corresponding $L^2$-error $\text{err}$ 
%between the solution to our scheme and the reference solution, evaluated at time $t=2.5\cdot10^{-6}$.
%It is clearly seen that the error decays with an almost perfect rate of $\delta^2\propto K^{-2}$.

For numerical analysis of the convergence rate, we have carried out two series of experiments.
In the first series, we fix the time step size $\tau=10^{-8}$ and vary the number of spatial grid points, using $K=25,50,100,200$. 
Figure \ref{fig:fig2}/Left shows the corresponding $L^2$-error 
between the solution to our scheme and the reference solution, evaluated at time $T=10^{-5}$.
It is clearly seen that the error decays with an almost perfect rate of $\delta^2\propto K^{-2}$.
For the second series of experiments, we keep the spatial discretization parameter $K=800$ fixed 
and run our scheme with the time step sizes $\tau=10^{-5},5\cdot10^{-6},10^{-6},5\cdot10^{-7},10^{-7},5\cdot10^{-8}$, respectively.
The corresponding $L^2$-error at $T=10^{-5}$ is plotted in Figure \ref{fig:fig2}/Right. 
It is proportional to $\tau$. 

\subsubsection{Discontinuous initial data}
\begin{figure}
  \centering
  \subfigure{\includegraphics[width=0.45\textwidth]{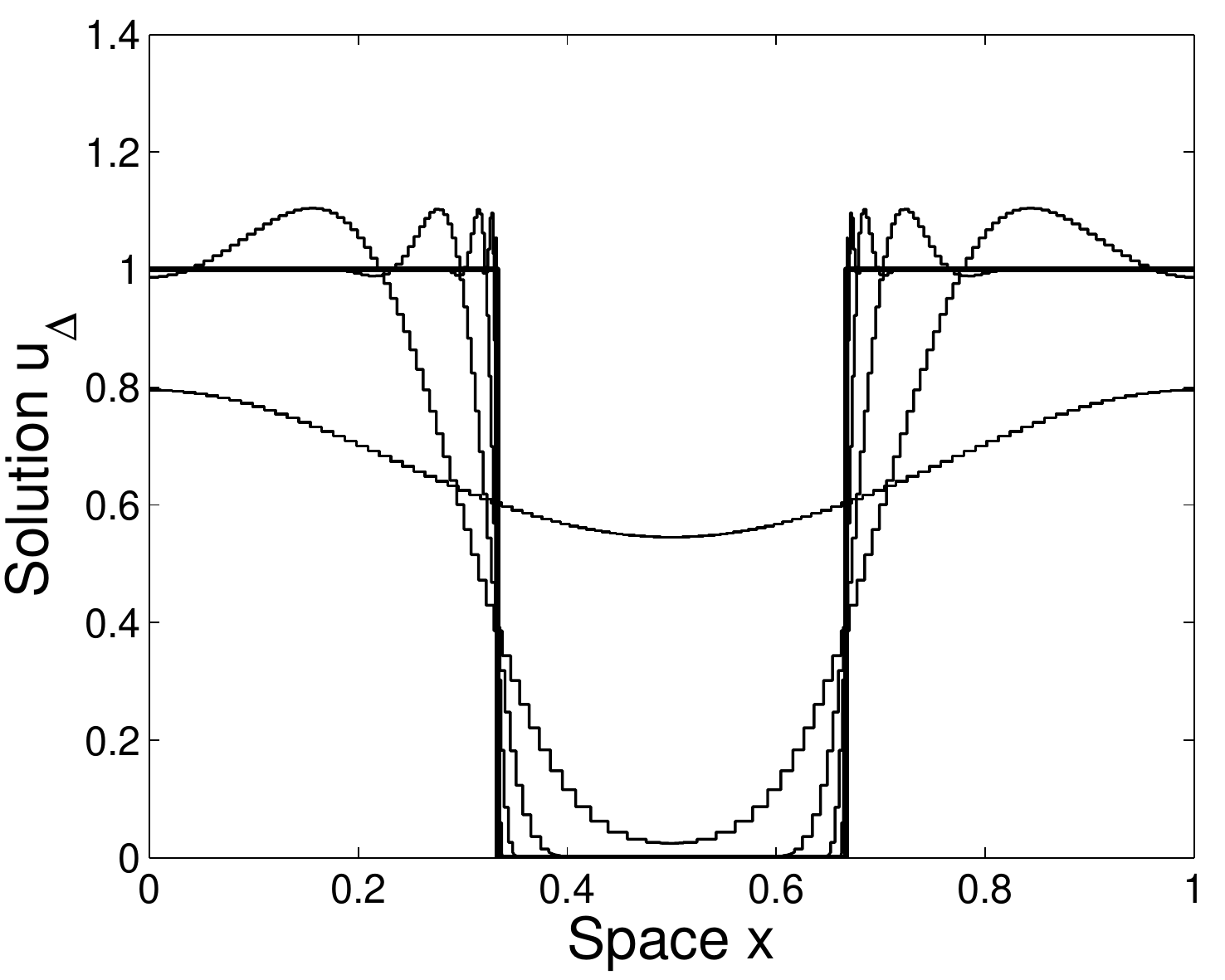}}
  % \hfill
  \subfigure{\includegraphics[width=0.46\textwidth]{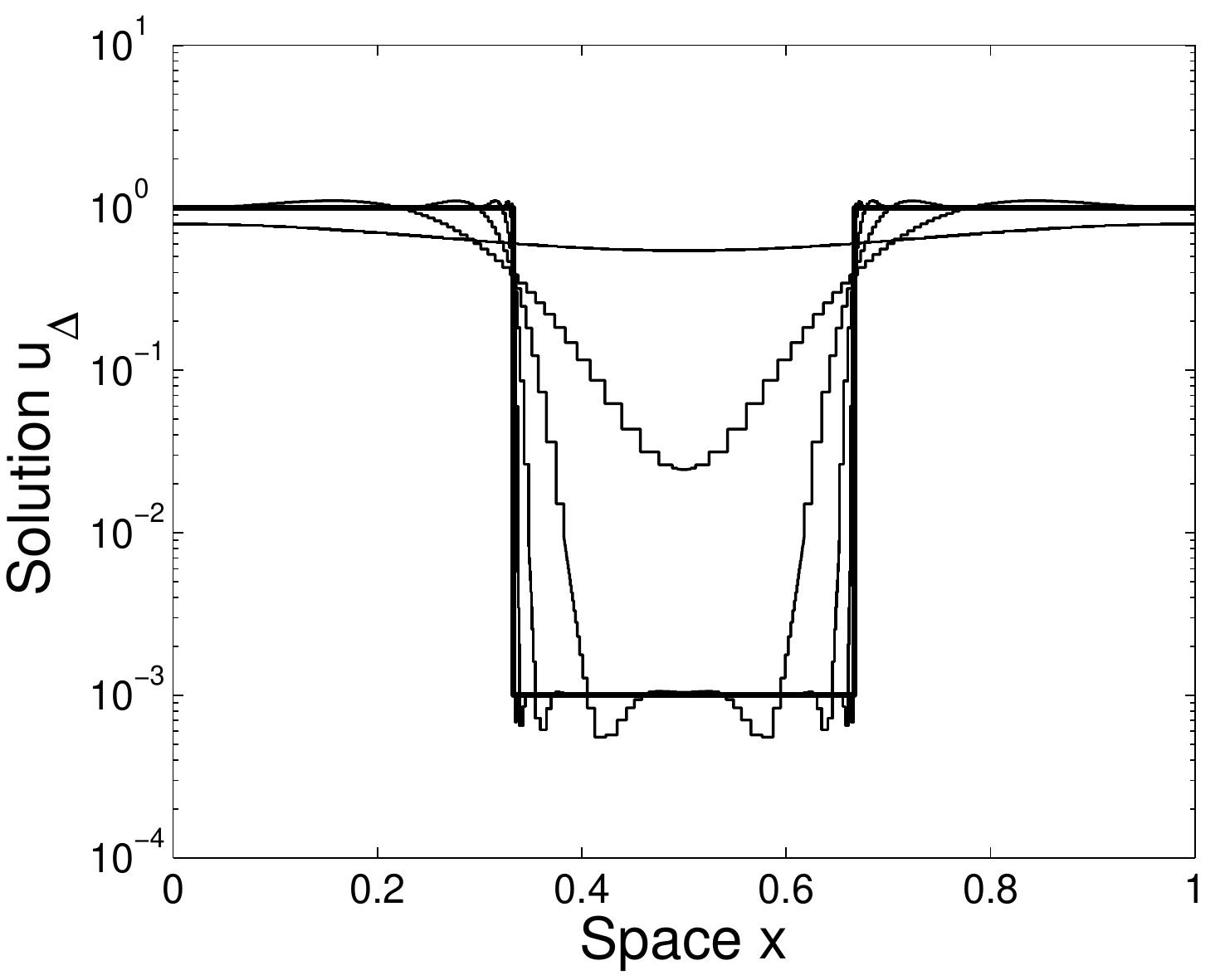}}
  \caption{Snapshots of the densities $\baru_\Delta$ for the initial condition \eqref{eq:u0discont}
	at times $t=0$ and $t=10^{i}$, $i=-13,-11,\ldots,-5,-3$, 
	using $K=200$ grid points with linear (left) and logarithmic (right) scaling.}
  \label{fig:fig3}
\end{figure}
\begin{figure}
  \centering
  \subfigure{\includegraphics[width=0.45\textwidth]{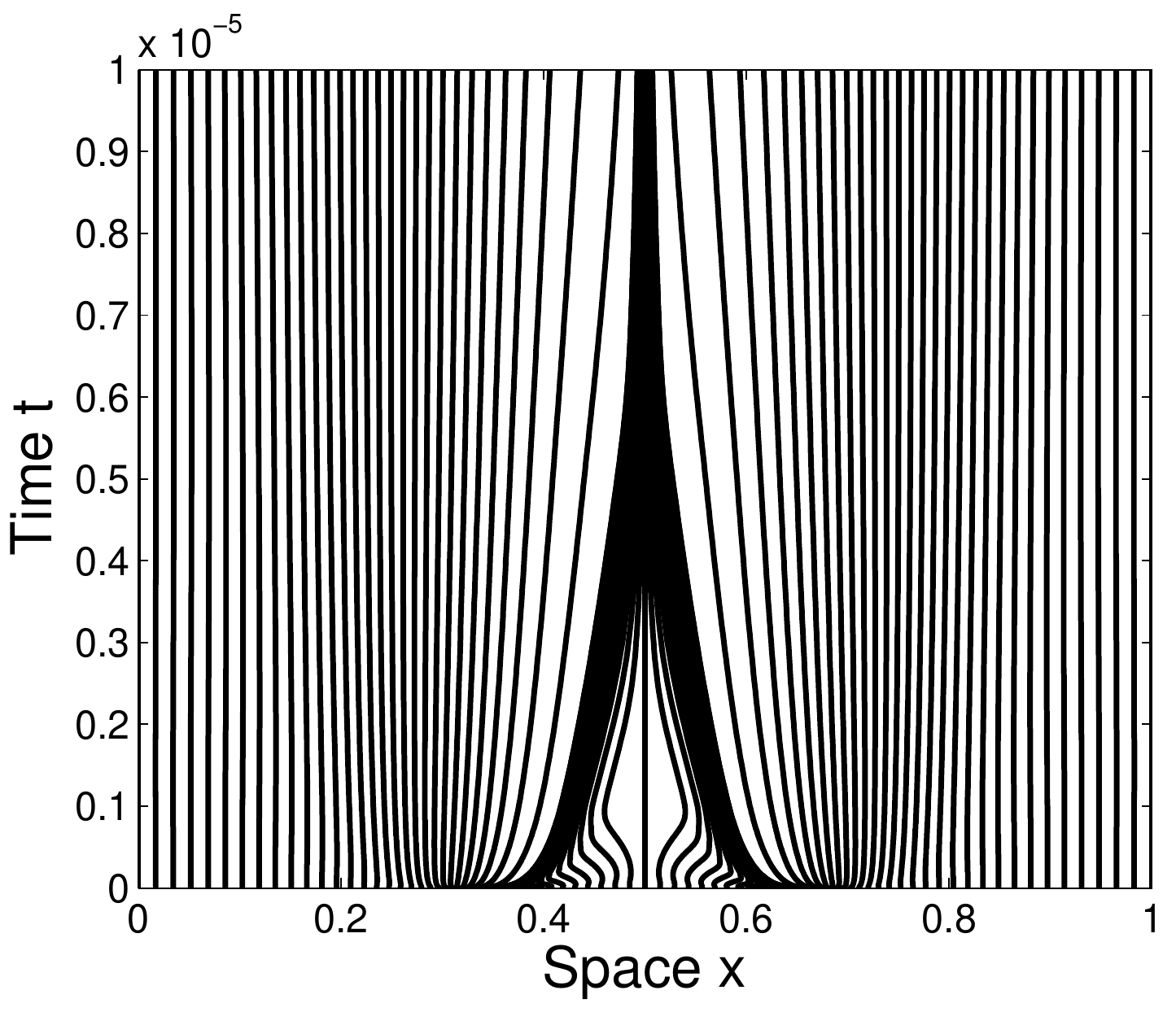}}
  % \hfill
  \subfigure{\includegraphics[width=0.46\textwidth]{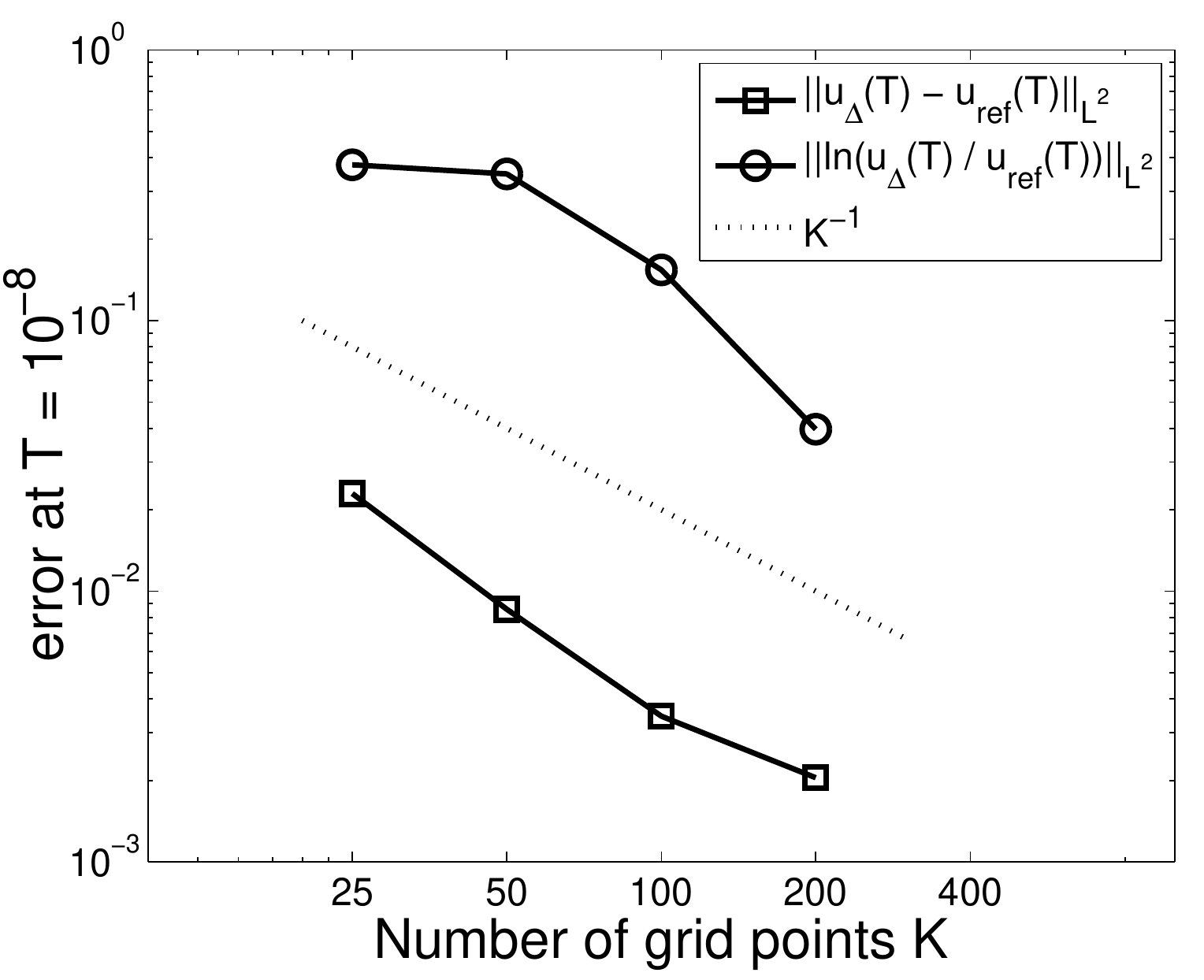}}
  \caption{\emph{Left:} associated particle trajectories of $\baru_\Delta$ using the initial condition \eqref{eq:u0discont}. 
        \emph{Right:} Numerical error analysis for $u_{\operatorname{discont}}^0$ from \eqref{eq:u0discont}
				with fixed $\tau$ and $K=25,50,100,200$ spatial grid points.
    The $L^2$-errors are evaluated at $T=10^{-8}$.}
  \label{fig:fig4}
\end{figure}
One of the conclusions of Theorem \ref{thm:main} is that
the discrete approximations $u_\Delta$ converge also for (a large class of) non-regular initial data $u^0$.
For illustration of this feature, we consider the discontinuous initial density function
\begin{align}\label{eq:u0discont}
	u_\text{discont}^0 
        = \begin{cases} 1 & x\in[0,\frac{1}{3}]\cup[\frac{2}{3},1],\\ 10^{-3},& x\in (\frac{1}{3},\frac{2}{3}) \end{cases}
\end{align}
instead of $u^0$ from \eqref{eq:u0}. 
According to our hypothesis \eqref{eq:genhypo}, we need to use a sufficiently high spatial and temporal resolution.
In practice, this is done in an adaptive way:
the $K$ points of the initial grid $\xvec_\Delta^0$ are not placed equidistantly, 
but with a higher refinement around the points of discontinuity;
the applied time step $\tau$ is extremely small (down to $10^{-13}$) during the initial phase of the evolution,
and is larger (up to $10^{-9}$) at later times.

Figure \ref{fig:fig3} provides a qualitative picture of the fully discrete evolution for $K=200$ grid points:
snapshots of the discrete density function $\bar u_\Delta$ are shown on the left,
corresponding snapshots of the logarithmic density are shown on the right.
Note that within a very short time, 
peaks of relatively high amplitudes are generated near the points where $u^0$ is discontinuous.
The associated Lagrangian maps are visualized in Figure \ref{fig:fig4}/Left.
Notice the fast motion of the grid points near the discontinuities.

To estimate the rate of convergence, we performed a series of experiments using $K=25,50,100$ and $200$ spatial grid points.
For comparison, we calculated a highly refined solution of the following semi-implicit reference scheme,
\begin{align*}
	\frac{u_{\text{ref}}^{n+1}-u_{\text{ref}}^n}{\tau}
	= -\Delta_2\big( u_{\text{ref}}^{n} \Delta_2\ln(u_{\text{ref}}^{n+1})\big),
\end{align*}
where $\Delta_2$ is the standard central difference operator $\Delta_2$.
The reference scheme is run with $K=800$ spatial grid point.
An adaptive choice of the time step $\tau$ needs to be made 
in order to avoid that the reference solution $u_\text{ref}$ breaks down due to loss of positivity.
The $L^2$-differences of the densities and of their logarithms have been evaluated at $T=10^{-8}$, see Figure \ref{fig:fig4}/Right.
As expected, the rate of convergence is no longer quadratic in $\delta\propto K^{-1}$;
instead, the error decays approximately linearly.

%%%%%%%%%%%%%%%%%%%%%%%%%%%%%%%%%%%%%%%%%%%%%%%%%%%%%%%%%%%%%%%%%%%%%%%%%%%%%
% APPENDIX
%%%%%%%%%%%%%%%%%%%%%%%%%%%%%%%%%%%%%%%%%%%%%%%%%%%%%%%%%%%%%%%%%%%%%%%%%%%%%

\begin{appendix}
  \section{Some technical lemmas}
  \begin{lem}
    For each $p>1$ and $\xvec\in\xseqN$ with $\zvec=\cz[\xvec]$, one has that
    \begin{align}
      \label{eq:xpower}
      \sum_{\kappa\in\hval}\left(\frac\delta{z_\kappa}\right)^p =\sum_{\kappa\in\hval}(x_\kappp-x_\kappm)^p \le (b-a)^p.
    \end{align}
  \end{lem}
  \begin{proof}
    The first equality is simply the definition \eqref{eq:zk} of $z_\kappa$.
    Since trivially $x_\kappp-x_\kappm<b-a$ for each $\kappa\in\hval$,
    and since $p-1>0$,
    it follows that
    \begin{align*}
      \sum_{\kappa\in\hval}(x_\kappp-x_\kappm)^p\le (b-a)^{p-1}\sum_{\kappa\in\hval}(x_\kappp-x_\kappm) = (b-a)^p.
      &\qedhere
    \end{align*}
  \end{proof}
  \begin{lem}
    For each $\xvec\in\xseqN$ with $\zvec=\cz[\xvec]$, one has that
    \begin{align}
      \label{eq:lohi}
      \frac{\delta}{b-a}\le z_\kappa \le M^{1-1/q}\left(\delta\sum_{k\in\ivalp}\left|\frac{z_\kph-z_\kmh}\delta\right|^q\right)^{1/q}
      + \frac{M}{b-a} \quad \text{for all $\kappa\in\hval$},
    \end{align}
    and consequently,
    \begin{align}
      \label{eq:Festimate}
      z_\kappa \le \big(2M\Fz[\xvec]\big)^{1/2} 
      + \frac{M}{b-a} \quad \text{for all $\kappa\in\hval$}.
    \end{align}
  \end{lem}
  \begin{proof}
    The first estimate in \eqref{eq:lohi} is an immediate consequence of the definition of $z_\kappa$ in \eqref{eq:zk}.
    To prove the second estimate, let $\kappa^*\in\hval$ be such that $z_{\kappa^*}=\max z_k$.
    Observe that there exists a $\kappa_*\in\hval$ such that
    \begin{align}
      \label{eq:dummy001}
      z_{\kappa_*} \le \frac{M}{b-a} \le z_{\kappa^*}.
    \end{align}
    Writing out $z_{\kappa^*}-z_{\kappa_*}$ as a sum over differences of adjacent values of $z_k$
    and applying the triangle and Cauchy Schwarz inequality,
    one obtains
    \begin{align*}
      z_{\kappa^*}-z_{\kappa_*}
      \le \sum_{k\in\ivalp}|z_\kph-z_\kmh| 
      \le \left(\delta\sum_{k\in\ivalp}1\right)^{1-1/q}\left(\delta\sum_{k\in\ivalp}\left|\frac{z_\kph-z_\kmh}\delta\right|^q\right)^{1/q}.
    \end{align*}
    Now combine this with \eqref{eq:dummy001}.
  \end{proof}
  \begin{lem}
    With $\hatu$ and $\baru$ being, respectively, the piecewise linear and the piecewise constant densities
    associated to a given vector $\xvec$,
    then
    \begin{align}
      \label{eq:Hinterpol}
      \HFz(\xvec) = \HF(\baru)\le\HF(\hatu).
    \end{align}
  \end{lem}
  \begin{proof}
    First observe that
    \begin{align}
      \label{eq:lnmean}
      \int_0^1 \ln\big(p(1-\lambda)+q\lambda\big)\dd\lambda = \frac{p\ln p-q\ln q}{p-q} -1 \ge \frac12(\ln p+\ln q),
    \end{align}
    which is an easy consequence of a Taylor expansion for the function $s\mapsto(1+s)\ln s$ around $s=1$,
    substituting $s=p/q$.
    On the one hand, we have that
    \begin{align*}
      \intom \baru(x)\log\baru(x)\dd x
      = \delta\sum_{\kappa\in\hval}\log z_\kappa
      = \delta\frac{\log z_{1/2}+\log z_{K-1/2}}2 + \delta\sum_{k\in\ivalp}\frac{\log z_\kph+\log z_\kmh}2,
    \end{align*}
    and on the other hand,
    \begin{align*}
      \intom\hatu(x)\log\hatu(x)\dd x 
      &= \int_0^M \log\hatz(\xi)\dd\xi \\
      &= \frac\delta2\big(\log z_0+\log z_K) + \delta\sum_{k\in\ivalp}\int_0^1 \ln\big(z_\kmh(1-\lambda)+z_\kph\lambda\big)\dd\lambda\\
      &\ge \delta\frac{\log z_{1/2}+\log z_{K-1/2}}2 + \delta\sum_{k\in\ivalp}\frac{\log z_\kph+\log z_\kmh}2,
    \end{align*}
    where we have used \eqref{eq:lnmean}.
    This clearly implies \eqref{eq:Hinterpol}.
  \end{proof}
  \begin{lem}[Gargliardo-Nirenberg inequality]
    \label{lem:GN}
    For each $f\in H^1(\Omega)$, one has that
    \begin{align}
      \label{eq:GN}
	\|f\|_{C^{1/6}(\Omega)} \leq (9/2)^{1/3} \|f\|_{H^1(\Omega)}^{2/3} \|f\|_{L^2(\Omega)}^{1/3}.
      \end{align}
    \end{lem}
    \begin{proof}
      Assume first that $f\ge0$.
      Then, for arbitrary $a<x<y<b$, the fundamental theorem of calculus and H\"older's inequality imply that
      \begin{align*}
        \big|f(x)^{3/2}-f(y)^{3/2}\big|
        \le \frac32\int_x^y 1\cdot f(z)^{1/2}|f'(z)|\dd z
        \le \frac32|x-y|^{1/4}\|f\|_{L^2(\Omega)}^{1/2}\|f'\|_{L^2(\Omega)}.
      \end{align*}
      Since $f\ge0$, we can further estimate
      \begin{align*}
        |f(x)-f(y)| \le \big|f(x)^{3/2}-f(y)^{3/2}\big|^{2/3}
        \le (3/2)^{2/3}|x-y|^{1/6}\|f\|_{L^2(\Omega)}^{1/3}\|f\|_{H^1(\Omega)}^{1/3}.
      \end{align*}
      This shows \eqref{eq:GN} for non-negative functions $f$.
      A general $f$ can be written in the form $f=f_+-f_-$, where $f_\pm\ge0$.
      By the triangle inequality, and since $\|f_\pm\|_{H^1(\Omega)}\le\|f\|_{H^1(\Omega)}$,
      \begin{align*}
        \|f\|_{C^{1/6}(\Omega)} \le \|f_+\|_{C^{1/6(\Omega)}}+\|f_-\|_{C^{1/6(\Omega)}}
        \le 2(3/2)^{2/3}\|f\|_{L^2(\Omega)}^{1/3}\|f\|_{H^1(\Omega)}^{1/3}.
      \end{align*}
      This proves the claim.
    \end{proof}
\end{appendix}

\bibliography{Horst}
\bibliographystyle{siam}%\bibliographystyle{plain}

\end{document}